\documentclass[11pt]{amsart}

\usepackage{amsmath, amsthm, amssymb}
\usepackage{url}
\usepackage{xcolor}
\usepackage{color}
\usepackage{tensor}
\usepackage{verbatim,xcolor}
\usepackage{pst-node}
\usepackage{mathrsfs}
\usepackage{tikz-cd} 
\usetikzlibrary{matrix,calc}
\definecolor{wine-stain}{rgb}{0.5,0,0}
\usepackage[colorlinks=true,citecolor=blue]{hyperref}

\renewcommand{\d}{\partial}
\newcommand{\cV}{\Nu}

\newcommand{\ddbar}{\sqrt{-1}\d\overline{\d}}
\newcommand{\nddbar}{\frac{\sqrt{-1}}{2\pi}\d\overline{\d}}

\newcommand{\ac}{\mathrm{arccot}}

\newcommand{\sT}{\mathcal{T}}

\newtheorem{thm}{Theorem}[section]
\newtheorem{prop}[thm]{Proposition}
\newtheorem{lem}[thm]{Lemma} 
\newtheorem{cor}[thm]{Corollary} 
\newtheorem{conj}[thm]{Conjecture}

\newtheorem{defn}[thm]{Definition}
\newtheorem{rem}[thm]{Remark}

\renewcommand{\[}{\begin{equation}}
\renewcommand{\]}{\end{equation}}

\numberwithin{equation}{section}

\newcommand{\al}{\alpha}
\newcommand{\be}{\beta}
\newcommand{\ga}{\gamma}
\newcommand{\la}{\lambda}
\newcommand{\ka}{\kappa}
\newcommand{\ep}{\epsilon}

\newcommand{\Ga}{\Gamma}
\newcommand{\La}{\Lambda}

\newcommand{\vp}{\varphi}
\newcommand{\vep}{\varepsilon}

\newcommand{\Bl}{\mathrm{Bl}}

\newcommand{\tpsi}{\tilde\psi}
\newcommand{\Fut}{\mathrm{Fut}}
\newcommand{\scrX}{\mathscr{X}}
\newcommand{\tc}{\tilde c}

\newcommand{\NN}{\mathbb{N}}

\newcommand{\QQ}{\mathbb{Q}}
\newcommand{\RR}{\mathbb{R}}
\newcommand{\CC}{\mathbb{C}}

\newcommand{\PP}{\mathbb{P}}

\newcommand{\Nu}{\mathcal{V}}

\newcommand{\sL}{\mathcal{L}}

\newcommand{\sN}{\mathcal{N}}
\newcommand{\sK}{\mathcal{K}}
\newcommand{\Tr}{\mathrm{Tr}}

\newcommand{\sI}{\mathcal{I}}
\newcommand{\sE}{\mathcal{E}}
\newcommand{\sO}{\mathcal{O}}
\newcommand{\cS}{\mathcal{S}}

\newcommand{\vol}{\mathrm{vol}}

\newcommand{\uvarphi}{\underline\varphi}

\makeatletter
\renewcommand*{\eqref}[1]{%
  \hyperref[{#1}]{\textup{\tagform@{\ref*{#1}}}}%
}
\makeatother

\makeatletter \def\l@subsection{\@tocline{2}{0pt}{1pc}{5pc}{}} \def\l@subsection{\@tocline{2}{0pt}{2pc}{6pc}{}} \makeatother

\title{Minimal slopes and bubbling for complex Hessian equations}
\author[V. Datar]{Ved Datar}
\author[R. Mete]{Ramesh Mete}
\address{Department of Mathematics, Indian Institute of Science, Bangalore - 560012}
\email{vvdatar@iisc.ac.in, rameshmete@iisc.ac.in}
\author[J. Song]{Jian Song}
\address{Department of Mathematics, Rutgers University, Piscataway, NJ 08854}
\email{jiansong@math.rutgers.ed}
\thanks{Work supported in part by the Infosys Young Investigator award, SERB MATRICS grant MTR/2022/000260, PhD scholarship from the Indian Institute of Science, and the National Science Foundation grant DMS-2203607}

\begin{document}

\maketitle
\begin{abstract}
The existence of smooth solutions to a broad class of complex Hessian equations is related to nonlinear Nakai type criterias on intersection numbers on K\"ahler manifolds. Such a Nakai criteria can be interpreted as a slope stability condition analogous to the slope stability for Hermitian vector bundles over K\"ahler manifolds. In the present work, we initiate a program to  find canonical solutions to such equations in the unstable case when the Nakai criteria fails. Conjecturally such solutions should arise as limits of natural parabolic flows and should be minimizers of the corresponding moment-map energy functionals. We implement our approach for the $J$-equation and the deformed Hermitian Yang-Mills equation on surfaces and some examples with symmetry. We prove that there always exist unique canonical solutions to these two equations on K\"ahler surfaces in the unstable cases.  Such canonical solutions with singularities are also shown to be the limits of the corresponding J-flow and the cotangent flow on certain projective bundles. We further present the bubbling phenomena for the $J$-equation by constructing minimizing sequences of the moment-map energy functionals, whose Gromov-Hausdorff limits are singular algebraic spaces.

\end{abstract}
\tableofcontents

\section{Introduction} 

\subsection{Overview}

Throughout the paper $(X^n,\omega)$ will denote a compact K\"ahler manifold (without boundary) of dimension $n$. We will denote the K\"ahler class of $\omega$ by $\be\in H^{1,1}(X,\RR)$, and $\al\in H^{1,1}(X,\RR)$ will denote be another class. We let $A = \omega^{-1}\chi$ for some smooth real $(1,1)$ form $\chi$, and consider equations for the form  $$F(A) = f(\la_1,\cdots,\la_n) =  h,$$ where $\la_1\geq \la_2\cdots\geq \la_n$ are eigenvalues of $A$. Under some additional conditions on $h$ and $f$ (for instance if $h$ is a constant and $f$ is the sum of reciprocals of the entries), existence of smooth solutions is equivalent to the positivity of certain intersection numbers - this is the so-called stable case (cf. \cite{Yau, DemPa, W1, W2, SW, LS, CoSz, Sz-pde, gchen,DP,So, CLT, GP} and references therein). The main theme of the present article is to explore existence of solutions in the unstable case. Such a solution will of course necessarily be singular. In the present work, we focus our attention on two protypical and important examples - the $J$-equation and the deformed Hermitian-Yang-Mills (dHYM) equations - to illustrate this point of view.

To set the stage, we first consider the more classical case of the complex Monge-Amp\`ere equations. It follows from Yau's resolution of the Calabi conjecture \cite{Yau} that if $\al$ is a K\"ahler class, then there exists a unique $\chi\in \al$ solving $$\begin{cases}\chi^n = e^f\omega^n\\\chi>0,\end{cases}, $$ where $f\in C^\infty(X)$ satisfies the normalization $\int_X e^f \omega^n = \alpha^n$. On the other hand, by Demailly and Paun's generalization of the classical Nakai-Moishezon criteria \cite{DemPa}, a necessary and sufficient condition for $\al$ to be K\"ahler, and hence for the existence of a solution to the above Monge-Amp\`ere equation, is that for any analytic set $Z^m\subset X$ of pure dimension $m$, and any $k=0,1,\cdots,m$ we have $$\int_Z \al^{k}\wedge\be^{m-k}>0.$$ Note that to check that $\al$ is a K\"ahler class requires the construction of a smooth, real, positive $(1,1)$ form in the class $\al$, and is a {\em pointwise} condition. On the other hand, the above criteria is completely numerical, and is in practice much easier to verify.  A natural question to ask is whether the complex Monge-Ampere equation can be solved in a class that is not K\"ahler or even numerically effective (nef), or equivalently when some of the intersection numbers above are negative. The optimal result in this direction was obtained in \cite{BEGZ}. Namely, if $\al$ is a big class, and $\Omega$ is a Radon measure that does not charge pluri-polar sets, then there exists a positive current $T\in \al$ such that $$\langle T^n\rangle = c\Omega.$$ Here $\langle\cdot \rangle$ denotes the non-pluripolar product introduced by Bedford-Taylor \cite{BT}, and $c$ is  the ratio of the volume $\vol(\al)$ of the big class $\al$ and the total mass $\Omega(X)$. In particular we can take $\Omega = \omega^n$, and  think of $c=  \vol(\al)/\be^n$ as the (Monge-Amp\`ere) {\em slope} of the pair $(\al,\be)$. It is well known (cf. \cite{fujita,BDPP,bouck-vol, Bou-Zar}) that there is an intersection theoretic characterization of the slope in this context, namely, $$c = \sup\frac{(\pi^*\al - D)^n}{\pi^*\be^n},$$ where the supremum is taken over all modifications $\pi:X'\rightarrow X$ such that $\pi^*\al - D$ is big and nef. At least for surfaces, it follows from the Zariski decomposition that the slope is in general larger than the topological slope $\al^2/\be^2$.

Motivated by the analysis of complex Monge-Amp\`ere equations in big classes, one can ask for similar weak solutions for other complex Hessian-type equations using the non-pluripolar product (cf. \cite{Dinew} for viscosity type weak solutions for some classes of complex Hessian equations). In the present work we develop this circle of ideas for two prototypical equations, namely the $J$-equation and the deformed Hermitian-Yang-Mills (dHYM) equations. We show that for surfaces and certain projective bundles, there is an analogous connection between weak solutions and minimal slopes. We also propose several conjectures in higher dimensions which we will explore further in future work.

\subsection{The $J$-equation}

Suppose now that $\al$ is a K\"ahler class, and $\be$, as before, is the K\"ahler class of a fixed K\"ahler form $\omega$.  We say that a smooth K\"ahler form $\chi \in \alpha$ satisfies the $J$ equation if 
\begin{equation}\label{jeqn}
n \frac{ \chi^{n-1}\wedge \omega}{\chi^n} = \mu, 
\end{equation}
where $\mu>0$ is a constant. The $J$-equation was introduced as a moment map equation in \cite{Don-moment} and in the context of finding lower bounds of Mabuchi energy in \cite{XXChen-2}. If equation (\ref{jeqn}) admits a smooth solution, then the constant $\mu$ is given by the topological intersection number defined below as the $J$-slope.

\begin{defn}  The $J$-slope of the triple $(X^n, \alpha, \beta)$  is defined by
\begin{equation}
\mu=\mu(X, \alpha, \beta) =n \frac{  \alpha^{n-1}\cdot \beta}{\alpha^n}. 
\end{equation}
\end{defn}
If  $\beta$ is a very ample divisor and $D$ a general divisor in it's linear system,  one can write the $J$-slope $\mu$ as
$$ \mu = n \frac{ \vol_D (\alpha |_D)} { \vol_X (\alpha)}.$$
It is shown in \cite{SW,gchen, DP, So} that  equation (\ref{jeqn}) admits a smooth solution if and only if 
$$
\mu(Z, \alpha|_Z, \beta|_Z) < \mu(X, \alpha, \beta),
$$
for any proper analytic subvariety $Z$ of $X$. In analogy with slope stability for hermitian bundles over K\"ahler manifolds, we then have the following natural notion of $J$-stability. 

\begin{defn} We say that the triple $(X, \alpha, \beta)$ is {\em $J$-slope semi-stable} (resp. stable) if  
\begin{equation*}\label{jsemistab}
\mu(Z, \alpha|_Z, \beta|_Z) \leq \mu(X, \alpha, \beta), ~\text{ (resp. $<$)}
\end{equation*}
for any proper analytic subvariety $Z$ of $X$. We say it is {\em $J$-slope unstable} if it is not $J$-slope semi-stable.

\end{defn}

In this paper, we will investigate the case when the $J$-stability {\em fails} and look for canonical solutions for the $J$-equation with singularities. In analogy with the notion of volume, we will first have to modify the notion of the $J$-slope. 

\begin{defn} We define the minimal $J$-slope for the triple $(X^n, \alpha, \beta)$  by 

\begin{equation}\label{defminj}
\zeta_{min} = \inf_{\pi, D} \left\{ n \frac{  (\pi^* \alpha - D)^{n-1} \cdot \pi^*\beta}{ (\pi^* \alpha - D)^n}  \right\}
\end{equation}
where the infimum is taken over all modifications $\pi: X' \rightarrow X$ and effective $\mathbb{R}$-divisors $D$ for which $\pi^*\alpha -D$ is big and nef. 

\end{defn}

The inverse of the minimal $J$-slope is a natural analogue for the volume of a big class on a K\"ahler manifold in the context of the $J$-equation. Our first result is to solve the $J$-equation for K\"ahler surfaces, regardless of $J$-stability. 

\begin{thm} \label{mainthm1} Let $(X, \alpha, \beta)$ be a compact K\"ahler surface with two K\"ahler classes. Let $\omega$ be a smooth K\"ahler metric in $\beta$. Then there exists a unique K\"ahler current $\sT \in \alpha$ solving the $J$-equation 
\begin{equation}\label{surfjeqn}
2\langle \mathcal{T}   \wedge \omega \rangle = \zeta_{min} \langle \mathcal{T}^2 \rangle, 
\end{equation}
where $\zeta_{min}$ is the minimal $J$-slope for $(X, \alpha, \beta)$ in Definition \ref{defminj}, and $\langle~,~ \rangle$ is the non-pluripolar product for positive currents. Furthermore, 

\begin{enumerate}

\item There exists an effective $\RR$-divisor $D$ on $X$ such that $$\zeta_{min} =   2 \frac{  (\alpha - D) \cdot \beta}{ ( \alpha - D)^2}.$$ In particular, the infimum is achieved without taking modifications. 
\item  $(X, \alpha, \beta)$ is $J$-slope unstable, if and only if  $\zeta_{min}< \mu.$

\end{enumerate}

\end{thm}

%
Analogous to the solution of the the complex Monge-Ampere equations in big classes and their relation with volume, we conjecture the following extension of Theorem \ref{mainthm1}.
\begin{conj} \label{conj1}Let $X$ be an $n$-dimensional compact K\"ahler manifold with two K\"ahler classes $\alpha$ and $\beta$. For any K\"ahler form $\omega\in \beta$, there exists a unique K\"ahler current $\mathcal{T} \in \alpha$ solving the $J$-equation
\begin{equation}
n\langle \mathcal{T}^{n-1} , \omega \rangle = \zeta_{min} \langle \mathcal{T}^n \rangle, 
\end{equation}
where $\zeta_{min}$ is the minimal $J$-slope for $(X, \alpha, \beta)$ as defined above, and $\langle~,~ \rangle$ is the non-pluripolar product for positive currents.

\end{conj}

Conjecture \ref{conj1} indeed holds when $(X, \alpha, \beta)$ is $J$-slope semi-stable (\cite{DMS}), and in this case one can prove that $\zeta_{min}= \mu$. We also conjecture that the solutions $\sT$ are canonical in the sense that they appear as limits of the $J$-flow, and also as limits of sequences of K\"ahler forms that minimise the moment map energy  $\sE_\omega$ introduced by Donaldson.  The moment map energy $\sE_\omega: \mathrm{PSH}(X, \chi)\cap C^\infty(X, \RR)\rightarrow \RR_+ $ is defined as $$\sE_\omega(\varphi) = \frac{1}{\al^n}\int_X\Big(\Lambda_{\chi_\varphi}\omega-\mu\Big)^2\chi_\varphi^n, ~\chi_\varphi = \chi+ \ddbar \varphi$$
for a fixed $\chi\in \alpha$.   
We can also view $\sE_\omega$ as a function on $\sK(\al)$, the set of all K\"ahler metrics in $\alpha$, by setting $\sE_\omega(\chi_\varphi) = \sE_\omega(\varphi)$. 

Next, we turn our attention to the convergence of the $J$ flow. Introduced in \cite{xx-chen} as the gradient flow of $\sE_\omega$, the $J$ flow can be written as $$\begin{cases}\frac{\partial \varphi}{\partial t} = \mu - \La_{\chi_\varphi}\omega\\ \chi_\varphi := \chi + \ddbar\varphi >0. \end{cases}$$  Convergence of the flow under the condition that there exists a sub-solution was established in \cite{SW} building on earlier works in \cite{W1,W2}. Our main result on convergence is the following.

\begin{thm} \label{mainthm2}  Let $X= \textnormal{Proj}_{M} (\mathcal{O}_M\oplus L^{\oplus(m+1)})$ be a projective bundle over an $n$-dimensional K\"ahler manifold $M$ with  $L\rightarrow M$ being a negative line bundle, and let $P_0$ denote the zero section of $X$. If $\chi_0\in \alpha$ and $\omega\in \beta $ are $U(m+1)$-invariant K\"ahler metrics, then the following hold for the solution $\chi(t)$ of the $J$-flow for  $\omega$ and $\chi(0)= \chi_0$.

\begin{enumerate}

\item If $(X, \alpha, \beta)$ is $J$-slope stable, then the flow converges smoothly to a unique solution $\chi_\infty$ solving  the $J$-equation 
$$ \Lambda_{\chi_\infty}\omega= (n+m+1) \frac{\alpha^{n+m}\cdot \beta}{\alpha^{n+m+1}}. $$

\medskip

\item If $(X, \alpha, \beta)$ is $J$-slope semi-stable, then the flow converges smoothly to a unique solution $\chi_\infty$ on $X\setminus P_0$ to the $J$-equation  
$$  \Lambda_{\chi_\infty}\omega= (n+m+1) \frac{\alpha^{n+m}\cdot \beta}{\alpha^{n+m+1}} $$
on $X\setminus P_0$.
Furthermore, $(X, \chi(t))$ converges in Gromov-Hausdorff sense to the metric completion of $(X\setminus P_0, \chi_\infty)$, which is homeomorphic to $X$ itself.

\medskip

\item If $(X, \alpha, \beta)$ is $J$-slope unstable, then the flow converges smoothly to a unique solution $\chi_\infty$ on $X\setminus P_0$ to  the $J$-equation  
$$ \Lambda_{\chi_\infty}\omega=  \zeta_{inv} < (n+m+1) \frac{\alpha^{n+m}\cdot \beta}{\alpha^{n+m+1}} $$
on $X\setminus P_0$, where $\zeta_{inv}$ is the $U(m+1)$-invariant minimal slope for $(X, \alpha, \beta)$ defined by $$\zeta_{inv} = \inf\left\{(n+m+1)\frac{\tilde\al_s^{m+n}\cdot\tilde\be}{\tilde\al_s^{m+n+1}}~|~ \tilde\al_s = \tilde\pi^*\al - sE \text{ is ample},~\tilde\be=\tilde\pi^*\be\right\},$$ and $\tilde\pi: \tilde X= \Bl_{P_0}X \rightarrow X$ is the blow up of $X$ along $P_0$ with exceptional divisor $E$ \footnote{Note that when $m=0$, $\tilde X = X$ and $E = P_0$.}. 
Furthermore, $\tilde\pi^*\chi_\infty$ extends to a conical metric $\tilde \chi_\infty$ on $\tilde X$ with cone angle $\pi$ along $E$ in the class $\tilde\pi^*\al - \la[E]$, where $\la$ is the unique solution to $$(n+m+1)\frac{\tilde\al_\la^{m+n}\cdot\tilde\be}{\tilde\al_\la^{m+n+1}} = \frac{n}{1+\la}.$$ In particular, the metric completion of $(X \setminus P_0, \chi_\infty)$ is isometric to $(\tilde X, \tilde\chi_\infty)$. 

\medskip

\item\label{mainthem2-energy-conv} The moment map energy for $\chi(t)$ converges to its infimum, i.e., $$\lim_{t\rightarrow \infty} \mathcal{E}_\omega(\chi(t)) = \inf_{\chi\in \mathcal{K}(\alpha)}\mathcal{E}_\omega(\chi).$$

\end{enumerate}

\end{thm}

Note that apriori,  that $\zeta_{inv}$ could be larger than $\zeta_{min}$. However, we conjecture that the two are actually equal. We now make a few additional remarks on the theorem above. Firstly, the first three parts of Theorem \ref{mainthm2} are proved in \cite{FL} for the case when $M = \PP^n$ and $L = \sO(-1)$. However, there are some crucial technical details missing in the proofs in \cite{FL}, and we take this opportunity to present complete proofs (cf. Remark  \ref{rem:fl-remarks}). In line with the philosophy of the present work,  we also prove that the limiting current satisfies a singular $J$-equation with a modified slope, and interpret the modified slope as the minimal (semi-stable) slope of classes on the blow-up of $X$ along $P_0$. The limiting current clearly has mass concentration at the zero section $P_0$ of $X$. This indicates a bubbling behaviour if one applies suitable gauge changes to the flow near $P_0$. While we cannot establish the bubbling along the $J$ flow at the moment, we can produce a minimizing sequence for the moment map energy that exhibits the expected bubbling. We now explain this new phenomenon. 

Let $Y$ be a copy of $X$ and let $D^Y_\infty$ be the $\infty$-section of $Y$ as the $\mathbb{P}^{m+1}$-bundle over $M$. Both $D^Y_\infty$ and the exceptional divisor $E$ in $\tilde X$ are identical.  We now let 
$$\tilde X ~\# ~Y$$
be the connect sum of $\tilde X$ and $Y$ by identifying $E\subset \tilde X$ and $D^Y_\infty\subset Y$.  Then we have the following theorem on the bubbling behaviour of a sequence of K\"ahler metrics $\chi_k$ that minimize the moment map energy $\sE_\omega$.  

\begin{thm}\label{thm:j-eq-bubbling}  Let $(X,\al,\be)$ be as in the previous theorem. Suppose $(X, \alpha, \beta)$ is $J$-slope unstable. Then for any $U(m+1)$-invariant K\"ahler metric $\omega\in \beta$, there exist a sequence of $U(m+1)$-invariant K\"ahler metrics $\chi_j \in \alpha$,  a smooth $U(m+1)$-invariant K\"ahler metric $\theta$ and a sequence of biholomorphisms 
$$\Phi_j: X \rightarrow X$$ satisfying the following. 

\begin{enumerate}

\item  Let $(Y, \theta_\infty)$ be the Cheeger-Gromov limit of $(X, \Phi^*_j\theta)$. Then $(Y, \theta_\infty)$ is isometric to $(X, \theta)$ and
 $\omega_j = \Phi^*_j \omega$ converges smoothly to $\omega_\infty$ on $Y$, the pullback of a K\"ahler form on $M$ by $\pi: Y \rightarrow M$. 
\medskip

\item Let $\tilde X \# Y$ be the union of $\tilde X$ and $Y$ by gluing $E\subset X$ and $D^Y_\infty\subset Y $. Then $\chi_\infty|_{E} = \theta_\infty|_{D^Y_\infty}$ and $$(X, \chi_j) \rightarrow (\tilde X \# Y, d_\infty)$$ in Gromov-Hausdorff topology as $j\rightarrow \infty$, where $\chi_\infty$ is the limiting conical K\"ahler metric in (3) of Theorem \ref{mainthm2} and $d_\infty$ is a certain distance function to be defined later. %

\item The convergence is smooth on $\tilde X\setminus E$ and $Y\setminus D^Y_\infty$.  Moreover we have
$$ \lim_{j\rightarrow \infty} \sE_\omega(\chi_j) = \inf_{\chi \in \sK(\alpha)}  \sE_\omega(\chi)= \sE_\omega(\chi_\infty) + \sE_{\omega_\infty} (\theta_\infty), $$
where $\sE_{\omega_\infty} (\theta_\infty)$ is the moment map energy on $Y$.

\medskip

\end{enumerate}

\end{thm}

We remark that $\tilde X ~\# ~Y$ also coincides precisely with a special algebraic degeneration of $X$, namely the deformation to the normal cone of $X$ along $P_0$ obtained by blowing-up $P_0 \subset X\times \{\infty\}$ in the product $X\times \mathbb{P}^1$. We expect that if $\chi(t)$ is the solution to the $J$-flow from Theorem \ref{mainthm2}, then $(X,\chi(t))$ should converge in Gromov-Hausdorff distance to $(\tilde X \# Y,d_\infty)$, but the proof probably requires a finer analysis of the metrics $\chi(t)$ and their blow-up behaviour near the zero section $P_0$. 

The bubbling phenomena should be rather universal in the $J$-unstable case.  Indeed we have the following general speculation: Let $X$ be a K\"ahler manifold with two K\"ahler metrics $\chi_0\in \alpha$ and $\omega\in \beta$. Then the $J$-flow starting with $\omega_0$ should converge smoothly away from an analytic subset $Z$ of $X$ to the K\"ahler current $\sT$ as in Conjecture \ref{conj1}. Furthermore, we expect that the metric completion of $(X\setminus Z, \sT)$ would be a modification $\pi: \tilde X \rightarrow X$, and that $\sT$ would extend to a positive current $\tilde \sT$ on $\tilde X$. Moreover, the infimum of the moment map energy should be equal to the sum of moment map energies for solutions of minimal $J$-equations on $X$ and components of $\pi^{-1}(Z)$. In particular, it is plausible that the flow will converge to a family of solutions to the minimal $J$-equations on $X$ and the $J$-unstable subvariteies of $X$. Such filtration of $J$-equations on $X$ and its destabilising sub-varieties should be viewed as an analogue of the Harder-Narasimhan-Seshadri filtration for Hermitian Yang-Mills equations. It would also be interesting to relate our work with the investigations carried out in \cite{Hattori, Zach, KS} on optimal degenerations for the $J$-functional and maximally destabilising analytic sets. Even in the surface case, one of the novelties in the present work is to related the maximally destabilising divisors to an appropriate Zariski decomposition, and to find a canonical weak solution to the $J$ equation.

\subsection{The deformed Hermitian Yang-Mills equation}

 Suppose now that $\al\in H^{1,1}(X,\RR)$ is a class that is not necessarily positive. The deformed Hermitian-Yang-Mills (dHYM) equation, introduced in \cite{JY} as a mirror to the special Lagrangian equation, seeks a smooth closed real $(1,1)$ form $\chi$ in $\al$ such that $$\mathrm{Re}(\chi+\sqrt{-1}\omega)^n = (\cot\vartheta)\mathrm{Im}(\chi+\sqrt{-1}\omega)^n,$$ where $\vartheta \in (0,\pi)$ is a topological constant defined below. The dHYM also arises as a moment map equation - the reader can refer to \cite{CY} for more details. By integrating both sides of the equation, we see that $cot\vartheta$ is a topological invariant. Analogous to the case of the $J$-equation, we have the following definitions. 
 
 \begin{defn} Let $X$ be a K\"ahler manifold of complex dimension $n$ and let $\alpha, \beta\in H^{1,1}(X, \mathbb{R})$. Suppose $\beta$ is big and semi-positive. We define the global dHYM-angle 
$$\vartheta= \vartheta(X, \alpha, \beta) \in (0, \pi)$$
 for the triple $(X, \alpha, \beta)$ by
$$\\Re(\alpha  +\sqrt{-1} \beta)^n  = (\cot \vartheta) \\Im (\alpha +\sqrt{-1} \beta)^n, $$
whenever $Im (\alpha +\sqrt{-1} \beta)^n\neq 0$.

\begin{enumerate}

\item $(X, \alpha, \beta)$ is said to be {\em dHYM-stable} if for any K\"ahler class $\gamma$, $$ \\Re(\alpha  +\sqrt{-1} \beta)^k \cdot \gamma^{n-k} \cdot X \geq  (\cot \vartheta)\\Im(\alpha  +\sqrt{-1} \beta)^k \cdot \gamma^{n-k} \cdot X$$ and for any proper $m$-dimensional subvariety $Z$ of $X$ and $1\leq k \leq m$, 
$$ \\Re(\alpha  +\sqrt{-1} \beta)^k \cdot \gamma^{m-k} \cdot Z > (\cot \vartheta)\\Im(\alpha  +\sqrt{-1} \beta)^k \cdot \gamma^{m-k} \cdot Z.$$

\item $(X, \alpha, \beta)$ is said to be {\em dHYM-semi-stable} if for any K\"ahler class $\gamma$,  any $m$-dimensional subvariety $Z$ of $X$ and $1\leq k \leq m$, 
$$ \\Re(\alpha  +\sqrt{-1} \beta)^k \cdot \gamma^{m-k} \cdot Z \geq (\cot \vartheta)\\Im(\alpha  +\sqrt{-1} \beta)^k \cdot \gamma^{m-k} \cdot Z.$$

\item $(X, \alpha, \beta)$ is said to be {\em dHYM-unstable} if it is not dHYM-semi-stable.
\end{enumerate}

\end{defn}

In \cite{CJY, gchen, CLT}, it is proved that the deformed Hermitian-Yang-Mills equation has a solution if and only if $(X,\al,\be)$ is dHYM-stable (cf. also \cite{B} for a different argument for projective three-folds based on \cite{DP}). Once again, we would like to investigate the solvability in the unstable case. We obtain results in two dimensions and pose some conjectures in higher dimensions.  
 
\begin{defn}

 Let $X$ be a K\"ahler surface with two cohomology classes  $\alpha, \beta \in H^{1,1}(X, \mathbb{R})$. Suppose $\beta$ is K\"ahler and $\alpha \cdot \beta >0$. Then the minimal angle $\theta_{min}\in (0, \pi)$ for $(X, \alpha, \beta)$ is defined by 
 \begin{equation}
\cot \theta_{min} := \sup_{D}   \frac{(\pi^*\al - D)^2 - (\pi^*\be)^2}{2(\pi^*\al-D) \cdot (\pi^*\be)},
 \end{equation}
  where the supremum is taken over all effective $\RR$-divisors $D$ on modifications $\pi:X'\rightarrow X$ such that $$ (\pi^*\al - D)\cdot\pi^*\be>0.$$

\end{defn}

 The cotangent of the minimal angle $\theta_{min}$ should be thought of as the optimal slope for the dHYM equation.
\begin{thm}\label{main2} 
Let $X$ be a K\"ahler surface with two cohomology classes  $\alpha, \beta \in H^{1,1}(X, \mathbb{R})$. Suppose $\beta$ is K\"ahler and 
\begin{equation}\label{albe0}
\alpha \cdot \beta >0.
\end{equation}
 For any K\"ahler form $\omega\in \beta$, 
there exists a unique closed current $\mathcal{T} \in \al$ such that 
$$ ~\\Re\langle (\mathcal{T}+\sqrt{-1}\omega)^2 \rangle = \left( \cot \theta_{min} \right) \\Im\langle (\mathcal{T}+\sqrt{-1}\omega)^2 \rangle,~  \mathcal{T}\geq \left( \cot \theta_{min} \right)\omega.$$ Here $\langle,\rangle$ denotes the non-pluripolar product, and the equality is that of measures. In fact, there exists an effective $\RR$-divisor $D$ on $X$ such that
 $$\cot \theta_{min}  = \frac{(\al - D)^2 - \be^2}{2(\al-D) \cdot \be}, ~~(\alpha-D)\cdot \beta>0.$$
In particular, one does not have to take modifications in the definition of $\cot\theta_{min}$. Moreover we have that $$ \cot\theta_{min} \geq \cot\vartheta(X,\al,\be)$$ with equality if and only if $(X,\al,\be)$ is $dHYM$-semistable.
 \end{thm}
Note that though $\mathcal{T}$ may not be a positive current, but since the current $\Theta :=  \mathcal{T} - (\cot\theta_{min})\omega$ is positive, the non-pluripolar product can be interpreted using multi-linearity. Explicitly, we set $$\langle(\mathcal{T} + \sqrt{-1}\omega)^2\rangle = \langle \Theta^2 \rangle  + 2(\cot\theta_{min} +\sqrt{-1})\langle\Theta \wedge \omega\rangle +  \left(\cot\theta_{min} + \sqrt{-1}\right)^2\omega^2. $$

The condition $\al\cdot\be>0$ is not strictly needed. For, if the intersection number were negative, one could simply work with $-\al$. We would also like to remark that if $\alpha$ is pseudoeffective, then $\alpha \cdot \beta$ is always nonnegative. The condition 
$$\alpha\cdot \beta>0$$ holds  if and only if the numerical dimension of $\alpha$ is greater than or equal to $1$. The conjectural picture in higher dimensions is much more subtle for the dHYM equations than for the $J$-equation, essentially due to the lack of positivity of the class $\al$.  We begin with the following definition.

\begin{defn} Let $X$ be a K\"ahler manifold of complex dimension $n>2$ and let $\alpha, \beta\in H^{1,1}(X, \mathbb{R})$. Suppose $\beta$ is big and semi-positive. We define the semi-stable admissible set  by 
\begin{eqnarray*}
&&\mathcal{V}(X, \alpha, \beta) \\
&=& \left\{ (\pi, X',  D): \pi: X' \rightarrow X~\textnormal{is~a~modification},  ~D~\textnormal{is~effective}, ~(X', \pi^*\al-D, \pi^*\beta)~\textnormal{is~semi-stable} \right\}
\end{eqnarray*}
%
and the minimal angle $\theta_{min} =\theta_{min}(X, \alpha, \beta) \in (0, \pi)$ by
$$\theta_{min} = \inf \left\{  \vartheta(X', \pi^*\alpha-D, \pi^*\beta): ~ (\pi, X', D)\in \mathcal{V}(X, \alpha, \beta) \right\}   $$
if $\mathcal{V}(X, \alpha, \beta)$ is not empty.
 
\end{defn}

\begin{conj} \label{dhymcon1} Let $X$ be a K\"ahler manifold of complex dimension $n$ and let $\alpha, \beta\in H^{1,1}(X, \mathbb{R})$. Suppose $\beta$ is K\"ahler, $\mathcal{V}(X, \alpha, \beta)$  is not empty and $\theta_{min}>0$. Then for any K\"ahler form $\omega\in \beta$,  there exists a unique closed current $\mathcal{T} \in \alpha$ such that 
$$\text{Re}\langle (\mathcal{T}  +\sqrt{-1} \omega)^n\rangle  = (\cot  \theta_{min}) \text{Im} \langle ( \mathcal{T} +\sqrt{-1} \omega)^n\rangle.$$

\end{conj}

 We shall see in Section \ref{subsec:dhym-surface}  that even when $n=2$, the condition $\al\cdot\be>0$ ensures that $\mathcal{V}(X, \alpha, \beta)$ is non-empty and the two definitions of minimal angles agrees with each other. 
  
 Next, we study a parabolic version of the dHYM equation. There are in fact several flows in literature whose critical points satisfy the dHYM equation (cf. \cite{JY, Tak}).  In the final section we focus on the cotangent flow introduced by Fu, Yau and Zhang in \cite{FYZ}, and prove a convergence result for the flow on $\Bl_{x_0}\PP^2$ with Calabi ansatz. We prove that the flow converges to the current $\sT$ constructed in the theorem above in a suitable weak sense. 

\begin{thm}\label{thm:conv-cotangent-flow-blow-up}
Let $X = \Bl_{x_0}\PP^2$, $\omega$ a fixed K\"ahler metric satisfying the Calabi ansatz in the class $\be = b[H] - [E]$ and $\al = p[H]-q[E]$ a real $(1,1)$ class with a representative $\chi \in \alpha$. We assume that $\al\cdot\be>0$. Let $\chi_0 := \chi+\ddbar\underline{\varphi}\in \al$  satisfying Calabi ansatz with $\La_{\omega}\chi_0 >0$, and let $\chi_t := \chi_{\varphi_t}$ be the cotangent flow emanating from $\chi_0$. Let $$c_0 = \cot\vartheta(X,\al,\be).$$ Then there exists a choice of $\chi_0$ (to be specified later) such that the following convergence behaviour holds:
\begin{enumerate}
\item If $q>c_0$, then $(X,\al,\be)$ is dHYM-stable and $\chi_t $ converges smoothly to a solution $\chi_\infty$ of the dHYM equation.
\item If $q = c_0$, then $(X,\al,\be)$ is dHYM-semi-stable but not stable. In this case $\chi_t\rightarrow\chi_\infty$ as currents, where $\chi_\infty$ has bounded local potentials and satisfies the dHYM equation $$\text{Re}(\chi_\infty+\sqrt{-1}\omega)^2 = c_0\text{Im}(\chi_\infty + \sqrt{-1}\omega)^2,$$ where the wedge product is taken in the sense of Bedford-Taylor. 
\item If $q<c_0$, then $(X,\al,\be)$ is dHYM-unstable. In this case also $\chi_t\rightarrow \chi_\infty$ as currents. We can write $\chi_\infty = \chi_\infty' + (\zeta - q)[E]$, where $\chi_\infty'$ has bounded local potentials and $$\zeta := \sup\Big\{\frac{(\al_t')^2 - \be^2}{2\al_t'\cdot\be}~|~ \al_t' = \al - tE,~\al_t'\cdot\be>0\Big\}.$$ Moreover $\zeta = \cot\theta_{min}$, and $\chi_\infty'$ solves $$\text{Re}(\chi_\infty' + \sqrt{-1}\omega)^2 = (\cot\theta_{min})\text{Im}(\chi_\infty' + \sqrt{-1}\omega)^2$$ on $X\setminus E$. Globally, $\chi_\infty$ solves $$\langle \chi_\infty,\chi_\infty\rangle - \omega^2 = 2\zeta\langle \chi_\infty,\omega\rangle,$$ and so $\chi_\infty = \sT$ that we construct in Theorem \ref{main2}. 
\end{enumerate}
\end{thm}

In fact, by adapting the arguments in the proof of Theorem \ref{mainthm2}, it is possible to prove that the cotangent flow has a unique limit for {\em any} $\chi_0$ satisfying the Calabi ansatz and $\La_\omega\chi_0>0$. The fact that $\zeta = \cot\theta_{min}$ follows from the fact that both are (unique) solutions of a certain quadratic equation involving the volume function as we shall see in Sections \ref{sec:surface} and \ref{sec:dhym-blow-up} below. The interested reader should also refer to the recent preprint \cite{ChJ} for some fascinating results on singularity formation for the line-bundle curvature flow with Calabi ansatz on $\Bl_{x_0}\PP^n$ in the non-super critical and semi-stable cases. We expect that similar results should also hold for the co-tangent flow.

\section{Singular solutions to complex Hessian equations on K\"ahler surfaces}\label{sec:surface}

\setcounter{equation}{0}

\subsection{The $J$-equation on K\"ahler surfaces}\label{subsec:j-surface}

The aim of this section is to prove Theorem \ref{mainthm1}. The starting point of the proof is the following elementary observation.
\begin{lem}\label{lem:bigness-jequation} 
Let $\al$ be a big and nef class, and $$t\geq 2\frac{\al\cdot\be}{\al^2}.$$ Then the class $\ga:= \al - t^{-1}\be$ is big.
\end{lem}
\begin{proof}
Note that $\al^2>0$. For any $\vep>0$ we set $\al_\vep = \al + \vep\be$. By Lamari's criteria \cite{Lam}, $\ga$ is big if $\ga^2 >0$ and $\ga\cdot\al_\vep>0$ for some $\vep>0$, since $\al_\vep$ itself is K\"ahler. But these two conditions are easily verified. Indeed we have
\begin{align*}
\ga^2 = \al^2 + \frac{1}{t^2}\be^2 - \frac{2}{t}\al\cdot\be \geq \frac{1}{t^2}\be^2>0\\
\ga\cdot\al_\vep = \al^2 - \frac{1}{t}\al\cdot \be + \vep\Big(\al - \frac{1}{t}\be\Big)\cdot\be \geq \frac{1}{4}\al^2>0,
\end{align*}
if we choose $\vep<<1$ so that $$\vep\left|\Big(\al - \frac{1}{t}\be\Big)\cdot\be\right| < \frac{\al^2}{4}.$$
\end{proof}
We now let 
\begin{align*}
\mu_0 &= \inf\{t>0~|~ \ga_t:= \al-\frac{1}{t}\be \text{ is big}\}
\end{align*}
Since $\be$ is K\"ahler, clearly $\mu_0>0.$ On the other hand, by the lemma above we have $\mu_0 < \mu$. \begin{lem}\label{lem:j-volume-equation}
There exists a unique $\xi \in (\mu_0,\mu]$ such that $$\vol\Big(\al - \frac{1}{\xi}\be\Big) = \frac{1}{\xi^2}\be^2.$$ Moreover, $t\geq \xi$ if and only if $$\vol\Big(\al - \frac{1}{t}\be\Big) \geq \frac{1}{t^2}\be^2.$$
\end{lem}
\begin{proof}
For $s\in \Big[\frac{1}{\mu},\frac{1}{\mu_0}\Big)$, consider the function $$f(s) = \frac{\vol(\al - s\be)}{s^2\be^2}.$$ The function is clearly a continuous and strictly decreasing function in $s$. The continuity of $f(s)$ of course uses the continuity of the volume function on $H^{1,1}(X,\RR)$.  We also have that $f(\mu_0^{-1}) = 0$ and $$f(\frac{1}{\mu}) = \mu^2\frac{\vol(\al - \frac{1}{\mu}\be)}{\be^2} \geq \mu^2\frac{(\al - \frac{1}{\mu}\be)^2}{\be^2} = 1,$$ and so by the intermediate value theorem there must be a unique $s_0$ such that $f(s_0) = 1$. We then take $\xi = 1/s_0$.
\end{proof}

\begin{cor}\label{cor:sing-solution-j-equation}
There exists a unique K\"ahler current $\mathcal{T}\in \al$ such that $$2\langle \mathcal{T},\omega\rangle = \xi \langle \mathcal{T}^2 \rangle,~~\mathcal{T} \geq \frac{1}{\xi}\omega.$$
\end{cor}
\begin{proof}
By the above lemma and Theorem 4.1 in \cite{BEGZ}, there exists a unique positive current $\Theta \in \al - \frac{1}{\xi}\be$ such that $$\langle \Theta,\Theta\rangle = \frac{1}{\xi^2}\omega^2.$$ Then $\mathcal{T}:= \Theta + \frac{1}{\xi}\omega$ is the required current. 
\end{proof}

\begin{prop} We have the following characterization of $\xi$:
$$\xi = 2\inf_D\frac{(\al-D)\cdot\be}{(\al-D)^2},$$ where the infimum is taken over all effective $\RR$-divisors $D$ on $X$ such that $\al - D$ is big and nef. Moreover, $\xi <\mu$ if and only if $(X,\al,\be)$ is unstable. 
\end{prop}

\begin{proof}
We call the right hand side $\zeta$. We first prove that $\zeta\leq \xi$. To see this, let $$\al - \frac{1}{\xi}\be = Z + N$$ be the Zariski decomposition, where $Z$ is the nef part and $N$ is the negative and effective part. Then $\al' := \al - N$ is K\"ahler, and so $$\zeta \leq 2\frac{\al'\cdot\be}{(\al')^2}  = \xi.$$ To see the ratio above is equal to $\xi$ we simply note that $$\Big(\al' - \frac{1}{\xi}\be\Big)^2 = Z^2 = \vol\Big(\al - \frac{1}{\xi}\be\Big) =  \frac{1}{\xi^2}\be^2.$$ On the other hand, suppose $\zeta<\xi$. Let $t\in (\max(\zeta,\mu_0),\xi)$. Then there exists an effective divisor $D$ on $X$ such that $\al' = \al - D$ is big and nef, and $$t > 2\frac{\al'\cdot\be}{(\al')^2}.$$ By Lemma \ref{lem:bigness-jequation} $\al' - t^{-1}\be$ is a big class.  We also have that $\al - \al' = D$ is pseudoeffective, so $$\vol\Big(\al - \frac{1}{t}\be\Big) \geq \vol\Big(\al'- \frac{1}{t}\Big) \geq \Big(\al'-\frac{1}{t}\be\Big)^2 > \frac{1}{t^2}\be^2,$$ and so $t > \xi$ by Lemma \ref{lem:j-volume-equation}, and we have a contradiction. 
\end{proof}
To complete the proof of the theorem \ref{mainthm1} we have the following: 
\begin{lem}
Let $\pi:X'\rightarrow X$ be a modification and $D$ an effective $\RR$-divisor such that $\pi^*\al - D$ is big and nef. Then $$2\frac{(\pi^*\al-D)\cdot\pi^*\be}{(\pi^*\al-D)^2} \geq \xi,$$ and hence in particular that $\zeta_{min} = \xi$.
\end{lem}
\begin{proof} We let $\al' = \pi^*\al - D$, $\be' = \pi^*\be$ and $\ga' = \al' - \xi^{-1}\be'$. We argue by contradiction. So suppose $$2\frac{\al'\cdot\be'}{(\al')^2} < \xi.$$ This is equivalent to $$(\ga')^2 > \frac{1}{\xi^2}(\be')^2 = \frac{1}{\xi^2}\be^2.$$ Once again by Lemma \ref{lem:bigness-jequation}, $\ga'$ is big. Now if we denote $\ga = \al - \xi^{-1}\be$, then $\pi^*\ga - \ga' = D$ is effective and so $$\vol(\ga) \geq \vol(\pi^*\ga) \geq \vol(\ga')\geq (\ga')^2 > \frac{1}{\xi^2}\be^2,$$ which is a contradiction since $\vol(\ga) = \xi^{-2}\be^2$ by Lemma \ref{lem:j-volume-equation}.
\end{proof}

\begin{rem}
A nice consequence of the theorem \ref{mainthm1}, and it's proof, is that the minimal slope is achieved. That is, there exists an effective $\RR$-divisor $D$ on $X$ such that $\al ' = \al - D$ is big and nef and $$\zeta_{min} = 2 \frac{\al'\cdot\be}{(\al')^2}.$$ As in the proof above this immediately follows from the Zariski decomposition $$\al - \frac{1}{\xi}\be= Z+D.$$ Moreover, since $Z$ is big and nef, we can write $$\zeta_{min} = \min\left\{\frac{2\al'\cdot\be}{(\al')^2}~|~ \al' = \al-D \text{ big and nef}, ~ D \text{ effective $\RR$-divisor and }(X,\al',\be) \text{ is $J$-slope semi-stable}\right\}.$$ That is, the minimum can be taken over the semi-stable pairs. 
\end{rem}

\begin{rem}
When $(X,\al,\be)$ is unstable, we also note that the divisor $D$ is a destabilising divisor ie. $\mu_D > \mu_X$. This follows easily be rearranging the following basic estimate and using the fact that $D^2<0$: $$2\frac{(\al-D)\cdot\be}{(\al - D)^2} = \zeta_{min} < \mu = 2\frac{\al\cdot\be}{\al^2}.$$ 
 
\end{rem}

\subsection{The deformed Hermitian-Yang-Mills equations on surfaces}\label{subsec:dhym-surface}

We now move on to analysing the deformed Hermitian Yang Mills equations. We once again first specialize to the  surface case. Without loss of generality, we may  assume that 
\begin{equation}\label{eq:signconv-dhym-surfaces}
\al\cdot\be>0.
\end{equation}
For, if $\al\cdot\be$ were negative, we could simply simply replace $\al$ by $-\al$ and continue with the analysis.
We let  $$\zeta_H := \sup_{D} \frac{(\al-D)^2 - \be^2}{2(\al - D)\cdot\be},$$ where the supremum is taken over all $\RR$-effective divisors $D$ on $X$ such that $$ (\al - D)\cdot\be>0.$$ 

The main theorem we wish to prove in this section is the following. 

\begin{thm}\label{thm:dhym-existence-singular-surface}
There exists a unique closed current $\sT\in \al$ such that $$\Re\langle( \sT+\sqrt{-1}\omega)^2\rangle = \zeta_H\cdot\Im\langle(\sT+\sqrt{-1}\omega)^2\rangle,~~ \sT\geq \zeta_H\omega.$$ Here $\langle,\rangle$ denotes the non-pluripolar product, and the equality is that of measures. Moreover the global slope is achieved. That is, there exists an effective $\RR$-divisor $D$ on $X$ such that $$\zeta_H = \frac{(\al-D)^2-\be^2}{2(\al-D)\cdot\be}.$$\end{thm}

We will prove shortly that $\zeta_H = \cot\theta_{min}$, and this will in particular complete the proof of Theorem \ref{main2} . We first need the following strengthening of Lemma \ref{lem:bigness-jequation}.
\begin{lem}\label{lem:bigness-dhym}
Let $\al \in H^{1,1}(X,\RR)$  and $\be$ be a K\"ahler class such that $\al\cdot\be>0$. Let $$t\leq\frac{\al^2 - \be^2}{2\al\cdot\be}.$$ Then $\al - t\be$ is a big class. In particular, with $c_0 = \cot \vartheta(X,\alpha,\beta)=\frac{\alpha^2 - \beta^2}{2\alpha\cdot\beta}$ as above, $\al - c_0\be$ is a big class.
\end{lem}
\begin{proof}
By Lamari's criteria \cite{Lam}, a class $\ga$ on $X$ is big if $\ga^2>0$ and $\ga\cdot\be >0$. We let $\ga = \al-t\be$. Then $$\ga^2 = \al^2 - 2t\al\cdot\be +t^2\be^2\geq (1+t^2)\be^2>0.$$ If $t<0$, then $$\ga\cdot\be = \al\cdot\be - t\be^2 >\al\cdot\be > 0.$$ On the other hand if $t\geq 0$, then $\al^2\geq \be^2$. In particular, $\al$ itself is big again by Lamari \cite{Lam}. Let $\al = Z+N$ be the Zariski decomposition with $Z$ big and nef, and $N^2<0$. We let $\al' = \al - N$ and $\al_\vep = \al'+\vep\be$ which is K\"ahler for any $\vep>0$. Note that $$0< (\al')^2 = \al^2 - 2\al\cdot N + N^2<\al^2 - 2\al\cdot N,$$ and so $\al^2>2\al\cdot N$. We then have \begin{align*}
\ga\cdot\al_\vep &= \al^2 - t\al\cdot\be - \al\cdot N + t\be\cdot E + \vep\be\cdot(\al - t\be) \\
&\geq  \al^2 - \frac{\al^2-\be^2}{2} - \frac{\al^2}{2}+\vep\be\cdot(\al - t\be)\\
&= \frac{\be^2}{2} + \vep\be\cdot(\al - t\be) > 0,
\end{align*} for $\vep<<1$. Note that we used the fact that $\be\cdot N\geq 0$ in the second line. Once again  by Lamari \cite{Lam}, it follows that $\ga$ is big.

\end{proof}
We now let $$t_0= \sup\{t \in \RR~|~ \al - t\be \text{ is a big class}\}.$$ By the above lemma we have that $c_0<t_0$.
\begin{lem}\label{lem:dhym-volume-equation}
There exists a unique $\xi \in [c_0,t_0)$ such that
\begin{equation}\label{eq:volume equation on Kahler surface, unique solution xi}
\vol(\al-\xi\be) = (1+\xi^2)\be^2.
\end{equation}
We also have that $\vol(\al - t\be) < (1+t^2)\be^2$ if $t\in (\xi,t_0)$. Moreover, if $Z+N$ is the Zariski decomposition of the big class $\al - \xi\be$, then $$(Z + \xi\be)\cdot\be>0.$$
\end{lem} 
\begin{proof}
Consider the function $$f(t) = \vol(\al-t\be)-(1+t^2)\be^2.$$ By Deng's result on the derivative of the  volume function \cite{Deng} (see also \cite{Witt}), we have that $$f'(t) = -2\langle\al - t\be\rangle\cdot\be - 2t\be^2.$$ We claim that $f'(t)$ is increasing on $(-\infty,t_0)$, i.e.\ $f$ is convex. To see this, let $t_1<t_2$.  Then by the super-additivity of the positive intersection product, we have $$\langle \al -  t_1\be\rangle \cdot\be \geq \langle \al - t_2\be\rangle\cdot\be + (t_2-t_1)\be^2,$$ and so 
\begin{align*}
f'(t_1) &\leq -2\langle\al-t_2\be\rangle\cdot\be - 2(t_2-t_1)\be^2 - 2t_1\be^2 = f'(t_2).
\end{align*}
We also  have that $f(t_0) = -(1+t_0^2)\be^2 <0$ and $$f(c_0) = \vol(\al-c_0\be) - (1+c_0^2)\be^2 \geq (\al  - c_0\be)^2 -(1+c_0^2)\be^2 = 0. $$ By the intermediate value theorem there is at least one zero between $[c_0,t_0)$ and by convexity the zero must be unique. Moreover, $f(t)$ has to be negative for all $t>\xi$. Finally, note that $f'(\xi)<0$, and so 
\begin{align*}
0> f'(\xi) = -2(\langle\al - \xi\be\rangle\cdot\be + \xi\be^2) = -2(Z+\xi\be)\cdot\be.
\end{align*}
Here we used the fact that for if $\ga = Z+N$ is a Zariski decomposition, then $\langle \ga\rangle\cdot\be = Z\cdot\be.$
\end{proof}

\begin{cor}
There exists a unique closed current $\mathcal{T}\in \al$ such that $$\mathrm{Re}\langle( \mathcal{T}+\sqrt{-1}\omega)^2\rangle = \xi\cdot\mathrm{Im}\langle(\mathcal{T}+\sqrt{-1}\omega)^2\rangle,~~~ \mathcal{T}\geq \xi\omega.$$
\end{cor}
\begin{proof}
By Lemma \ref{lem:dhym-volume-equation} and Theorem $4.1$ \cite{BEGZ}, there exists a unique closed positive current $\Theta\in \al -\xi\be$ such that $$\langle \Theta^2 \rangle = (1+\xi^2)\omega^2.$$ Then $\mathcal{T} := \Theta + \xi\omega$ is the desired unique closed current in $\alpha$. Indeed,
\begin{align*}
&\mathrm{Re}\langle( \mathcal{T}+\sqrt{-1}\omega)^2\rangle - \xi\cdot\mathrm{Im}\langle(\mathcal{T}+\sqrt{-1}\omega)^2\rangle \\
&=	\mathrm{Re}\langle( \Theta+(\xi+\sqrt{-1})\omega)^2\rangle - \xi\cdot\mathrm{Im}\langle(\Theta+(\xi+\sqrt{-1})\omega)^2\rangle \\
&= \langle \Theta^2 \rangle + 2\xi \langle \Theta \wedge \omega \rangle + (\xi^2 -1)\omega^2 - 2\xi \langle \Theta \wedge \omega \rangle - 2\xi^2 \omega^2\\
&=0.
\end{align*}
\end{proof}

\begin{proof}[Proof of Theorem  \ref{thm:dhym-existence-singular-surface}] We only need to prove that $\xi = \zeta_H$. We first argue that $\xi \leq \zeta_H$.  To see this, let $$\al - \xi\be = Z + N$$ be the Zariski decomposition of the big class $\al - \xi\be$, where $Z$ is big and nef, and $N = \sum a_iD_i$ is an effective $\RR$-divisor with $(D_i\cdot D_j)$ negative definite. Note that by the choice of $\xi$, we have 
\begin{equation}\label{eq:dhym-existence-1}
Z^2 = (1+\xi^2)\be^2.
\end{equation}
Moreover, by the second part in Lemma \ref{lem:dhym-volume-equation}, $$(\al - N)\cdot\be = (Z+\xi\be)\cdot\be >0.$$ So $N$ can be considered as an admissible divisor in the definition of $\zeta_H$. So if we set $\al' := \al - N$, we have $$\zeta_H\geq \frac{(\al')^2 - \be^2}{2\al'\cdot\be} = \frac{(Z+\xi\be)^2 - \be^2}{2(Z+\xi\be)\cdot\be} = \xi,$$ where we used equation \eqref{eq:dhym-existence-1}. 

Next, suppose we have a strict inequality $\xi<\zeta_H$. Then there exists a $\xi<t<\zeta_H$. Without loss of generality, we may assume that $t<t_0$ ie. $\al - t\be$ is big. Since $t<\zeta_H$, there exists an effective divisor $D$ such that $(\al - D)\cdot\be>0$ and $$t< \frac{(\al- D)^2 - \be^2}{2(\al-D)\cdot\be}.$$ For simplicity, we denote $\al' = \al - D$. By Lemma \ref{lem:bigness-dhym} applied to the classes $\al'$ and $\be$, we see that $\al'  -t\be$ is big. Next, since $t>\xi$ and $(\al-t\be) - (\al' - t\be)$ is pseudoeffective, we also must have that $$\vol(\al'-t\be) - (1+t^2)\be^2 \leq \vol(\al - t\be) - (1+t^2)\be^2<0.$$ But then, 
\begin{align*}
0&> \vol(\al'-t\be) - (1+t^2)\be^2\\
&\geq (\al'-t\be)^2 - (1+t^2)\be^2\\
&= (\al')^2 - \be^2 - 2t\al'\cdot\be
>0,
\end{align*} 
a contradiction. Therefore, we must have $\xi = \zeta_H$.
\end{proof}

We note the following corollary of the main theorem. 
\begin{cor}
There exists an effective $\RR$-divisor $D$ on the compact K\"ahler surface $X$ such that $(\alpha-D)\cdot\beta >0$ and  $$\zeta_H = \frac{(\al - D)^2-\be^2}{2(\al-D)\cdot\be}.$$ 
\end{cor}

We finally prove the following. 

\begin{prop}
$\zeta_H = \cot\theta_{min}$. In particular, conjecture \ref{dhymcon1} holds in dimension two under the assumption that $\al\cdot\be >0$. 
\end{prop}
\begin{proof}
We first observe that if $\al\cdot\be>0$ then the semi-stable admissible set $\mathcal{V}(X,\al,\be)$ is non-empty. In fact, with the divisor $D$ from the above corollary, we see that the triple $(X,\al-D,\be)$ is dHYM-semi-stable. Recall that $\theta_{min}\in (0,\pi)$ is defined by $$\cot\theta_{min} = \sup \left\{ \frac{(\pi^*\al-D)^2 - \pi^*\be^2}{2(\pi^*\al-D)\cdot\pi^*\be}: ~ (\pi, X', D)\in \mathcal{V}(X, \alpha, \beta) \right\}.$$ We now claim that $$\cot\theta_{min} = \zeta_H.$$ Clearly $\cot\theta_{min}\geq \zeta_H$. Suppose we have a strict inequality. Then there exists a modification $\pi:X' \rightarrow X$ and an effective $\RR$-divisor $D$ on $X'$ such that 
\begin{equation}\label{eq:dhym-conj-surf-proof}
\frac{(\al')^2 - (\be')^2}{2\al'\cdot\be'} > \zeta_H,
\end{equation} 
where we let $\al' = \pi^*\al - D$ and $\be' = \pi^*\be$. Let $\ga = \al - \zeta_H\be$ and $\ga' = \al'-\zeta_H\be'$. Then \eqref{eq:dhym-conj-surf-proof} is equivalent to the inequality $$(\ga')^2 > (1+\zeta_H^2)\be^2.$$ But then, since $\pi^*\ga = \ga' +  D$, $$\vol(\ga) \geq \vol(\pi^*\ga) \geq \vol(\ga') \geq (\ga')^2>(1+\zeta_H^2)\be^2,$$ a contradiction.
\end{proof}

\subsection{Some more conjectures and speculations for the $J$-equation}Taking inspiration from the surface case we make the following conjecture. 

\begin{conj}
The pair $(X,\al,\be)$ is semi-stable if and only if $\zeta_{min} = \mu$. 
\end{conj}

We can prove one direction of the above conjecture, at least when $\al$ and $\be$ lie in the rational Neron-Severi group $NS_{\QQ}$. 

\begin{prop}
Let $\al$ and $\be$ be the first Chern classes of ample $\QQ$-line bundles $A$ and $B$ respectively. If $(\al,\be)$ is unstable, then $\zeta_{min} <\mu$. 
\end{prop}
\begin{proof}
Without loss of generality, for instance by taking a sufficiently large power, we may assume that $A$ and $B$ are ample line bundles. By hypothesis, there exists a sub-variety $Z^m\subset X$ such that $$m\int_Z\al^{m-1}\cdot\be > n\mu\int_Z\al^m.$$ Without loss of generality, for instance once again by raising $B$ to a large power, we may assume that the linear system $|B|$ contains a smooth divisor $D$. Let $\sI_Z$ denote the ideal sheaf of $Z$. We will use the following well known formulae to compute the intersection numbers: 
\begin{align*}
\al^n = \lim_{k\rightarrow \infty}   \frac{\dim H^0(X,A^k)}{k^n/n!}\\
\al^{n-1}\cdot\be = \lim_{k\rightarrow\infty}\frac{\dim H^0(D,(A\Big|_D)^k)}{k^{n-1}/(n-1)!}\\
\dim H^0(X,A^k\otimes \sI_Z^{kx}) = a_0(x)k^n + O(k^{n-1}), 
\end{align*}
where $x\in \QQ$ and $$a_0(x) = \frac{\al^n}{n!} - \frac{x^{n-m}}{(n-m)!}\int_Z\frac{\al^m}{m!} + O(x^{n-p+1}),$$  for $x<<1$. Now, let $\pi:X' = \Bl_ZX$ and $\al_\vep' = \pi^*\al - \vep E$, where $E$ is the exceptional divisor. For the rest of the argument we assume that $m<n-1$. If $m=n-1$ then $X' = X$ and the argument is simpler.  For any rational $\vep$ and $\vep<<1$, $A_\vep' = \pi^*A\otimes (-\vep E)$ is an ample $\QQ$ line bundle. Recalling that $H^0(X',(A_\vep')^k) = H^0(X,A^k\otimes \sI_Z^{k\vep})$ and using the above formulae we have 
\begin{align*}
(\al_\vep')^n &= \lim_{k\rightarrow \infty}\frac{\dim H^0(X,A^k\otimes I_Z^{\vep})}{k^{n}/n!}\\
&= \al^n - {n\choose m}\vep^{n-m}\int_Z\al^m +O(\vep^{n-m+1}).
\end{align*}
Next, if $D$ is sufficiently general then $Z\cap D$ is a divisor in $Z$ and once again applying the same formulae, this time to $Z\cap D$ and replacing $m,n$ by $m-1$ and $n-1$ respectively, we see that
\begin{align*}
(\al_\vep')^{n-1}\cdot\be' &=   \lim_{k\rightarrow \infty}\frac{\dim H^0(D,(A\Big|_{D})^k\otimes I_{Z\cap D}^{k\vep})}{k^{n-1}/(n-1)!}\\
&= \al^{n-1}\cdot\be - {n-1\choose m-1}\vep^{n-m}\int_{Z\cap D}\al^{m-1} + O(\vep^{n-m+1})\\
&=  \al^{n-1}\cdot\be - {n-1\choose m-1}\vep^{n-m}\int_{Z}\al^{m-1}\cdot\be + O(\vep^{n-m+1}).
\end{align*}
We then have that 
\begin{align*}
\zeta_{min} &\leq \frac{(\al_\vep')^{n-1}\cdot\be'}{(\al_\vep')^n}\\
&= \frac{ \al^{n-1}\cdot\be - {n-1\choose m-1}\vep^{n-m}\int_{Z}\al^{m-1}\cdot\be + O(\vep^{n-m+1})}{ \al^n - {n\choose m}\vep^{n-m}\int_Z\al^m +O(\vep^{n-m+1})}\\
&= \mu - {n-1\choose m-1}\vep^{n-m}\Big(\int_Z\al^{m-1}\cdot\be  - \mu\frac{n}{m}\int_Z\al^m\Big) + O(\vep^{n-m+1})\\
&<\mu
\end{align*}
if $\vep<<1$.

\end{proof}

Next, again taking inspiration from the surface case, we also expect that  $$\zeta_{min} = \inf\frac{(\al')^{n-1}\cdot\be'}{(\al')^n},$$ where the infimum is now taken over all modifications $\pi:X'\rightarrow X$ and all {\em semi-stable} classes $(\al',\be')$ of the form $\be' = \pi^*\be$ and $\al' = \pi^*\al -D$, where $D$ is an effective divisor such that $\al'$ is big and nef. Note that for $\vep<<1$, $\al - \vep \be$ is K\"ahler, and hence is in particular pseudoeffective. Moreover $(\al' = \al - (\al-\vep\be),\be)$ is obviously stable, and hence the class of the divisors considered above is non-empty. 

Heuristically, suppose the infimum in the definition of the minimal slope is achieved, then arguing as  in the proof of the proposition above one can justify this expectation. Of course, whether the minimal slope is achieved or not seems to be a difficult question in general, and should be intimately connected to the resolution of Conjecture \ref{conj1}.

\section{The $J$-flow with symmetry on projective bundles}\label{sec:j-equation-calabi-symm}\label{sec:j-flow-calabi} We study the convergence of the $J$ flow on projective bundles with Calabi ansatz in this section and in particular prove Theorems \ref{mainthm2} and \ref{thm:j-eq-bubbling}.  The convergence results were first established for the projective bundles $X_{m,n} = \PP(\sO\oplus \sO_{\PP^{n}}(-1))$ on $\PP^n$ in \cite{FL}. As we shall see below, these results extend with no effort to more general projective bundles. Our novelty, in line with the theme of the present work, is to introduce a version of the minimal slope for projective bundles, and to relate the convergence of the flow and the limiting current to this slope. We also take this opportunity to settle some technical difficulties in \cite{FL} and fill in some details. The reader can refer to Remark \ref{rem:fl-remarks} for a detailed comparison.  In the final sub-section, we discuss bubbling phenomenon for the $J$-equation. 

\subsection{The set-up and some preliminary observations} We will briefly review the Calabi ansatz \cite{calabi} following the moment map method of Hwang-Singer \cite{HS} (cf. also \cite{gabor-book} for more details). Let $X = \PP(\sO_M\oplus L^{\oplus (m+1)})\rightarrow M$ be a projective bundle, where $\pi:L\rightarrow M$ is a negative line bundle on a compact K\"ahler manifold $(M^n,\omega_M)$ such that $\omega_M\in c_1(L^{-1})$. We also denote the natural projection map from $X$ to $M$ by $\pi:X\rightarrow M$. There is natural ``divisor at infinity'' given by $D_\infty:= \PP(L^{\oplus(m+1)})$. We let $P_0$ be the projectivization of the zero section $(1,0,\cdots,0)\in \sO\oplus L^{\oplus(m+1)}$. Note that $P_0$ is of dimension $n$. For the sake of notation we set $E = \sO\oplus L^{\oplus(m+1)}$, and $r = m+2$ to be the rank of $E$.  We denote the class of  $D_\infty$ by $\eta$. Then note that $\eta = c_1(\sO_{\PP(E)}(1))$. In particular, the restriction of $\eta$ to each fiber is simply $\sO_{\PP^{r-1}}(1)$. We also think of classes in $H^*(M)$ as classes in $H^*(X)$ via the pullback $\pi^*$, and in fact this gives an injection.

Next, we outline the construction of K\"ahler metrics with Calabi ansatz. We fix a reference K\"ahler form $\omega_M\in c_1(L^{-1})$ and choose a Hermitian metric $h$ on $L$ such that $$\omega_M = -Ric(h) =  \nddbar\log h.$$ Denote by the $\pi:X\rightarrow M$ the natural projection map, and consider K\"ahler metrics $\omega$ on $X$ of the form $$\omega = \pi^*\omega_M + \nddbar u(\rho),$$ where $$\rho = \log|w|^2h(z),$$ $z = (z_1,\cdots,z_n)$ are coordinates on the base $M$ and $w = (w_1,\cdots,w_{m+1})$ are fiber coordinates on $L^{\oplus (m+1)}$.  Working at a point $(z_0,w_0)$, and a triviliazation in which $d\log h(z_0) = 0$, one can compute that $$\omega = (1+u'(\rho)) \pi^*\omega_M + \frac{he^{-\rho}}{2\pi}\Big( u'\delta_{ij} + he^{-\rho}(u'' - u')\bar w_iw_j\Big)\sqrt{-1}dw_i\wedge d\bar w_j.$$ Note that $\rho = \infty$ corresponds to the divisor $D_\infty$, while $\rho=-\infty$ corresponds to $P_0$. It was observed by Calabi \cite{calabi} that $\omega$ extends to a K\"ahler metric on $X$ if and only if the following conditions hold: 
\begin{itemize}
\item $u'>0$ and $u''>0$.
\item $U_{0} = u(\ln r)$ is extendible to a smooth function across $r = 0$ with $U_0'(0) >0$. 
\item $U_\infty(r) = u(-\ln r) + b\ln r$ is extendible to a smooth function across $r = 0$ for some $b>0$ and $U_\infty'(0) >0$. 
\end{itemize}
From the last two properties one can easily infer that $u'$ is a bounded function with $$\lim_{\rho\rightarrow-\infty}u'(\rho) =0,~\lim_{\rho\rightarrow \infty}u'(\rho) = b>0,$$ and the K\"ahler class of $\omega$ equals $\be = [\omega_M] + b\eta.$ A local model for such a function is given by $$u_0(\rho) = b\log(1+e^\rho).$$  We fix one such $\omega$ for the rest of this section. To solve the $J$-equation we seek K\"ahler metrics $\chi \in \al := [\omega_M] + a\eta$ such that $$\chi = \pi^*\omega_M + \nddbar v(\rho),$$ where once again $v$ is monotonically increasing, strictly convex and satisfies the above asymptotics when $\rho\rightarrow \pm\infty$. We call K\"ahler metrics satisfying the above conditions as metrics satisfying the Calabi ansatz. 

Let $\omega$ and $\chi$ be K\"ahler metrics in the classes $\al$ and $\be$ respectively satisfying the Calabi ansatz. Then as in \cite{FL}, we can define a function $\psi:[0,a]\rightarrow[0,b]$ by $\psi(v'(\rho)) = u'(\rho)$. It is then easy to see that the eigenvalues of $\omega$ with respect to $\chi$ are given by $$\frac{1+\psi}{1+x} \text{ (with multiplicity $n$)},~ \frac{\psi}{x} \text{ (with multiplicity $m$)},\text{ and } \psi'\text{ (with multiplicity one)}.$$ Then solving the $J$-equation is equivalent to solving the following first order linear ODE for a monotonically increasing function: 
\begin{equation}\begin{cases}
\psi'(x) + \psi(x)\Big(\frac{n}{1+x} + \frac{m}{x}\Big) = \mu - \frac{n}{1+x}\\
\psi(0) = 0,~\psi(a) = b.
\end{cases}
\end{equation}
More generally for $s\in [0,a)$ we consider the following family of first order ODEs: 
\begin{equation}\label{eq:ode-j-equation-s}\begin{cases}
\psi_s'(x) + \psi_s(x)\Big(\frac{n}{1+x} + \frac{m}{x}\Big) = \mu_s - \frac{n}{1+x}\\
\psi_s(s) = 0,~\psi_s(a) = b.
\end{cases}
\end{equation}
It is easily seen that a general solution to the ODE is given by $$\tpsi_s(x) = \frac{\mu_s I_{m,n,s}(x) - n I_{m,n-1,s}(x) + A_s}{(1+x)^nx^m},$$ where $$I_{m,n,s}(x) = \int_s^x (1+t)^nt^m\,dt.$$ The boundary conditions force $A_s = 0$, and for $s\in [0,a)$, $$\mu_s = \frac{(1+a)^na^mb + nI_{m,n-1,s}(a)}{I_{m,n,s}(a)}.$$ Note that $\mu_s\rightarrow\infty$ as $s\rightarrow a$ from the left. We first begin with a geometric characterization of $\mu_s$. \begin{prop}\label{lem:intersection-number-blow-up}
Let $\tilde\pi:\tilde X = \Bl_{P_0}X\rightarrow X$ be the blow down map with exceptional divisor $E$, and $\tilde\al_s = \tilde\pi^*\al - sE$. Then $$\mu_s = (m+n+1) \frac{(\tilde\al_s)^{m+n}\cdot\tilde\pi^*\be}{(\tilde\al_s)^{m+n+1}}.$$ In particular, 
\begin{enumerate}
\item $\mu_0 = \mu(X,\al,\be)$, and
\item if $\mu\leq n$, then $\la$ is the unique real number in $[0,a)$ (with $\la=0$ when $\mu = n$) such that $$\zeta_{inv} = (m+n+1) \frac{(\tilde\al_\la)^{m+n}\cdot\tilde\pi^*\be}{(\tilde\al_\la)^{m+n+1}}.$$ 
\end{enumerate}
\end{prop}
The proof involves standard intersection theoretic calculations, and we relegate it do the appendix. Next, we have the following elementary observation on the monotonicity of the solutions to the above ODE. As a consequence we obtain an existence result for the $J$ equation on projective bundles. 

\begin{lem}\label{lem:j-equation-calabi-solvability-s}
The following are equivalent. 
\begin{enumerate}
\item $\mu_s(1+s)\geq n$.
\item The solution $\tpsi_s(x)$ to \eqref{eq:ode-j-equation-s} satisfies $\tpsi_s(x)>0$ for all $x\in (s,a]$.
\item The solution $\tpsi_s(x)$ to \eqref{eq:ode-j-equation-s} satisfies $\tpsi_s'(x)>0$ for all $x\in (s,a]$. We also have that $\lim_{x\rightarrow s^+}\tpsi_s'(x) $ exists, and the limit is strictly positive if and only if $\mu_s(1+s) > n.$ 
\end{enumerate}
\end{lem}
\begin{proof}
Suppose $(1)$ holds. From the expression for the general solution we have the identity $$(1+x)^nx^m \psi(x) = \int_s^x (1+t)^{n-1}t^m\Big(\mu_s(1+t) - n\Big)\,dt.$$ From this it clearly follows that $\psi(x) >0$ on $(s,a]$.  Conversely it is easy to see that $(2)\implies (1)$. Next, suppose $(2)$ holds. Then differentiating the above integral formula, and using the product rule we have 
\begin{align*}
(1+x)^{n-1}x^m\Big(\mu_s(1+x) - n\Big) &=\Big[(1+x)^nx^m\psi(x)\Big]' \\
&= n(1+x)^{n-1}x^m\psi(x) + m(1+x)^nx^{m-1}\psi(x) + (1+x)^nx^m\psi'(x).
\end{align*}
Since $\psi(s) = 0$, we then have that $\lim_{x\rightarrow s^+}\psi'(x) \geq 0$ with equality if and only if $\mu_s(1+s) = n.$  Next, differentiating the defining equation for $\psi$ we see that $$\psi''(x) + \Big(\frac{n}{1+x} + \frac{m}{x}\Big)\psi'(x) = \Big(\frac{n}{(1+x)^2} + \frac{m}{x^2}\Big)\psi(x) + \frac{n}{(1+x)^2} > 0.$$ From a standard ODE comparison principle, this implies that $\psi'(x) >0$ for $x\in (s,a]$.

Conversely suppose there is a solution to  \eqref{eq:ode-j-equation-s}  with $\psi'(x)>0$ for all $x\in (s,a]$, then $\psi(x)>0$ for all $x\in (s,a]$, and hence $(2)$ and $(1)$ hold. \end{proof}

\begin{cor}
With $(X,\al,\be)$ as above, the following are equivalent. 
\begin{enumerate}
\item $\mu>n.$
\item There exists a smooth solution $\chi\in \al$ to the $J$-equation $$\La_{\chi}\omega = \mu.$$
\item $(X,\al,\be)$ is $J$-slope stable. 
\end{enumerate}
\end{cor}

Next, we try to analyse the case when $\mu_s(1+s) = n$. To this effect, we let $$\la := \inf\{s\in [0,a)~|~ \text{ there exists a solution to $\eqref{eq:ode-j-equation-s}$ such that $\psi(x)>0$ on $(s,a]$}\}.$$ Then we have the following corollary. 

\begin{cor}\label{cor:mu-s-convexity}
With $\la$ defined as above and $\mu\leq n$, we have
\begin{align*}
\la = \inf\{s\in [0,a)~|~ \mu_s(1+s)>n\}.
\end{align*}
Furthermore, $\la$ is the \underline{unique} solution to the equation $$\mu_\la(1+\la) = n$$ on the interval $[0,a)$ with $\la = 0$ if and only if $\mu = n$. In fact $\mu_s$ is a convex function for $s\in [0,a)$ and $\la$ is it's unique minimizer.
\end{cor}

\begin{proof}
 We only need to prove the last two statements. Consider the polynomial $$p(s) = \Big(\mu_s(1+s) - n\Big)I_{m,n,s}(a).$$ This is degree $m+n+1$ polynomial satisfying 
 \begin{align*}
 p(0) = (\mu - n)I_{m,n}(a)<0\\
 p(a) = (1+a)^{n+1}a^mb>0.
 \end{align*}
 So there exists at least one zero in $(0,a)$. We claim that $p(s)$ is increasing on $(0,a)$, thereby proving the uniqueness of the root. First by differentiating the equation $$(1+a)^na^mb = \mu_sI_{m,n,s}(a) - nI_{m,n-1,s}(a)$$ with respect to $s$, we observe that $$\frac{d\mu_s}{ds} = \frac{\Big(\mu_s(1+s) - n\Big)(1+s)^{n-1}s^m}{I_{m,n,s}(a)}.$$  But then, 
 \begin{align*}
 p'(s) &= \Big(\frac{d\mu_s}{ds}(1+s) + \mu_s\Big)I_{m,n,s}(a) - \Big(\mu_s(1+s) - n\Big)(1+s)^ns^m\\
 &= \mu_sI_{m,n,s}(a)>0,
 \end{align*}
 and hence $p$ is strictly increasing. Note that $s = \la$ is a unique critical point of $\mu_s$.   Convexity of $\mu_s$ follows from the observation that $$\mu_s(1+s) - n = \frac{p(s)}{I_{m,n,s}(a)}$$ is increasing in $s$ and hence so is $\frac{d\mu_s}{ds}$. It then follows that $s= \la$ is the unique minimizer of $\mu_s$. 
\end{proof}

Next, we have the following important consequence.

\begin{cor}\label{cor:j-flow-calabi-ansatz-comparison} Let $\mu\leq n$. Then for all $s\in (0,a)$ with $s\neq \la$ and all $x\in (0,a)$ we have $$\tpsi_s(x) < \tpsi_\la(x).$$  
\end{cor}
\begin{proof}
For each fixed $x\in (0,a)$, we compute 
\begin{align*}
\frac{d}{ds}[x^m(1+x)^n\tilde\psi_s(x)] &= \frac{d\mu_s}{ds}I_{m,n,s}(x) - s^m(1+s)^n\mu_s + ns^m(1+s)^{n-1}\\
&= \Big(\mu_s(1+s) - n\Big)(1+s)^{n-1}s^m\left(\frac{I_{m,n,s}(x)}{I_{m,n,s}(a)} - 1\right),
\end{align*}
which is negative when $s>\la$ and positive when $s<\la$. Here we used the fact that $\mu_s(1+s)>n$ when $s>\la$ and $\mu_s(1+s)<n$ when $s<\la$. Hence, the required conclusion follows immediately.
\end{proof}

Next, we write down the $J$-flow in terms of the moment interval function $\psi$. Let $\chi_{\varphi(t)} =\pi^*\omega_M+ \ddbar v(\rho,t)$, and $\psi(x,t):[0,a]\times[0,\infty)\rightarrow [0,b]$ such that $$\psi(v'(\rho,t),t) = u'(\rho).$$ Then $\chi_{\varphi(t)}$ solving the $J$ flow is equivalent to $\psi$ solving the following parabolic equation: 
\begin{equation}\label{eq:j-flow-calabi}
\frac{\partial \psi}{\partial t} = Q(\psi)\Big(\frac{\d^2\psi}{\d x^2} + \Big(\frac{m}{x}+\frac{n}{1+x}\Big)\frac{\d\psi}{\d x} - \Big(\frac{m}{x^2}+\frac{n}{(1+x)^2}\Big)\psi - \frac{n}{(1+x)^2}\Big),
\end{equation} 
where $Q:[0,b]\rightarrow \RR$ is given by $$Q(y) = u''([u']^{-1}(y)).$$ We denote the initial function by $\psi_0 := \psi(\cdot,0)$.
From now on, we denote the spacial derivatives of $\psi$ by $\psi'$ and the time derivative by $\dot\psi$. We begin the with following basic observation. 
\begin{lem}\label{lem:est-derivative-calabi-ansatz}
There exists a solution to \eqref{eq:j-flow-calabi} for all time $t>0$. Moreover, there exists a constant $C>0$ independent of $t$ such that for all $(x,t)\in [0,a]\times [0,\infty)$, $$0<\psi'(x,t)<C.$$
\end{lem}
\begin{proof} It is well known \cite{xx-chen} that the $J$ flow on {\em any} K\"ahler manifold exists for all time. Differentiating the defining equation for $\psi$, we see that  at a point $x = v'(\rho)$, $$\psi'(x,t)\frac{\partial^2 v}{\partial \rho^2} = \frac{\partial^2 u}{\partial \rho^2}.$$ The positivity of $\psi'$ then follows from the convexity (in the space variable) of $v(\rho,t)$ and $u(\rho)$. Next, it follows from the maximum principle (cf. \cite{xx-chen}) that there exists a constant $C$ such that $$\La_{\chi_{\varphi(t)}}\omega<C. $$  The upper bound then follows from the fact that $$\La_{\chi_{\varphi(t)}}\omega = \psi'(x,t)+ \psi(x,t)\Big(\frac{n}{1+x} + \frac{m}{x}\Big) + \frac{n}{1+x}.$$ Alternately, the bounds can also be proved directly by differentiating \eqref{eq:j-flow-calabi} and an application of maximum principle. 
\end{proof} 

We end this sub-section with comparing and contrasting our method of proof with that in \cite{FL}. 

\begin{rem}[Comparison with the work of Fang-Lai \cite{FL}]\label{rem:fl-remarks} As remarked above, a version of Theorem \ref{mainthm2} is proved in \cite{FL} for the special case when $M = \PP^n$ and $L = \sO_{\PP^n}(-1)$. While our method of proof follows the overall strategy of \cite{FL}, there are some crucial differences. This is necessitated by the fact that some of the arguments in \cite{FL} seem to be incomplete.
There are two main technical issues in \cite{FL} to be addressed, which we will explain and resolve below. Note that the functions between the momentum intervals are labeled $f$ in \cite{FL} instead of the $\psi$ that we use in the present work. 
\begin{itemize}

 \item Some details are missing in the proof of Lemma 2.2 in \cite{FL}. More precisely,  the boundedness of $\varphi - \psi$  does not seem to guarantee uniqueness/convergence.  It is not clear how the uniform limit $f_\infty$ (cf. pg.\ 6551 in \cite{FL}) would satisfy the critical equation.  We take a different approach and settle these issues in Proposition \ref{prop:conv-j-flow-calabi-psi} by studying the convergence of the point-wise slope. 
 
 \medskip

\item  The second-order estimates for $f$ in \cite{FL} seem sketchy. Indeed, in the last line on pg. 6555 in \cite{FL}, it is noted that $\lim_{n\rightarrow\infty}\frac{\partial f}{\partial x}(x_n,t_n) = f'_\infty(x_\infty),$ however  the differentiability of $f_\infty$ is unclear a priori. We address this issue by proving uniform $C^2$-estimates away from the zero section $P_0$. This is achieved by adapting the arguments in the proof of Theorem 1.3 in \cite{SW} and by leveraging the fact that $f_\infty$ (or $\psi_\infty$ as we call it) is a sub-solution away from the zero section.

\end{itemize}


\end{rem}

\subsection{Comparison and monotonicity along the flow} If $\mu>n$, the $J$-equation has a smooth solution and hence by \cite{SW}, the $J$ flow converges smoothly and exponentially fast to this solution. So we only focus on the semi-stable and unstable cases, and assume $\mu\leq n$. It turns out that the behaviour of the flow depends on the sign $\mu - n-m-1$. For the sake of completeness we will state all our results on the behaviour of the flow (ie. monotonicity, comparison principle etc.) for both the cases, but we will provide proofs only for the case when $\mu<n+m+1$ since this subsumes the semi-stable and unstable cases. As in \cite{FL}, we first prove (weak) convergence starting with a very specific initial K\"ahler metric $\chi_0$ so that 
\begin{equation}\label{eq:j-flow-calabi-st-line}
\psi_0(x) = \frac{b}{a}x.
\end{equation} We start with the following elementary observation about the initial function. 

\begin{lem}\label{lem:}
If $\mu\leq n$, then $\tpsi_\la$ is convex. In particular, for all $x\in[\la,a]$, $$\psi_0(x) \geq \tpsi_\la(x).$$
\end{lem}
\begin{proof}
The second part follows easily from the convexity and the fact that the inequality is satisfied at the end points. To prove the convexity we re-write $$\tpsi_{\la}(x) = T_1+T_2,$$ where
\begin{align*}
T_1(x) = \frac{\mu_\la I_{m,n}(x) - nI_{m,n-1}(x)}{x^m(1+x)^n},\\
T_2(x) = \frac{nI_{m,n-1}(\la) - \mu_\la I_{m,n}(\la)}{x^m(1+x)^n}.
\end{align*}
We claim that $T_1$ and $T_2$ are both convex. First note that 
\begin{align*}
nI_{m,n-1}(\la) - \mu_\la I_{m,n}(\la) = \int_0^\la t^m(1+t)^{n-1}\Big(n - \mu_\la(1+t)\Big)\,dt \geq 0,
\end{align*}
since $\mu_\la(1+\la) = n$. On the other hand $x^{-m}(1+x)^{-n}$ is clearly convex, and so $T_2$ is convex. For the convexity of $T_1$, we further re-write it as $$T_1 = x + (\mu_\la -n-m-1)\frac{I_{m,n}(x)}{x^m(1+x)^n},$$ using the elementary identity 
\begin{equation}\label{eq:identity-Imn}
(n+m+1)I_{m,n}(x) = nI_{m,n-1}(x) + x^{m+1}(1+x)^n.
\end{equation}  We now claim that $$q_{m,n}(x):= \frac{I_{m,n}(x)}{x^m(1+x)^n}$$ is concave on $(0,\infty)$. The convexity of $T_1$ then follows immediately since $\mu_\la<n<n+m+1$. The proof of the claim is elementary albeit tedious. We provide details of only the key calculations. The motivated reader can easily fill the gaps. By the binomial expansion we have $$q_{m,n}(x) = \sum_{k=0}^{n}\binom{n}{k}\frac{1}{m+1+k}\frac{x^{k+1}}{(1+x)^n}.$$ Differentiating twice we get $$(1+x)^{n+2}q''_{m,n}(x) =  \sum_{k=0}^{n}\binom{n}{k}\frac{1}{m+1+k}\Big( (k+1-n)(k-n)x^{k+1} + 2(k-n)(k+1)x^k + k(k+1)x^{k-1}\Big).$$ Grouping the terms carefully we have
\begin{align*}
		&(1+x)^{n+2}q''_{m,n}(x) \\ 
				&= \sum_{k=1}^{n-1}\binom{n}{k-1}\frac{(k-n)(k-1-n)x^{k}}{m+k} + \sum_{k=0}^{n-1}\binom{n}{k}\frac{2(k-n)(k+1)x^k}{m+1+k} + \sum_{k=0}^{n-1}\binom{n}{k+1}\frac{(k+1)(k+2)x^{k}}{m+2+k} \\
		&= \sum_{k=1}^{n-1}\Bigg( \binom{n}{k-1}\frac{(k-n)(k-1-n)}{m+k} + \binom{n}{k}\frac{2(k-n)(k+1)}{m+1+k} + \binom{n}{k+1}\frac{(k+1)(k+2)}{m+2+k}\Bigg) x^k \\
		&\hspace*{2cm} + \Big(\frac{-2n}{m+1} + \frac{2n}{m+2}\Big) \\
		&= -2mn \sum_{k=1}^{n-1} \binom{n-1}{k} \frac{x^k}{(m+k)(m+1+k)(m+2+k)} - \frac{2n}{(m+1)(m+2)} \\
		&<0.
	\end{align*}
Therefore, $q_{m,n}$ is a concave function on $(0,\infty)$.
\end{proof}

Next, we show that the inequality above is preserved throughout the flow. 

\begin{lem}\label{lem:jflow-calabi-comparison} Let $\mu\leq n$, and let $\psi(x,t)$ solve the $J$ flow \eqref{eq:j-flow-calabi} with $\psi(x,0) = \psi_0(x)$ given by \eqref{eq:j-flow-calabi-st-line}. Then for all $(x,t)\in [0,a]\times[0,\infty)$ we have $$\psi(x,t)\geq \tilde\psi(x),$$ where $$\tilde\psi(x) = \begin{cases} 0,~x\in [0,\la]\\\tpsi_\la(x),~x\in [\la,a].\end{cases}$$ \end{lem}
\begin{proof}
It is enough to show that $\psi(x,t) \geq \tpsi(x)$ on $[\la,a]$. Consider the function $H:[\la,a]\times[0,\infty)\rightarrow \RR$ defined by $$H(x,t) = \psi(x,t) -\tpsi_\la(x). $$ Then $H$ satisfies the following evolution equation: 
\begin{equation}\label{eq:jflow-calabi-comparison-evolution-H}
\frac{\partial H}{\partial t} = Q(\psi)\Bigg(H''+ \Big(\frac{m}{x}+\frac{n}{1+x}\Big)H' - \Big(\frac{m}{x^2}+\frac{n}{(1+x)^2}\Big)H \Bigg).
\end{equation}Note that we used the fact that $\tpsi_\la$ solves the second order ODE 
\begin{align*}
\tpsi_\la''+ \Big(\frac{m}{x}+\frac{n}{1+x}\Big)\tpsi_\la' - \Big(\frac{m}{x^2}+\frac{n}{(1+x)^2}\Big)\tpsi_\la = \frac{n}{(1+x)^2}.
\end{align*} 
We argue by contradiction. Suppose there exists $T>0$ such that $$ H^*:=H(x^*,t^*) := \inf_{(x,t)\in [0,a]\times [0,T]}H(x,t) <0.$$ Since $$H(\la,t),~H(a,t),~H(x,0)\geq 0,$$ we have that $(x^*,t)\in (\la,a)\times (0,T]$. We then get a contradiction by applying the maximum principle. Indeed, 
\begin{align*}
0\geq \frac{\partial H}{\partial t}(x^*,t^*) \geq -Q(\psi(x^*,t^*)) \Big(\frac{m}{{x^*}^2}+\frac{n}{{(1+x^*)}^2}\Big)H^*>0.
\end{align*}

\end{proof}
\begin{lem}\label{lem:jflow-initial-decreasing}For all $x\in [0,a]$, $$\frac{\partial}{\partial t}\Big|_{t=0}\psi(\cdot,t) \leq 0.$$ 
\end{lem}
\begin{proof}First note that by the identity \eqref{eq:identity-Imn} at $x = a$ and the fact that $\mu\leq n$, we have $$0\geq \Big(\mu - (n+m+1)\Big)I_{m,n}(a) = a^m(1+a)^n(b-a),$$ and so $b<a$. 
Next, using the $J$ flow equation we compute 
\begin{align*}
\frac{\partial}{\partial t}\Big|_{t=0}\psi(\cdot,t) &= Q(\psi_0)\Big(\psi_0'' + \Big(\frac{m}{x}+\frac{n}{1+x}\Big)\psi_0' - \Big(\frac{m}{x^2}+\frac{n}{(1+x)^2}\Big)\psi_0 - \frac{n}{(1+x)^2}\Big)\\
&=Q(\psi_0)\Big(\frac{b}{a}\Big(\frac{m}{x}+\frac{n}{1+x}\Big) - \frac{bx}{a}\Big(\frac{m}{x^2}+\frac{n}{(1+x)^2}\Big)- \frac{n}{(1+x)^2}\Big)\\
&= Q(\psi_0)\frac{(b-a)n}{a(1+x)^2}\\
&\leq 0.
\end{align*}
\end{proof}

Next, we prove that the monotonicity is preserved along the flow. 
\begin{lem}\label{lem:jflow-calabi-mon-t}
Let $\mu\leq n$, and let $\psi(x,t)$ solve the $J$ flow \eqref{eq:j-flow-calabi} with $\psi(x,0) = \psi_0(x)$ given by \eqref{eq:j-flow-calabi-st-line}. Then for all $(x,t)\in [0,a]\times[0,\infty)$ we have $$\frac{\partial\psi}{\partial t}\leq 0.$$
\end{lem}
\begin{proof}
Let $H(x,t) = \dot\psi(x,t)$. Then $H$ satisfies the following evolution equation 
$$
\frac{\partial H}{\partial t}= \frac{Q'(\psi)\dot\psi}{Q(\psi)}H + Q(\psi)\Big(H'' +(\frac{m}{x} + \frac{n}{1+x})H' -(\frac{m}{x^2}+\frac{n}{(1+x)^2})H\Big).
$$ We once again argue by contradiction. Suppose there exists a $T>0$ such that $$\sup_{[0,a]\times [0,T]}H(x,t) >0.$$
By Calabi ansatz, it is not difficult to see that for all $t\geq 0$, there exists a uniform constant $C_T$ such that $$\Big|\frac{\psi'(x,t)}{x} - \frac{\psi(x,t)}{x^2}\Big|\leq C_T.$$We then set 
\begin{align*}
B &= 1+\sup_{[0,a]\times[0,T]}\Big|\frac{Q'(\psi)\dot\psi(x,t)}{Q(\psi)}\Big| \\
&= 1+\sup_{[0,a]\times[0,T]} \Big| Q'(\psi)\left(\frac{\d^2\psi}{\d x^2} + \Big(\frac{m}{x}+\frac{n}{1+x}\Big)\frac{\d\psi}{\d x} - \Big(\frac{m}{x^2}+\frac{n}{(1+x)^2}\Big)\psi - \frac{n}{(1+x)^2}\right)\Big|.
\end{align*}
Now consider the function $F(x,t) = e^{-Bt}H(x,t).$ It satisfies the following evolution equation:
\begin{align*}
\frac{\partial F}{\partial t} &= \Big(-B + \frac{Q'(\psi)\dot\psi}{Q(\psi)}\Big)F +  Q(\psi)\Big(F'' +(\frac{m}{x} + \frac{n}{1+x})F' -(\frac{m}{x^2}+\frac{n}{(1+x)^2})F\Big).
\end{align*}
Let $$F^* = F(x^*,t^*) := \sup_{[0,a]\times[0,T]}F(x,t) >0.$$ As before, clearly $(x^*,t^*) \in (0,a)\times (0,T]$ and so by maximum principle, 
\begin{align*}
0&\leq \frac{\partial F}{\partial t} \\
&\leq  \Big(-B + \frac{Q'(\psi)\dot\psi}{Q(\psi)}\Big)F^*  -Q(\psi)(\frac{m}{x^2}+\frac{n}{(1+x)^2})F^*\\
&\leq -F^* <0,
\end{align*}
a contradiction. 
\end{proof}

\subsection{Proof of convergence}
We first prove a weaker version of the theorem \ref{mainthm2} for the special initial metric corresponding to $\psi_0$ constructed in the previous section. At the level of $\psi$ we have the following basic result. 
\begin{prop}\label{prop:conv-j-flow-calabi-psi}
Let $\mu\leq n$, and let $\psi(x,t)$ solve the $J$ flow \eqref{eq:j-flow-calabi} with $\psi(x,0) = \psi_0(x)$ given by \eqref{eq:j-flow-calabi-st-line}. Then $\psi(x,t)$ converges in $C^1$ to $$\psi_\infty(x) = \begin{cases} 0,~x\in [0,\la]\\\tpsi_\la(x),~x\in [\la,a].\end{cases}$$ 
\end{prop}
\begin{proof} From the $C^1$-estimate in Lemma \ref{lem:est-derivative-calabi-ansatz} and the monotonicity it follows that $$\psi(x,t)\rightarrow\psi_\infty(x)$$ uniformly as $t\rightarrow\infty$, where $\psi_\infty$ is an increasing and continuous  function on $[0,a]$. Since $\psi(x,t)$ is decreasing in $t$, it follows that $\psi_\infty(x)<b$ whenever $x<a$. Moreover,  by Lemma \ref{lem:jflow-calabi-comparison}, we also have that $\psi_\infty(x)\geq \tpsi(x)$ for all $x$.  We now consider the point-wise slope $$\sigma[\psi] := \psi'+ \psi\Big(\frac{n}{1+x} + \frac{m}{x}\Big) + \frac{n}{1+x},$$ which is defined for any differentiable monotonically increasing function $\psi:[0,a]\rightarrow[0,b]$ and set $\sigma(x,t) := \sigma[\psi(x,t)].$  Note that this is simply $\La_{\chi_t}\omega$. The equation for $\psi$ then takes the form $$\frac{\partial\psi}{\partial t} = Q(\psi)\cdot\sigma'(x,t).$$ In particular for each $t$, $\sigma(x,t)$ is decreasing on $[0,a]$. The proof can be broken down into several steps.
\begin{itemize}
\item We define 
\begin{align*}
\la':= \sup\{x~|~ \psi_\infty(x) = 0\}\\
\end{align*}We claim that there exists a sequence of times $t_j\rightarrow \infty$ such that for any compact set $K\subset (\la',a)$, $$\sigma(x,t_j)\rightarrow \sigma_\infty$$ uniformly on $K$ for some constant  $\sigma_\infty$. To see this, fix any $0<\delta<<1$ and $T>0$ and observe that $$\int_{T}^\infty\int_{\la'+\delta}^{a-\delta}\Big|\frac{\partial\psi}{\partial t}\Big|\,dx\,dt = \int_{\la'+\delta}^{a-\delta}\Big(\psi(x,T)-\psi_\infty(x)\Big)\,dx\leq C,$$ for some constant $C$ (independent of $\delta$). Now on $[\la'+\delta,a-\delta]$, $\vep<\psi_\infty < b-\vep$ for some $\vep = \vep(\delta)>0$. In particular there is a constant $c = c(\delta)>0$ such that $Q(\psi(x,t))>c$ for all $x\in [\la'+\delta,a-\delta]$ and  $t>T$. Then $$\int_T^\infty\Big(\sigma(\la'+\delta,t) - \sigma(a-\delta,t)\Big)\,dt = \int_T^\infty\int_{\la'-\delta}^{a-\delta}(-\sigma'(x,t))\,dt \leq C = C(\delta),$$ and so there exists a sequence of times $t_j\rightarrow \infty$ such that $$\lim_{j\rightarrow \infty}\Big(\sigma(\la'+\delta,t_j) - \sigma(a-\delta,t_j)\Big) = 0.$$ Passing to a further sub-sequence we may assume that $\sigma(a-\delta,t_j)\rightarrow \sigma_\infty$. But then by monotonicity, $$\sigma(x,t_j)\rightarrow\sigma_\infty$$ uniformly on $[\la'+\delta,a-\delta].$ Apriori $\sigma_\infty$ may depend on $\delta$, but since the above convergence must hold for all $\delta$ small enough, a simple diagonalization argument proves that $\sigma_\infty$ is indeed independent of $\delta$ and the uniform convergence holds for any compact subset $K\subset(\la',a)$. 
\item We next claim that $\psi_\infty$ is differentiable on $(\la',a)$ and satisfies $$\sigma[\psi_\infty] = \sigma_\infty.$$ To see this, let $x_0,x\in (\la',a)$ and $t_j$ be the sequence in the above claim. Then 
\begin{align*}
\sigma_\infty(x-x_0) &= \lim_{j\rightarrow\infty}\int_{x_0}^x\sigma(y,t_j)\,dy \\
&= \lim_{j\rightarrow\infty}\int_{x_0}^x\Big(\psi'(y,t_j)+ \psi(y,t_j)\Big(\frac{n}{1+y} + \frac{m}{y}\Big) + \frac{n}{1+y}\Big)\,dy\\
&=\psi_\infty(x) - \psi_\infty(x_0) + \int_{x_0}^x \Big(\psi_\infty(y)\Big(\frac{n}{1+y} + \frac{m}{y}\Big) + \frac{n}{1+y}\Big)\,dy.
\end{align*} Rearranging we have the integral formula $$\psi_\infty(x) = \psi_\infty(x_0) + \sigma_\infty(x-x_0) -  \int_{x_0}^x \Big(\psi_\infty(y)\Big(\frac{n}{1+y} + \frac{m}{y}\Big) + \frac{n}{1+y}\Big)\,dy,$$ from which the conclusion follows. 
\item Next, we claim that $\la' = \la$. Consequently, $\sigma_\infty = \mu_\la$ and $\psi_\infty \equiv \tpsi$. Since $\psi_\infty$ solves a $J$-equation type ODE on $[\la',a]$ with boundary conditions $\psi_\infty(\la') = 0$ and $\psi_\infty(a) = b$, it forces $\psi_\infty \equiv\tpsi_{\la'}$ on $[\la',a]$. Furthermore since $\psi_\infty$ is monotonically increasing, by Lemma \ref{lem:j-equation-calabi-solvability-s}, we must have $\mu_{\la'}(1+\la')\geq n$, and hence by Corollary \ref{cor:mu-s-convexity} (or rather it's proof) since $\mu_s(1+s)$ is increasing in $s$, we must have that $\la'\geq\la$. On the other hand, by Lemma \ref{lem:jflow-calabi-comparison} and Corollary \ref{cor:j-flow-calabi-ansatz-comparison}, we have $$\tilde\psi_{\la'}(x) = \psi_\infty(x) \geq \tilde\psi(x)$$ on $[\la',a]$ which is in contradiction to Corollary \ref{cor:j-flow-calabi-ansatz-comparison} if $\la'>\la$. Hence we must have that $\la' = \la$. The second part of the claim follows immediately. To summarize, $\psi(x,t)$ converges uniformly to  $\psi_\infty$ on compact subsets of $(\la,a)$ and $\psi_\infty$ satisfies the following equation on $(\la,a)$: $$\psi_\infty'(x) + (n-1)\frac{\psi_\infty(x)}{x} = \zeta_{inv}.$$ In particular, $$\psi_\infty(x) = \begin{cases} \tpsi_\la(x),~x\geq \la\\0,~x\in [0,\la]. \end{cases}$$
\item For any compact set $K\subset(\la,a)$ we claim that $\psi'(x,t)$ converges uniformly to $\psi_\infty'(x)$, and hence there exists a constant $C = C(K)$ such that $$C^{-1}\leq \frac{\partial \psi}{\partial x} \leq C$$ for all $x\in K$ and $t\geq 0$.
This follows easily from the observation $$\psi'(x,t) = \sigma(x,t)-\psi(x,t)\Big(\frac{n}{1+x} + \frac{m}{x}\Big) + \frac{n}{1+x} \xrightarrow{\text{u.c}} \sigma_\infty - \psi_\infty(x)\Big(\frac{n}{1+x} + \frac{m}{x}\Big) + \frac{n}{1+x} = \psi_\infty'(x).$$
\end{itemize}
\end{proof}
At this point, one can easily obtain higher order estimates, and smooth convergence (away from $P_0$) along the flow for the {\em special initial data}. Instead, we will now obtain higher order estimates and convergence along the flow starting with an {\em arbitrary initial metric} satisfying the Calabi ansatz, thereby completing the proof of Theorem \ref{mainthm2}. To this end we consider the following quantities:
\begin{itemize}
\item Let $f:[0,b]\times[0,\infty)\rightarrow[0,a]$ be defined by $$f(u'(\rho),t) = v'(\rho,t).$$ Equivalently, $f(\cdot,t) = [\psi(\cdot,t)]^{-1}$.
\item Analogously we define $f_\infty:[1,b]\rightarrow[\la,a]$ by $f_\infty = [\psi_\infty\Big|_{[\la,a]}]^{-1}$. Then by the above proposition, $f(x,t)$ converges increasingly to $f_\infty$ and the convergence is uniform on compact subsets of $(0,b)$.
\item The normalized local potential by $$\hat v(\rho,t) = v(\rho,t) - v(0,t) = \int_0^\rho f(u'(r),t)\,dr.$$
\item We similarly define $$ v_\infty(\rho) = \int_0^\rho f_\infty(u'(r))\,dr.$$
\item We define the global potentials by $$\phi(\rho,t) = v(\rho,t) - v_0(\rho),~\phi_\infty(\rho) = v_\infty(\rho) - v_0(\rho),$$ where $v_0(\rho) = v(\rho,0).$
\item Finally we let $\eta:\RR\rightarrow\RR$ be any smooth function such that $$\eta(\rho) = \begin{cases} \la\rho,~\rho \leq -1\\0,~\rho\geq \rho_0>0,\end{cases}$$ where we choose $\rho_0$ large enough so that $v_0-\eta$ is a convex function. 
\end{itemize} 
Recall that since $\ddbar v(\rho,t)$ extends as a K\"ahler metric on $X$, by the Calabi ansatz, $v'(\rho,t)\rightarrow 0$ as $\rho\rightarrow-\infty$. On the other hand, a key point, as we shall show below, is that $v_\infty'(\rho)\rightarrow \la$ as $\rho\rightarrow-\infty$, or equivalently $v_\infty$ is asymptotic to $\la\rho$ or $\eta$ as $\rho\rightarrow-\infty$. Geometrically this means that $\ddbar v_\infty$ extends to a K\"ahler current with bounded potentials on $\tilde X$.
\begin{lem}\label{lem:j-flow-calabi-ansatz-lower-bound} If we let $\hat\phi_\infty:=\phi_\infty - \eta$, then $\hat\phi_\infty$ is uniformly bounded on $(-\infty,\infty)$. \end{lem}
\begin{proof}Note that $v_\infty$ is a convex function with $v_\infty'(\RR) = [\la,a]$. On the other hand the convex function $$\tilde v_0(\rho) = v_0(\rho) + \eta(\rho)$$ also satisfies $\tilde v_0(\RR) = [\la,a]$. But then by standard facts about convex functions, $|v_\infty - \tilde v_0|$ must be a bounded function on $\RR$.

\end{proof}
We immediately have the following corollary.

\begin{cor}
Let $\cS_E$ be the defining section of the exceptional divisor $E$ in the blow up $\tilde\pi:\tilde X = \Bl_{P_0}X\rightarrow X$ and $h_E$ a hermitian metric. Then there exists a constant $C>0$ such that on $x\in X\setminus P_0$ and all $t\geq 0$, $$\la\log \tilde\pi\circ |\cS_E|_{h_E}^2(x)-C\leq \phi(x,t) \leq C. $$  
\end{cor}
The proof of the next Lemma follows immediately from the increasing convergence of $f(x,t)$ to $f_\infty$. 
\begin{lem} There exists a constant $C>0$ such that for all $(\rho,t) \in (-\infty,\infty)\times[0,\infty)$ $$\eta(\rho) - C\leq \phi(\rho,t) - \phi(0,t)\leq \eta(\rho) + C,$$ and $$\lim_{t\rightarrow \infty}\Big|\frac{\partial \phi}{\partial t}(0,t)\Big| = \mu - \sigma_\infty.$$

\end{lem}

\begin{proof}[Proof of Theorem \ref{mainthm2}  for general initial metric with Calabi ansatz]  We let $\phi(\rho,t)$ be the solution to the $J$ flow with the special initial metric, and let $\varphi(\rho,t)$ be the solution to the $J$ flow for an arbitrary initial metric satisfying Calabi ansatz. Then we have $$\frac{\partial\varphi}{\partial t} = \mu - \La_{\chi_\varphi}\omega,$$ and $$\frac{\partial}{\partial t}\phi(\rho,t) = \mu - \La_{\chi_\phi}\omega.$$ A simple application of maximum principle shows that $$||\varphi - \phi||_{C^0(\RR\times[0,\infty))}\leq C.$$ Recall that $v_\infty$ solves the following equation:$$\frac{u''(\rho)}{v_\infty''(\rho)} + n\frac{1+u'(\rho)}{1+v'_\infty(\rho)} + m\frac{u'(\rho)}{v_\infty'(\rho)} = \sigma_\infty.$$ The idea is to use $v_\infty$ to construct a strict sub-solution on $\tilde X$ in the class $\al'_\vep = \tilde\pi^*\al - (\la +\vep)[E] $ for some $0<\vep<<1$. From the above line we get that there exists an $\vep_0>0$ such that 
\begin{align*}
\frac{u''(\rho)}{v_\infty''(\rho)} + (n-1)\frac{1+u'(\rho)}{1+v'_\infty(\rho)} + m\frac{u'(\rho)}{v_\infty'(\rho)} < \sigma_\infty-\vep_0\\
\frac{u''(\rho)}{v_\infty''(\rho)} + n\frac{1+u'(\rho)}{1+v'_\infty(\rho)} + (m-1)\frac{u'(\rho)}{v_\infty'(\rho)} < \sigma_\infty-\vep_0\\
n\frac{1+u'(\rho)}{1+v'_\infty(\rho)} + m\frac{u'(\rho)}{v_\infty'(\rho)} \leq  \sigma_\infty.
\end{align*}
The reason we do not have a strict inequality in the final line is that $\lim_{\rho\rightarrow -\infty}u''(\rho)/v_{\infty}''(\rho) = 0$ as can be seen by direct computation. In order to get a strict inequality, we modify $v_\infty$ in the following way: Let $$V_\infty(\rho) = v_\infty(\rho) + \ep\eta.$$ Then if $\vep$ is small enough we have 
\begin{align*}
\frac{u''(\rho)}{V_\infty''(\rho)} + (n-1)\frac{1+u'(\rho)}{1+V'_\infty(\rho)} + m\frac{u'(\rho)}{V_\infty'(\rho)} < \sigma_\infty-\ep\\
\frac{u''(\rho)}{V_\infty''(\rho)} + n\frac{1+u'(\rho)}{1+V'_\infty(\rho)} + (m-1)\frac{u'(\rho)}{V_\infty'(\rho)} < \sigma_\infty-\ep\\
n\frac{1+u'(\rho)}{1+V'_\infty(\rho)} + m\frac{u'(\rho)}{V_\infty'(\rho)} < \sigma_\infty - \ep^2.
\end{align*}

We now let $$\hat\varphi(\rho,t) = \varphi(\rho,t) - \phi(0,t).$$ Then $$\frac{\partial \hat\vp}{\partial t} = \mu - \La_{\chi_0 + \ddbar\hat\vp }\omega - \frac{\partial \phi}{\partial t}(0,t) = \sigma_\infty + \delta(t) - \La_{\chi_0 + \ddbar\hat\vp }\omega,$$ where $$\delta(t) = (\mu - \sigma_\infty) - \frac{\partial \phi}{\partial t}(0,t)\xrightarrow{t\rightarrow\infty}0.$$ Furthermore, let $\theta = \ddbar \tilde\pi^*V_\infty \in \al'_\vep$ and $$F = \hat\vp - (V_\infty - v_0).$$ Then $F$ is smooth on $ X \setminus P_0$ and tends to $\infty$ along $E$. The key estimate that we need is the following. 

\begin{lem}
Let $\chi_t = \chi_0 + \ddbar\vp$. There exist constants $A,C>0$ such that $$\La_{\omega}\chi_t(\cdot) \leq Ce^{AF(\cdot,t)}$$
\end{lem}
\begin{proof}
We adapt the arguments in \cite{SW}. The evolution equation for $\log \La_{\omega}\chi_t$ is given by $$\Big(\frac{\partial}{\partial t} - \Delta_t\Big)\log \La_{\omega}\chi (\cdot,t)  \leq  -\frac{1}{n\La_\omega\chi}\Big(h^{k\bar l}\tensor{R}{_k_{\bar l}^i^{\bar j}}\chi_{i\bar j} - \chi^{k\bar l}R_{k\bar l}\Big),$$ where $h^{i\bar j} = \chi^{i\bar l}\omega_{k\bar l}\chi^{k\bar j}$, $\Delta_t = n^{-1}h^{i\bar j}\partial_i\partial_{\bar j}$ and $R$ is the curvature tensor of the fixed K\"ahler metric $\omega$. Note that $h$ is not K\"ahler. Next, we also have 
\begin{align*}
\Big(\frac{\partial}{\partial t} - \Delta_t\Big) F &= \sigma_\infty+\delta(t) - \La_{\chi}\omega -  h^{i\bar j}F_{i\bar j}\\
&=\sigma_\infty+\delta(t) - \La_{\chi}\omega - \chi^{i\bar l}\omega_{k\bar l}\chi^{k\bar j}(\chi_{i\bar j} - \theta_{i\bar j})\\
&= \sigma_\infty+\delta(t) -2 \La_{\chi}\omega + \chi^{i\bar l}\omega_{k\bar l}\chi^{k\bar j}\theta_{i\bar j}.
\end{align*}
Let $H = \log\La_\omega\chi - AF$, where $A$ is chosen large enough so that 
\begin{equation}
\frac{1}{An\La_\omega\chi}\Big(h^{k\bar l}\tensor{R}{_k_{\bar l}^i^{\bar j}}\chi_{i\bar j} - \chi^{k\bar l}R_{k\bar l}\Big) < \vep.
\end{equation}
One can do that since $\La_\omega\chi$ is uniformly lower bounded away from zero. 

Let $(x_0,t_0)$ be the maximum of $H$ on $X\times [0,T]$. Let $T_0$ such that for all $t>T_0$, $$|\delta(t)|<\vep.$$ Without loss of generality we may assume that $t_0>T_0$. Since $F$ goes to $\infty$ near $P_0$, clearly $x_0\in X\setminus P_0$. By the maximum principle, 
\begin{align*}
0&\leq \Big(\frac{\partial }{\partial t} - \Delta_t\Big)H(x_0,t_0) \\
&\leq -\frac{1}{n\La_\omega\chi}\Big(h^{k\bar l}\tensor{R}{_k_{\bar l}^i^{\bar j}}\chi_{i\bar j} - \chi^{k\bar l}R_{k\bar l}\Big) - A\Big( \sigma_\infty+\delta(t) -2 \La_{\chi}\omega + \chi^{i\bar l}\omega_{k\bar l}\chi^{k\bar j}\theta_{i\bar j}\Big),
\end{align*}
and so at $(x_0,t_0)$ we have 
\begin{equation}\label{eq:ineuality--slope-proof}
\sigma_\infty -2 \La_{\chi}\omega + h^{i\bar j}\theta_{i\bar j} < 2\vep.
\end{equation}
Choosing $\vep$ small enough, in particular $4\vep<\ep^2$, we also have
\begin{align*}
4\vep< \sigma_\infty - \sup_X\Big(\frac{u''(\rho)}{V_\infty''(\rho)} &+ (n-1)\frac{1+u'(\rho)}{1+V'_\infty(\rho)} + m\frac{u'(\rho)}{V_\infty'(\rho)} ,\\
&~\frac{u''(\rho)}{V_\infty''(\rho)} + n\frac{1+u'(\rho)}{1+V'_\infty(\rho)} + (m-1)\frac{u'(\rho)}{V_\infty'(\rho)},~n\frac{1+u'(\rho)}{1+V'_\infty(\rho)} + m\frac{u'(\rho)}{V_\infty'(\rho)}\Big).
\end{align*}
Now suppose we let $\varphi = \tilde v - v_0$, then the eigenvalues of $h^{i\bar j}$ are given by $$\underbrace{\frac{1+u'}{(1+\tilde v')^2},\cdots,\frac{1+u'}{(1+\tilde v')^2}}_\text{$n$-times},\underbrace{\frac{u'}{(\tilde v')^2},\cdots,\frac{u'}{(\tilde v')^2}}_\text{$m$-times}, \frac{u''}{(\tilde v'')^2}.$$ By \eqref{eq:ineuality--slope-proof}, 
\begin{align*}
2\vep&> \sigma_\infty+ \frac{u''V_\infty''}{(\tilde v'')^2} + n\frac{(1+u')(1+V_\infty')}{(1+\tilde v)^2} + m\frac{u'V_\infty'}{(\tilde v')^2} \\
&- 2\frac{u''}{\tilde v''} - 2n\frac{1+u'}{1+\tilde v'} - 2m\frac{u'}{\tilde v'}\\
&= \sigma_\infty + \frac{u''V_\infty''}{(\tilde v'')^2} - - 2\frac{u''}{\tilde v''} + mu'V_\infty'\Big(\frac{1}{\tilde v'} - \frac{1}{V_\infty'}\Big)^2 - m\frac{u'}{V_\infty'} \\
&+  n(1+u')(1+V_\infty')\Big(\frac{1}{1+\tilde v'} - \frac{1}{1+V_\infty'}\Big)^2 - n\frac{1+u'}{1+V_\infty'} \\
&> \sigma_\infty  - m\frac{u'}{V_\infty'} - n\frac{1+u'}{1+V_\infty'} - 2\frac{u''}{\tilde v''}\\
&> 4\vep  - 2\frac{u''}{\tilde v''}.
\end{align*}
Rearranging the terms differently we also have
\begin{align*}
2\vep&>  \sigma_\infty+ \frac{u''V_\infty''}{(\tilde v'')^2} + n\frac{(1+u')(1+V_\infty')}{(1+\tilde v)^2} + m\frac{u'V_\infty'}{(\tilde v')^2} \\
&- 2\frac{u''}{\tilde v''} - 2n\frac{1+u'}{1+\tilde v'} - 2m\frac{u'}{\tilde v'}\\
&= \sigma_\infty + (n-1)(1+u')(1+V_\infty')\Big(\frac{1}{1+\tilde v'} - \frac{1}{1+V_\infty'}\Big)^2 + u''V_\infty''\Big(\frac{1}{\tilde v'} - \frac{1}{V_\infty''}\Big)^2 + mu'V_\infty'\Big(\frac{1}{\tilde v'} - \frac{1}{V_\infty'}\Big)^2\\
&  + \frac{(1+u')(1+V_\infty')}{(1+\tilde v)^2}- \frac{u''}{V_\infty''} - 2\frac{1+u'}{1+\tilde v'}-(n-1)\frac{1+u'}{1+V_\infty'}- m\frac{u'}{V_\infty'}\\
&> \Big(\sigma_\infty -(n-1)\frac{1+u'}{1+V_\infty'}- m\frac{u'}{V_\infty'}  -  \frac{u''}{V_\infty''}\Big)- 2\frac{1+u'}{1+\tilde v'}\\
&>4\vep - 2\frac{1+u'}{1+\tilde v'}.
\end{align*}
Finally we also have 
\begin{align*}
2\vep&>  \sigma_\infty+ \frac{u''V_\infty''}{(\tilde v'')^2} + n\frac{(1+u')(1+V_\infty')}{(1+\tilde v)^2} + m\frac{u'V_\infty'}{(\tilde v')^2} \\
&- 2\frac{u''}{\tilde v''} - 2n\frac{1+u'}{1+\tilde v'} - 2m\frac{u'}{\tilde v'}\\
&=\sigma_\infty + n(1+u')(1+V_\infty')\Big(\frac{1}{1+\tilde v'} - \frac{1}{1+V_\infty'}\Big)^2 + u''V_\infty''\Big(\frac{1}{\tilde v'} - \frac{1}{V_\infty''}\Big)^2 + (m-1)u'V_\infty'\Big(\frac{1}{\tilde v'} - \frac{1}{V_\infty'}\Big)^2\\
&  + \frac{u'V_\infty'}{(\tilde v)^2}- \frac{u''}{V_\infty''} - 2\frac{u'}{\tilde v'}-n\frac{1+u'}{1+V_\infty'}- (m-1)\frac{u'}{V_\infty'}\\
&> \Big(\sigma_\infty - \frac{u''}{V_\infty''}-n\frac{1+u'}{1+V_\infty'}- (m-1)\frac{u'}{V_\infty'}\Big) - 2\frac{u'}{\tilde v'}\\
&>4\vep -   2\frac{u'}{\tilde v'}.
\end{align*}
Together we obtain that at $(x_0,t_0)$, $$\inf\Big(\frac{u'}{\tilde v'}, \frac{1+u'}{1+\tilde v'}, \frac{u''}{\tilde v''}\Big)>\vep,$$ and hence $$\La_\omega\chi (x_0,t_0) < \frac{m+n+1}{\vep}.$$ For any other $(x,t)\in X\times [0,T]$, 
\begin{align*}
\log\La_\omega\chi(x,t) &\leq \log\La_\omega\chi(x_0,t_0) + A(F(x,t) - F(x_0,t_0)\\
&\leq  C + AF(x,t),
\end{align*}
and so $$\La_\omega\chi(x,t) \leq e^C\cdot e^{AF(x,t)}.$$
\end{proof}

\begin{cor}
For any compact set $K\subset X\setminus P_0$, and any $l\in \NN$, there exists a constant $C_{K,l}$ such that $$||\varphi||_{C^l(K)}\leq C_{K,l}.$$
\end{cor}

\begin{cor}
$\La_{\chi}\omega$ converges to $\sigma_\infty$ smoothly on $X\setminus P_0$ as $t\rightarrow \infty$.
\end{cor}
\begin{proof}
First, recall that the energy $$E(t) = \int_X(\La_{\chi_t}\omega)^2\chi_t^n$$ is decreasing in time.  Moreover, $$\frac{d}{dt}E(t) = -\int_X|\nabla \La_{\chi_t}\omega|_{h(t)}^2\chi_t^n,$$ where $h(t)$ is the Hermitian (not necessarily K\"ahler) form defined by $$h(t)^{i\bar j} = \chi_t^{i\bar l}\omega_{k\bar l}\chi_t^{k\bar j}.$$ Given any $\vep>0$ there exists a $T$ such that for all $T'>T$, 
\begin{equation}\label{eq:j-flow-small-energy-est}
\int_T^{T'}\int_X|\nabla \La_{\chi_t}\omega|_{h(t)}^2\chi_t^n\,dt = E(T) -E(T') <\frac{\vep}{100}.
\end{equation} By the Corollary above and a diagonal argument, given any sequence $t_j\rightarrow \infty$, after passing to a subsequence, $\La_{\chi_{t_j}}\omega$ converges to a smooth function $H$ on $X\setminus P_0$. We claim that $H$ is a constant. If not, then there exists a compact set $K\in X\setminus P_0$ and an $\vep>0$ such that for all $j$ sufficiently large, $$\inf_K|\nabla \La_{\chi_{t_j}}|_{h(t)}^2>\vep.$$ Next, it follows by computing the time-derivative of $|\nabla \La_{\chi_t}\omega|_{h(t)}^2$ and applying the Corollary that $|\nabla \La_{\chi_t}\omega|_{h(t)}^2$ is uniformly Lipschitz in $t$ on $K$, and so there exists a $\delta>0$ such that  $$\inf_{K\times [t_j,t_j+\delta]}|\nabla \La_{\chi_t}\omega|_{h(t)}^2 > \frac{\vep}{2},$$ which contradicts \eqref{eq:j-flow-small-energy-est} above. 

Finally, we need to show that $H\equiv \sigma_\infty$. By Arzela-Ascoli, $\varphi_{t_j}\rightarrow\varphi_\infty$ and $\chi_t \rightarrow \chi_\infty = \chi_0 + \ddbar \varphi_\infty$ on $X\setminus P_0$. Recall that we have $\chi_0 = \ddbar \tilde v_0(\rho)$. We then let $\tilde v_\infty  = \varphi_\infty + \tilde v_0$ so that $\chi_\infty = \ddbar \tilde v_\infty$. Clearly $\tilde v_\infty$ is increasing and convex. Let $\lim_{\rho\rightarrow-\infty}\tilde v_\infty(\rho) = \la'\in [0,a]$. We set $\Psi:[\la',a]\rightarrow [0,b]$ by $$\Psi(\tilde v'(\rho)) = u'(\rho). $$ Then $\Psi$ solves the ODE $$\Psi' + \Big(\frac{m}{x} + \frac{n}{1+x}\Big)\Psi = H - \frac{n}{1+x},$$ and so we must have that $H = \mu_{\la'}$. Since $\Psi$ is increasing we must also have that $\la' \geq \la$. On the other hand, since $||\varphi - \phi||$ remains bounded, there exists an $R>0$ such that for all $\rho<-R$, we have that $$\tilde v_\infty(\rho) =  \varphi_\infty - \tilde v_0 \geq  \la \rho - C.$$ On the other hand for $\rho<-R$ we also have that $\tilde v_\infty(\rho) \leq \la'\rho + C$, and so we must that $\la'\leq \la$, and hence $\la = \la'$ and $H = \sigma_\infty$. By uniqueness of solutions to the ODE, we must also have that $\varphi_\infty$ is independent of the sequence, and hence $$\varphi(\cdot,t)\rightarrow \varphi_\infty$$ smoothly on $X\setminus P_0$ as $t\rightarrow \infty$. 
\end{proof}
 To complete the proof of Theorem \ref{mainthm2}, we finally claim that $\tilde\pi^*\chi_\infty$ extends to a conical metric on $\tilde X$ in the class $\tilde\pi^*\al - \la[E]$ and with cone angle of $\pi$ along $E$. To see this, by an elementary calculation, we first observe that $$\tilde\psi_\la(x) = c(x-\la)^2 + o(|x-\la|^2)$$ near $x = \la$, where $$c = \frac{n}{(1+\la)^2} > 0.$$ From this, it easily follows that $$v_\infty'(\rho) = \la + \frac{1}{\sqrt{c}}e^{\rho/2} + o(e^{\rho/2})$$ as $\rho\rightarrow-\infty.$ The claim then follows from the Lemma below. \end{proof}
We used the following generalization of Calabi's asymptotics, which can be proved easily by adapting Calabi's original arguments.
\begin{lem}
Let $v:\RR\rightarrow\RR$ be a smooth, convex, increasing function such that 
\begin{itemize}
\item There is a smooth function $V_0:\RR\rightarrow\RR$ and $\be\in (0,1]$ such that $$V_0(r) = v(\be^{-1}\ln r) - \la\be^{-1}\ln r,$$ for $r>0$.
\item There is a smooth function $V_\infty:\RR\rightarrow\RR$ such that $$V_\infty(r) = v(-\ln r) + a\ln r,$$ for $r>0$.
\end{itemize}
Then $\chi = \pi^*\omega_M + \nddbar v$ is a smooth K\"ahler metric on $X\setminus P_0$ such that $\pi^*{\chi}$ extends to a conical metric on $\tilde X$ in the class $\tilde\pi^*\al - \la[E]$ with cone angle $2\pi\be$ along the exceptional divisor $E$. In particular, $X\setminus P_0$ is geodesic convex. 
\end{lem}

\subsection{Minimizing sequences for the $L^2$-energy functional}  In this section, we analyze the minimizing sequences for the Donaldson functional. We continue with our notation from the previous section, but for simplicity, assume that $\al = c_1(A)$ and $\be = c_1(B)$ for some ample $\mathbb{Q}$-line bundles on $X$.  Recall that the Donaldson energy on $X$ is defined by $$\sE_\omega(\chi) = \int_X(\La_{\chi}\omega)^2\chi^{m+n+1},$$ where $\chi\in \al$ and $\omega \in \be$. Note that the Donaldson energy makes sense even if $\omega$ is only a semi-positive form. Let $\tilde\pi:\tilde X = \Bl_{P_0}X\rightarrow X$ as before. Recall that the minimal slope for $X$ is then given by $$\zeta_{inv} = (m+n+1)\frac{(\tilde\pi^*\al - \la E)^{n+m-1}\cdot\pi^*\be}{(\tilde\pi^*\al - \la E)^{n+m+1}}$$ for a unique $\la \in (0,a).$ Recall that the deformation to the normal cone of $X$ along $P_0$ is given by $\Bl_{P_0\times\{\infty\}}X\times \PP^1$. We then have a natural projection map $\pi_2:\Bl_{P_0\times\{\infty\}}X\times \PP^1\rightarrow \PP^1$. The ``central fiber" $\pi_2^{-1}(\infty) = \tilde X \cup_E Y$, where $Y$ is in fact a copy of $X$ and $E$ is the exceptional divisor in $\tilde X$ identified with the infinity divisor $D^Y_\infty \subset Y$. 
\begin{prop}\label{prop:lower-bound-j-energy}
 Let $\pi:Y\rightarrow M$ be the projection map via the identification of $Y$ with $X$, and let $\theta\in \pi^*[\omega_M] + \la[D_\infty]$ be {\em any} K\"ahler metric on $Y$ satisfying Calabi ansatz. Let $\omega\in\be$ and $\chi_0\in \al$ satisfy the Calabi ansatz, and let $\chi_t$ be the solution to the $J$ flow (with respect to $\omega$) with initial metric $\chi_0$. Then 
\begin{align*}
\inf_{\chi\in \sK(\al)}\sE_\omega(\chi) &= \lim_{t\rightarrow\infty}\sE_\omega(\chi_t) \\
&= \int_{X\setminus P_0}(\La_{\omega}\chi_\infty)^2\chi_\infty^{m+n+1} + \sE_{\pi^*\omega_{M}}(\theta)\\
&= \zeta_{inv}^2(\tilde\pi^*\al - \la[E])^{m+n+1} +  \sE_{\pi^*\omega_{M}}(\theta).
\end{align*}
\end{prop}
\begin{proof}We let $$e =\zeta_{inv}^2(\tilde\pi^*\al - \la[E])^{m+n+1} +  \sE_{\pi^*\omega_{M}}(\theta).$$ 
\begin{itemize}
\item We first take $\chi_0\in \al$ to be the special K\"ahler metric satisfying Calabi ansatz that was constructed in the previous section, and let $\chi_t$ be the Calabi flow starting with $\chi_0 = \chi$. We claim that $$e = \lim_{t\rightarrow\infty}\sE_\omega(\chi_t).$$Then by the monotonicity of $\sE$ along the $J$ flow, $$\lim_{t\rightarrow\infty}\sE_\omega(\chi_t) \leq \sE_\omega(\chi_0).$$ We now proceed to compute the limit. We let $\psi:[0,a]\times[0,\infty)\rightarrow [0,b]$ as before and such that $\psi(x,0) = \psi_0(x)$ given by the straight line path. For simplicity we also let $$c_{m,n} = (n+m+1){n+m\choose n}d,$$ where $d = \int_M\omega_M^n.$ Then \begin{align*}
\lim_{t\rightarrow\infty}\sE_\omega(\chi_t) &=c_{n,m}\lim_{t\rightarrow\infty}\int_0^a\Big(\psi'(x,t)+ \psi(x,t)\Big(\frac{n}{1+x} + \frac{m}{x}\Big) + \frac{n}{1+x}\Big)^2x^m(1+x)^n\,dx\\
&= c_{n,m}\int_0^a\Big(\psi'_\infty(x)+ \psi_\infty(x)\Big(\frac{n}{1+x} + \frac{m}{x}\Big) + \frac{n}{1+x}\Big)^2x^m(1+x)^n\,dx\\
&= n^2c_{n,m}\int_0^\la x^m(1+x)^{n-2}\,dx + c_{n,m}\zeta_{inv}^2c_{n,m}I_{m,n,\la}(a)\\
&=n^2c_{n,m}\int_0^\la x^m(1+x)^{n-2}\,dx + \zeta_{inv}^2(\tilde\pi^*\al - \la[E])^{m+n+1},
\end{align*}
where we used the identity \eqref{eq:identity-blow-up-1} (modulo multiplication by $c_{m,n}$) from the Appendix for the second term.  Now suppose $$\theta = \pi^*\omega_M + \nddbar f(\rho)$$ is an arbitrary metric on $Y$ with Calabi ansatz in the class $[\pi^*\omega_M] + \la\eta$, then an easy calculation shows that $$\La_{\theta}(\pi^*\omega_M) = \frac{n}{1+f'},~ f':\RR\rightarrow[0,\la]$$ and $$\theta^{n+m+1} = \frac{(n+m+1)!}{n!}(1+f')^n(f')^mf''\pi^*\omega_M^n\wedge \frac{d\rho\wedge d\sigma_{2m+1}}{2\pi^{m+1}}.$$ Integrating and changing coordinates $x = f'(\rho)$ we then have 
\begin{align*}
\sE_{\pi^*\omega_M}(\theta) &= \int_Y [\La_{\theta}(\pi^*\omega_M)]^2\theta^{m+n+1}\\
&= c_{n,m}n^2 \int_0^\la x^m(1+x)^{n-2}\,dx,
\end{align*}
and the claim is proved. Note that we also used the elementary observation that $$c_{n,m} = \frac{(n+m+1)!}{n!}\cdot\frac{|\mathbb{S}^{2m+1}|}{2\pi^{m+1}}\int_M\omega_M^n.$$
\item In \cite{LS}, a lower bound for the infimum of the $L^2$ energy is proved in terms of the Futaki invariant, and an exact identity is proved for the case of $\Bl_{x_0}\PP^n$. One can, as we shall show below, easily adapt the arguments in \cite{LS} to prove the following Atiyah-Bott type identity for more general projective bundles\footnote{In the context of the $J$ equation such an identity is conjectured by Donaldson to hold on general K\"ahler manifolds. This is still an open question.}: \begin{equation}
\inf_{\chi\in \al}||\La_{\chi}\omega-\mu||_{L^2(X,\chi)} =  \sup_{\mathscr{X}}-\frac{\mathrm{Fut_\be}(\mathscr{X})}{||\mathscr{X}||},
\end{equation}
In view of the lower bound, it is enough to show that there is a sequence of test configurations $(\mathscr{X}_j,\mathscr{A}_j)$ such that $$\lim_{t\rightarrow\infty}||\La_{\chi_t} - \mu||_{L^2(X,\chi_t)} = \lim_{j\rightarrow\infty}-\frac{\mathrm{Fut_\be}(\mathscr{X}_j)}{||\mathscr{X}_j||}$$ We recall briefly the definition of the Futaki invariant. Let $(X_t,A_t)$ denote the fibre over $t$, and for a suitably divisible $k$, let $d_k = \dim H^0(X_t,A_t^k)$ (which is independent of $t$ by flatness) and $H_k$ be the $\mathbb{C}^*$ action on $H^0(X_0,A_0^k)$. By the (equivariant) Hirzebruch-Riemann-Roch theorem, we have the following expansions: 
\begin{align*}
d_k&= a_0k^n + O(k^{n-1})\\
\Tr(H_k) &= b_0k^{n+1} + O(k^n)\\
\Tr(H_{k}^2) &= c_0 k^{n+2} + O(k^{n+1}).
\end{align*}
Next, let $D$ be a generic element in the linear system $|B|$. Then $(\scrX,\mathscr{A})$ induces a test configuration of $(D,D\Big|_A)$, and we get corresponding coefficients $a_0'$ and $b_0'$ from the Hirzebruch-Riemann-Roch expansion. Then the Futaki invariant of $\scrX$ and it's norm are defined by $$\Fut_\be(\scrX) = b_0' - \frac{a_0'}{a_0}b_0 = b_0' - \mu b_0,~ ||\scrX||:= c_0 - \frac{b_0^2}{a_0}.$$ Note that the norm is the leading term in the expansion of $\Tr(H_k - \frac{H_k}{d_k}Id)^2$. Now let
\begin{align*}
h &= \psi'_\infty(x)+ \psi_\infty(x)\Big(\frac{n}{1+x} + \frac{m}{x}\Big) + \frac{n}{1+x}-\mu\\
&=\begin{cases} 
\frac{n}{1+x}-\mu,~0\leq x\leq \la\\
\frac{n}{1+\la}-\mu,~\la\leq x\leq a.
\end{cases}
\end{align*}
Note that $h$ is a continuous convex function. We now approximate $h$ by piecewise linear functions $h_j$ which will correspond to test configurations of $(\mathscr{X},\mathscr{A})$. These are precisely the bundle versions of the Donaldson's toric test configurations \cite{Don}, and were first introduced by Szekelyhidi in \cite{Sz-thesis}. As in \cite{LS}, let $h_j$ be a sequence of piecewise linear, continuous, convex functions with rational slopes such that converge uniformly to $h$. We also assume that each $h_j$ is constant in a neighbourhood of $a$. Then each $h_j$  gives rise to a test configuration $(\mathscr{X}_j,\mathscr{A}_j)$ of $(X,A)$. In fact, as observed in \cite{LS} the configuration corresponds to the deformation to the normal cone $\Bl_{Z_k\times\{\infty\}}X\times \PP^1$ of a sub-scheme $Z_k$ supported on $P_0$. The Hamiltonian of the resulting $\CC^*$ action on the central fibre is given by $h_j$. If we denote the corresponding leading terms in the weight expansion by $b_{j,0}$ and $b_{j,0}'$, then as in \cite{LS} we have that 
\begin{align*}
b_{j,0} = \int_0^ah_jx^{m}(1+x)^ndx\\
b_{j,0}'=n\int_0^ah_jx^m(1+x)^{n-1} +h_j(a)a^m(1+a)^nb\\
||\scrX_j|| = \Big(\int_0^a h_j^2x^m(1+x)^n\,ds\Big)^{\frac{1}{2}}.
\end{align*}
We now compute
\begin{align*}
\lim_{t\rightarrow\infty}\int_X(\La_{\chi_t}\omega - \mu)^2\chi_{t}^n &= \int_0^ah\Big(\psi'_\infty(x)+ \psi_\infty(x)\Big(\frac{n}{1+x} + \frac{m}{x}\Big) + \frac{n}{1+x}-\mu\Big)x^{m}(1+x)^ndx\\
&= \int_0^ah(x)\frac{d}{dx}[x^m(1+x)^n\psi_\infty(x)]\,dx + \int_0^ah\Big( \frac{n}{1+x}-\mu\Big)x^{m}(1+x)^ndx\\
&= -\int_0^ah'(x)\psi_\infty(x)x^m(1+x)^n\,dx + h(a)a^m(1+a)^nb \\
&~+ \int_0^ah\Big( \frac{n}{1+x}-\mu\Big)x^{m}(1+x)^ndx\\
&= n\int_0^ahx^m(1+x)^{n-1} +h(a)a^m(1+a)^nb - \mu\int_0^ahx^{m}(1+x)^ndx\\
&=\lim_{j\rightarrow\infty}\Big(n\int_0^ah_jx^m(1+x)^{n-1} +h_j(a)a^m(1+a)^nb - \mu\int_0^ah_jx^{m}(1+x)^ndx\Big)\\
&= -\lim_{j\rightarrow\infty}\Fut_\be(\scrX_j),
\end{align*}
where we integrated by parts in the third line and used the fact that $h'\psi_{\infty}\equiv 0$ in the fourth line. Then 
\begin{align*}
\lim_{t\rightarrow\infty}||\La_{\varphi(t)}\omega - \mu||_{L^2(X,\chi_{\varphi(t)})} &= \frac{\lim_{t\rightarrow\infty}||\La_{\varphi(t)}\omega - \mu||^2_{L^2(X,\chi_{\varphi(t)})}}{\Big(\int_0^ah^2x^m(1+x)^n\,dx\Big)^{\frac{1}{2}}}\\
&=\lim_{k\rightarrow\infty}-\frac{\Fut_\be(\scrX_k)}{||\scrX_k||}.
\end{align*}

\end{itemize}
It now follows easily from the previous two steps that $$e = \inf_{\chi\in\sK(\al)}\sE_\omega(\chi).$$
\end{proof}

\begin{proof}[Proof of Theorem \ref{thm:j-eq-bubbling}] We can finally prove Theorem \ref{thm:j-eq-bubbling}. We first approximate $\chi_\infty$ on $X\setminus P_0$ by smooth metrics $\eta_j$ in the class $(1+\la)[\pi^*\omega_M] + (a-\la)[D_\infty]$. Instead of using Demailly's regularization theorem, we do the smoothening more explicitly. Recall that $\chi_\infty = \pi^*\omega_M+  \nddbar v_\infty$, where $v_\infty$ is a convex function with $v_\infty':\RR\rightarrow (\la,a).$ We let $\tilde v_\infty := v_\infty - \la\rho$ so that $\tilde v_\infty$ is also convex but $\tilde v_\infty':\RR\rightarrow (0,a-\la).$ We also let $\tilde u:\RR\rightarrow \RR$ be given by $$\tilde u (\rho) = (a-\la)\log(1+e^\rho).$$ Then $\tilde u$ is convex with $\tilde u':\RR\rightarrow (0,a-\la)$ and $\eta := \pi^*\omega + \nddbar\tilde u$ is a smooth metric in the class $\pi^*\omega_M + (a-\la)[D_\infty]$. Let $\ka:\RR\rightarrow[0,1]$ be a smooth cut-off function such that $\ka\equiv 1$ on $\rho>1$ and $\ka\equiv 0$ on $\rho<-1$. We first claim that there is a unique solution $\vartheta_k:\RR\rightarrow\RR$  to $$(1+\vartheta_k')^n(\vartheta_k')^m\vartheta_k'' = F_k(\rho):= \ka(\rho+k)(1+\tilde v_\infty')^n(\tilde v_\infty')^m\tilde v_\infty'' + (1-\ka(\rho+k))(1+\tilde u)^n(\tilde u')^m\tilde u''.$$ In fact, suppose $$\vartheta_k(\rho) = \int_{0}^\rho f_k(\tilde u'(s))\,d,$$ then $f_k:(0,a-\la)\rightarrow \RR$ satisfies $$f_k(y)^m(1+f_k(y))^nf_k'(y) = F_k((\tilde u')^{-1}(y)).$$ Integrating we see that $$f_k(y) = I_{m,n}^{-1}\Big(\int_0^y F_k((\tilde u')^{-1}(t))\,dt\Big),$$ where as before $$I_{m,n}(x) = \int_0^xt^m(1+t)^n\,dt.$$ Note that $I_{m,n}$ is strictly increasing and hence invertible. It is easy to see that for each $k$, $\vartheta_k'(\rho) $ is asymptotic to $\tilde u(\rho)$ as $\rho\rightarrow-\infty$, and asymptotic to $\tilde v_\infty$ as $\rho\rightarrow\infty$, and so $$\eta_k = (1+\la)\pi^*\omega + \nddbar \vartheta_k(\rho)$$ defines a smooth metric in the class $(1+\la)[\pi^*\omega_M] + (a-\la)[D_\infty].$  For later use, we also note that there exists a constants $C>0$ such that $$C^{-1}e^{\rho} \leq |\vartheta_k'(\rho)|,|\vartheta_k''(\rho)|\leq Ce^{\rho/2}$$ for all $\rho<-1$ and all $k$. The bound on the right is due to the fact that $\tilde\pi^*\chi_\infty$ extends to a conical metric in the class $\tilde\pi^*\al - \la[E]$ of angle $\pi$ along $E$. 

Next, let $\theta = \pi^*\omega_M + \nddbar \zeta(\rho) \in [\pi^*\omega_M] + \la[D_\infty] $ be a fixed K\"ahler metric with Calabi ansatz, and let $\theta_k = \Phi_k^*\theta = \pi^*\omega_M + \nddbar \zeta_k(\rho)$, where $\zeta_k(\rho) = \zeta(\rho+k)$. Finally we let $$\chi_k = \eta_k+\theta_k -(1+\la) \pi^*\omega_M = \pi^*\omega_M + \nddbar(\vartheta_k  +  \zeta_k) \in \al = [\pi^*\omega_M] + a[D_\infty].$$ We set $v_k = \vartheta_k + \zeta_k.$ For any $N>0$, we let $$U_N = \{\rho <-N\} \text{ and }V_N = \{-N-1<\rho<-N+1\}.$$ We also let $$E_{N,k}  = \int_{-N}^\infty\Big(\frac{u''}{v_k''} + n\frac{1+u'}{1+v_k'} + m\frac{u'}{v_k''}\Big)^2(1+v_k')^n(v_k')^mv_k'' \,d\rho$$ and let $$R_{N,k} =  \int^{-N}_{-\infty}\Big(\frac{u''}{v_k''} + n\frac{1+u'}{1+v_k'} + m\frac{u'}{v_k''}\Big)^2(1+v_k')^n(v_k')^mv_k''\,d\rho, $$ so that $\sE_\omega(\chi_k) = E_{N,k} + R_{N,k}$. As before, we let $$e = \inf_{\chi\in \al}\sE_\omega(\chi) = \sE_{\pi^*\omega_M}(\theta) + \sE_{\omega}(\chi_\infty).$$Notice that for any fixed $N$, $\chi_k$ converges smoothly to $\chi_\infty$ on $X\setminus U_N$. In particular given any $\vep>0$ one can choose $N>>1$ such that for any $k>K_1 = K_1(N,\vep)$ $$|E_{N,k} - \sE_\omega(\chi_\infty)| < \vep.$$ We next analyse $R_{N,k}$. Changing coordinates $\rho\rightarrow \rho+k$ we see that by choosing $N$ sufficiently large, we can ensure that $$\Big|R_{N,k} - n^2\int_{-\infty}^{k-N}\Big(\frac{1}{1+\zeta'}\Big)^2(1+\zeta')^n(\zeta')^m\zeta''\,d\rho\Big|<\vep.$$ We have used the fact that $u',u''$ are asymptotic to $e^{\rho}$. Next, we note that $$\lim_{k\rightarrow\infty}n^2\int_{-\infty}^{k-N}\Big(\frac{1}{1+\zeta'}\Big)^2(1+\zeta')^n(\zeta')^m\zeta''\,d\rho = \sE_{\pi^*\omega_M}(\theta).$$ The upshot is that $$\lim_{k\rightarrow\infty}\sE_\omega(\chi_k) = e.$$ This proves part (3) in the theorem. 

Next we prove the Gromov-Hausdorff convergence. Let $X_\infty = \tilde X\sqcup_{E,D_\infty^Y}Y $ with $Y = X$ endowed with the metric $\theta$. A key point to note is that $$\tilde\pi^*\chi_\infty\Big|_{E} = \theta\Big|_{D_\infty^Y}.$$ This follows by explicit computations in the blow-up coordinates, and relies on the fact that $\lim_{\rho\rightarrow\infty}\zeta'(\rho) = \lim_{\rho\rightarrow-\infty}v_\infty'(\rho) = \la$. Moreover, one can also prove that the distance function $d_E$ on $E$ induced by $\tilde\pi^*\chi_\infty$ (or equivalently by $\theta\Big|_{D_\infty^Y}$) is equal to the restriction of the distance function $d_{\tilde\pi^*\chi_\infty}$ to $E$. That is, $(E,\tilde\pi^*\chi_\infty\Big|_E)$ is not only a Riemannian sub-manifold, but also a metric sub-space. We now define a distance function $d_\infty$ on $X_\infty$ to be $d_\infty\Big|_{\tilde X} = d_{\tilde\pi^*\chi_\infty},~d_\infty\Big|_Y = d_\theta,$ and for $x\in \tilde X\setminus E$ and $y\in Y\setminus D_\infty^Y$, $$d_\infty(x,y) = \inf_{\ga}|\ga|,$$ where the infimum is taken over paths $$\ga = \ga_1 + \ga_2+\ga_3,$$ where each $\ga_1,\ga_2,\ga_3$ are  smooth paths contained completely in $\tilde X\setminus E$, $E = D_\infty^Y$ and $Y\setminus D_\infty^Y$ respectively,  and the lengths $|\ga_j|$ are measured using the metrics induced by $\tilde\pi^*\chi_\infty$, $\tilde\pi^*\chi_\infty\Big|_{E} = \theta\Big|_{D_\infty^Y}$ and  $\theta$ respectively. Note that the resulting distance function $d_\infty$ is the usual distance function defined on a gluing of two metric spaces. In fact, explicitly we have $$d_\infty(x,y) = \begin{cases} d_{\tilde\pi^*\chi_\infty}(x,y),~x,y\in \tilde X\\ d_{\theta}(x,y),~ x,y\in Y\\ \inf_{p\in E \cong D_\infty}\Big(d_{\tilde\pi^*\chi_\infty}(x,p) + d_{\theta}(p,y)\Big),~x\in \tilde X\text{ and }y\in Y.\end{cases}$$

\begin{prop} Let $d_k$ denote the distance function on $X$ induced by $\chi_k$. Then
$$\lim_{k\rightarrow\infty}d_{GH}\Big((X,d_k),(X_\infty,d_\infty)\Big) = 0.$$
\end{prop}

\begin{proof}
We fix sufficiently small $\vep>0$. Our goal is to show that for any sufficiently large $k>k_0$, there are maps (not necessary continuous) $$F_k:X\rightarrow X_\infty,~G_k:X_\infty\rightarrow X$$ such that for all $p,q\in X$ we have 
\begin{align}
|d_k(p,q) - d_\infty(F_k(p),F_k(q))|<12\vep \label{eq:gh-conv-goal1}\\
|d_k(p,G_k\circ F_k (p))<\vep. \label{eq:gh-conv-goal2}
\end{align} The key point in the proof is the observation that the Riemannian metric associated to any K\"ahler form with $U(m+1)$-symmetry, say $\chi = a\pi^*\omega_M + \ddbar v(\rho)$, can be decomposed on $X\setminus P_0 \approx \tilde X\setminus E$ as $$g_\chi = (a+v')g_M + v'(\rho)g_{\CC\PP^m} + v''(\rho)g_{cyl},$$ where 
\begin{itemize}
\item we think of each $\CC\PP^{m+1}$ fibre as an $\mathbb{S}^1$ bundle over $\CC\PP^m$ and $g_{\CC\PP^m}$ denotes the Fubini-Study metric on $\CC\PP^m$. 
\item we let $\tau$ be the contact one form associated to the $\mathbb{S}^1$-bundle and $$g_{cyl} = d\rho^2 + \tau^2.$$
\item $v''(\rho)\rightarrow 0$ as $|\rho|\rightarrow \infty$. Moreover $v'(-\infty) := \lim_{\rho\rightarrow-\infty}v'(\rho) = 0$ for a metric on $X$, and $v'(-\infty) = \la>0$ for a metric on the blow-up.  
\end{itemize} 
For any point $p\in X$ (or $\tilde X)$ we can therefore write $$p = (p_M,p_{\CC\PP^m},\sigma_p,\rho_p),$$ where $\rho_p$ is the (log)-radial coordinate as before, $p_M\in M$, $p_{\CC\PP^m}\in \CC\PP^m$ and $\sigma_p \in \mathbb{S}^1$. Note that as $\rho\rightarrow\infty$, the metric $g$ above is closed to the product metric $g_M + v'(\infty)g_{\CC\PP^m}$ on the infinity section, while if $\varphi'(-\infty) = \la>0$, then for $\rho$ close to $-\infty$, the metric is close to $g_M + \la g_{\CC\PP^m}$ near the zero section, and the metric completion is $\tilde X$.

For any natural number $N$, we set $$X_N = \{\rho>-N\},Y_N=\{\rho<-N\}\subset X.$$ We also let $\tilde X_N $ be the $X_N$ identified as a subset of $\tilde X$. We first specify how to choose $N$ and $k_0$. In what follows, for any K\"ahler metric $\chi$ on $X$ (or on any of the spaces under consideration) we will say that a piecewise smooth curve $\ga:[0,1]\rightarrow X$ is $(\vep,\chi)$-geodesic if $$\sL_{\chi}(\ga) < d_{\chi}(\ga(0),\ga(1)) + \vep.$$ Similarly we will call a subset $A\subset X$, $(\vep,\chi)$-convex if for any two points $p,q\in A$, there exists a $(\vep,\chi)$-geodesic connecting $p$ and $q$ contained completely inside $A$.
\begin{enumerate}
\item Pick $N>>1$ such that 
\begin{itemize}
\item We have $$d_\infty\Big(\partial \tilde X_N,E\Big)<\vep.$$
\item For all $\rho<-N$ and all $k\geq 1$, $$v_k''(\rho) < \vep^2.$$
\end{itemize} 
\item Choose $k_1>>1$ such that for any $k\geq k_1$, 
\begin{itemize}
\item For any piecewise smooth curve $\ga:[0,1]\rightarrow X_N$, $$|\sL_{\chi_k}(\ga) - \sL_{\tilde\chi_\infty}(\ga)|<\vep.$$ In particular, $$d_{GH}\Big((X_N,\chi_k),(\tilde X_N,\tilde\chi_\infty)\Big)<\vep,$$ where we simply the take identity map as the Gromov-Hausdorff approximation. 
\item $|\chi_k-\tilde\chi_\infty|_{C^0(\tilde X_N,\tilde\chi_\infty)}<\vep.$
\end{itemize}
\item \label{item:k2} Finally, choose $k_2>>1$ such that for any $k\geq k_2$, 
\begin{itemize}
\item For any piecewise smooth curve $\ga:[0,1]\rightarrow Y_N$, $$|\sL_{\chi_k}(\ga) - \sL_{\theta_k}(\ga)| < \vep.$$
\item The inclusion map is an $\vep$-Gromov-Hausdorff isometry from $(Y_N,\theta_k)$ to $(Y,\theta).$
\item $d_{\theta_k}(\partial Y_N,D^X_\infty) < \vep$, where $D^X_\infty$ is the divisor at infinity in $X$.
\end{itemize}
\end{enumerate}

We then set $k_0 = \max\{k_1,k_2\}.$ For any point $p\in X\setminus P_0$, we will now write $$p = (p_M,p_{\CC\PP^m},\sigma_p,\rho_p),$$ where $p_M = \pi(p)$, $p_{\CC\PP^m}$ is the corresponding point in the $\CC\PP^m$ in the fibre over $p_M$ (as explained above), and $\rho_p = \rho(p)$ and $\sigma_p$ is the corresponding point in the circle bundle.  For $k\geq k_0$ we construct the map in the following way: 
$$F_k(p) = F_k(p_M,p_{\CC\PP^m},\sigma_p,\rho_p) =  \begin{cases} (p_M,p_{\CC\PP^m},\sigma_p,\rho_p) \in \tilde X,~ \text{ if $p\in X_N$}\\ (p_M,p_{\CC\PP^m},\sigma_p,\rho_p+k)\in Y,~\text{ if $p\in \overline{Y_N}$}.   \end{cases}$$The proof of \eqref{eq:gh-conv-goal1}  has several steps. 

\begin{itemize}
\item We first claim that $Y_N$ is $(\vep,\chi_k)$-convex in $X$ and $Y$ is $(\vep,d_\infty)$-convex in $X_\infty$. It is enough to show that any two points $x,y\in \partial Y_N$ can be joined by an $(\vep,\chi_k)$-geodesic completely contained  in $Y_N$. Writing the metric $g_k$ corresponding to $\chi_k$ as $$g_k = (1+v_k')g_M + v_k'(\rho)g_{\CC\PP^m} + v_k''(\rho)g_{cyl},$$ we see that the metric restricted to $\{\rho = c\}$ (the $\mathbb{S}^1$-bundle over $Z\times \CC\PP^m$) is given by $$(1+v_k'(c))g_M + v_k'(c)g_{\CC\PP^m} + v_k''(c)\tau^2.$$ The metric restricted to the base is $(1+v_k'(c))g_M + v_k'(c)g_{\CC\PP^m}$ which is increasing in $c$. Further in the $\mathbb{S}^1$-direction, the metric is small. In fact, $v_k''(\rho) \leq \vep^2$ on $\rho =- N$ by our choice of $N$. That is, if a minimal geodesic joining $x$ to $y$ leaves $Y_N$, then the length can be made smaller (upto an error $O(\vep)$ coming from the $v_k''$ terms) by projecting the curve to the level set $\partial Y_N$. This proves the first claim. The convexity of $Y$ in $X_\infty$ can also be argued similarly. 
\item We first claim that for any $p,q\in Y_N$, $$|d_{k}(p,q) - d_\infty(F_k(p),F_k(q))| < 2\vep.$$ Since $Y_N$ is $\vep$-convex in $(Y,\theta_k)$, there exists a path $\ga:[0,1]\rightarrow Y_N$ connecting $p$ to $q$ such that $$\sL_{\theta_k}(\ga) \leq d_{\theta_k}(p,q) + \vep.$$ But then by \eqref{item:k2}, we have that $$d_{k}(p,q) \leq d_{\theta_k}(p,q) + 2\vep.$$ On the other hand, since $\chi_k \geq \theta_k$, we can conclude that $$|d_k(p,q) - d_{\theta_k}(p,q)| < 2\vep.$$ On the other hand, since $p,q\in Y_N$, $$d_{\theta_k}(p,q) = d_{\theta}(F_k(p),F_k(q)) = d_\infty(F_k(p),F_k(q)),$$ and the claim is proved. 
\item Next, we claim that if $p\in X_N$ and $q\in Y_N$, then $$|d_{k}(p,q) - d_\infty(F_k(p),F_k(q))| < 10\vep.$$ Let $\ga_k:[0,1]\rightarrow X$ be a $\chi_k$- minimal geodesic connecting $p$ to $q$. Let $$\ga_k(t_1) = p' = (p_M',p_{\CC\PP^m}',\sigma_{p'},-N) \in \partial X_N$$ be the first point of $\ga_k$ on $\partial X_N$ and $$\ga_k(t_2) = q' = (q_M',q_{\CC\PP^m}',\sigma_{q'},-N)\in \partial X_N$$ be the last point of $\ga_k$ on $\partial X_N$. We also let $$p_{\tilde X}' =(p_M',p_{\CC\PP^m}',\sigma_{p'},-N) \in\partial\tilde X_N,~ p_Y' = (p_M',p_{\CC\PP^m}',\sigma_{p'},-N) \in Y$$ and $$q_{\tilde X}' = (q_M',q_{\CC\PP^m}',\sigma_{q'},-N)\in \partial \tilde X_N,~q_Y'=(q_M',q_{\CC\PP^m}',\sigma_{q'},-N)\in Y.$$ We let $\ga_{k,1}$ be the part of $\ga_k$ from $[0,t_1]$ and $\ga_{k,3}$ be the part of $\ga_k$ from $[t_2,1]$. Then 

\begin{eqnarray*}
d_{\chi_k}(p,q)&= & \mathcal{L}_{\chi_k}(\gamma_{k,1})  + \mathcal{L}_{\chi_k}(\gamma_{k,3}) + d_{\chi_k}(p', q') + O(\epsilon) \\
& = &\mathcal{L}_{d_\infty }(F_k(\gamma_{k,1}))  + \mathcal{L}_{d_\infty}(F_k(\gamma_{k,3})) + d_{\theta_k}(p', q') + O(\epsilon) \\
&\geq &d_\infty(F_k(p), p'_{\tilde X} ) + d_\infty(q'_{Y}, F_k(q)) + d_\theta (p'_Y, q'_Y) + O(\epsilon) \\
& = &d_\infty(F_k(p), p'_{\tilde X} ) + d_\infty(q'_{Y}, F_k(q)) + d_\infty (p'_Y, q'_Y) + O(\epsilon) \\
& = &d_\infty(F_k(p), p'_{\tilde X} ) + d_\infty(q'_{Y}, F_k(q)) + d_\infty (p'_Y, q'_Y)+d_\infty (p'_{\tilde X} , p'_{Y}) + O(\epsilon) \\
&\geq & d_\infty(F_k(p), F_k(q))+ O(\epsilon).
\end{eqnarray*}
For the second line we use that $\chi_k$ is close to $\theta_k$ on $\partial X_N$ and the $\vep$-convexity of $Y_N$ while for the penultimate line we use the fact that the radial distance in $d_\infty$ is small.

Now we will prove the other direction. Let $\tilde\gamma_k(t):[0,1]\rightarrow X_\infty$ be an $\epsilon$-minimal geodesic joining $$F_k(p) =\tilde\gamma_k(0), ~F_k(q)=\tilde\gamma_k(1) \in X_\infty.$$
 We define 
 $$p'=(p'_Z, p'_{\mathbb{CP}^m}, \vartheta_{p'}, -N), q'=(q'_Z, q'_{\mathbb{CP}^m}, \vartheta_{q'}, -N) \in \partial Y_N$$ 
 such that 
 $$p'_{\tilde X}=\tilde\gamma_k(t_1)=(p'_Z, p'_{\mathbb{CP}^m}, \sigma_{p'}, -N ) \in \tilde X$$
  is the first point of $\tilde\gamma_k$ in $\{\rho=-N\} \in \tilde X$  and 
  $$q'_Y =\tilde\gamma_k(t_2)=(q'_Z, q'_{\mathbb{CP}^m}, \sigma_{q'}, -N+k) \in Y$$ is the last point at $\{ \rho=-N+k\} \cap Y $. We let $\tilde\gamma_{k, 1}$ be the part of $\tilde\gamma_k$ with $t\in [0, t_1)$ and $\tilde\gamma_{k, 3}$ be the part of $\tilde\gamma_k$ with $t\in (t_2, 1]$.
  
  We also let  
  $$p'_Y=\tilde\gamma_k(t_1)=(p'_Z, p'_{\mathbb{CP}^m}, \sigma_{p'}, -N+k ) \in Y$$
  $$q'_{\tilde X}=\tilde\gamma_k(t_1)=(q'_Z, q'_{\mathbb{CP}^m}, \sigma_{q'}, -N ) \in \tilde X.$$
  Then
\begin{eqnarray*}
d_\infty( F_k(p), F_k(q))&= & \mathcal{L}_{d_\infty}(\tilde\gamma_{k,1})  + \mathcal{L}_{d_\infty}(\tilde\gamma_{k,3}) + d_\infty (p'_{\tilde X} , q'_Y) + O(\epsilon) \\
& = &\mathcal{L}_{\chi_k }( (F_k)^{-1}(\tilde\gamma_{k,1}))  + \mathcal{L}_{\chi_k}( (F_k)^{-1} (\tilde\gamma_{k,3})) + d_\infty(p'_Y, q'_Y) + O(\epsilon) \\
&\geq &d_{\chi_k}(p, p' ) + d_{\chi_k} (q, q') + d_\theta (p'_Y, q'_Y) + O(\epsilon) \\
& = &d_{\chi_k}(p, p' ) + d_{\chi_k} (q, q') + d_{\theta_k} (p', q') + O(\epsilon) \\
& = &d_{\chi_k}(p, p' ) + d_{\chi_k} (q, q') + d_{\chi_k} (p', q') + O(\epsilon) \\
&\geq & d_{\chi_k}(p,q)+ O(\epsilon).
\end{eqnarray*}

We have now proved the claim. 

\item Next, we claim that for any $p, q\in X_N$, 

$$|d_{\chi_k}(p, q) - d_{\infty} (F_k(p), F_k(q))| < \epsilon. $$

Let $\gamma_k:[0,1]\rightarrow X$ be a minimal $\chi_k$-geodesic joining $p=\gamma_k(0)$ and $q=\gamma_k(1)$. 

Let $p'=\gamma_k(t_1)$, $q'=\gamma_k(t_2)$ be the first point and last point of $\gamma_k$ in $\partial Y_N$. Let $\gamma_{k, 1}$ be part of $\gamma_k(t)$ with $t\in [0, t_1)$ and $\gamma_{k, 3}$ be part of $\gamma_k(t)$ with $t\in (t_2, 1]$. 
Then if we let 
$$q'_Y=F_k(q') = (q'_Z, q'_{\mathbb{CP}^m}, \vartheta_{q'}, -N+k) \in Y,~ q'_{\tilde X}  = (q'_Z, q'_{\mathbb{CP}^m}, \vartheta_{q'}, -N) \in \tilde X, $$ we have that
\begin{eqnarray*}
&&d_{\chi_k}(p, q) \\
&=& \mathcal{L}_{\chi_k} (\gamma_{k, 1}) + \mathcal{L}_{\chi_k}(\gamma_{k, 3}) + d_{\chi_k}(p', q') + O(\epsilon)\\
&=& \mathcal{L}_{d_\infty}(F_k(\gamma_{k,1})) + \mathcal{L}_{d_\infty}(F_k(\gamma_{k,3})) + d_\infty(F_k(p'), F_k(q'))+ O(\epsilon)\\ 
&\geq & d_\infty(F_k(p), F_k(p')) + d_\infty (F_k(q), q'_{\tilde X}) + d_\infty(F_k(p'), q'_Y)+ O(\epsilon)\\ 
&\geq&d_\infty(F_k(p), F_k(p')) + d_\infty (F_k(q), q'_{\tilde X}) + d_\infty(F_k(p'), q'_Y)+ d_\infty(q'_Y, q'_{\tilde X}) +  O(\epsilon)\\ 
&\geq & d_\infty(F_k(p), F_k(q)) + O (\epsilon).
\end{eqnarray*}
Here $d_\infty(q'_Y, q'_{\tilde X})$ follows from the fact that $\phi_Y''$ and $\phi''_{\tilde X}$ are $O(\epsilon)$ in the region $\left( {\rho> -N+k} \right)\cap Y \cup \left( {\rho< -N } \right) \cap \tilde X $.

Finally, we will now prove the other direction by bounding $d_\infty(p, q)$ from below.  Let $\gamma_\infty=\gamma_\infty(t)$ with $t\in [0, 1]$ be an $\epsilon$-minimal geodesic joining $p=\gamma_k(0)$ and $q=\gamma_k(1)$. Let $$p'=(p'_Z, p'_{\mathbb{CP}^m}, \vartheta_{p'}, -N), q'=(q'_Z, q'_{\mathbb{CP}^m}, \vartheta_{q'}, -N)\in \partial Y_N$$ be the two points defined as below, and let

$$p'_{\tilde X} =(p'_Z, p'_{\mathbb{CP}^m}, \vartheta_{p'}, -N)=\gamma_\infty(t_1)$$ be the first point of $\gamma_\infty$ in $\{\rho=-N\}\cap \tilde X$. 
$$q'_{\tilde X}=(q'_Z, q'_{\mathbb{CP}^m}, \vartheta_{q'}, -N)= \gamma_\infty (t_2)$$ 
is the last point of $\gamma_\infty$ in $\{ \rho= -N\} \cap \tilde X$. Let $\gamma_{\infty, 1}$ be part of $\gamma_\infty(t)$ with $t\in [0, t_1)$ and $\gamma_{\infty, 3}$ be part of $\gamma_\infty(t)$ with $t\in (t_2, 1]$. 
Then if we let 
$$p'_Y=F_k(p') = (p'_Z, p'_{\mathbb{CP}^m}, \vartheta_{p'}, -N+k) \in Y$$
$$q'_Y=F_k(q') = (q'_Z, q'_{\mathbb{CP}^m}, \vartheta_{q'}, -N+k)  \in Y,$$
we have that
\begin{eqnarray*}
&&d_\infty (F_k(p), F_k(q)) \\
&=& \mathcal{L}_{d_\infty} (\gamma_{\infty, 1}) + \mathcal{L}_{d_\infty}(\gamma_{\infty, 3}) + d_{\infty}(p'_{\tilde X}, q'_{\tilde X}) + O(\epsilon)\\
&=& \mathcal{L}_{\chi_k}((F_k)^{-1}(\gamma_{\infty,1})) + \mathcal{L}_{\chi_k}( (F_k)^{-1}(\gamma_{\infty,3})) + d_\infty(p'_Y, q'_Y))+ O(\epsilon)\\ 
&\geq & d_{\chi_k} (p, p') + d_{\chi_k}  (q, q') + d_{(Y, \theta)} (F_k(p'), F_k(q'))+ O(\epsilon)\\ 
&\geq & d_{\chi_k} (p, p') + d_{\chi_k}  (q, q') + d_{\chi_k} (p', q')+ O(\epsilon)\\ 
&\geq & d_{\chi_k} (p, q) + O (\epsilon).
\end{eqnarray*}

This completes the proof of the claim, and the proposition.

\end{itemize}
\end{proof}

\end{proof}

\section{Cotangent flow with symmetry on $\Bl_{x_0}\PP^2$}\label{sec:dhym-blow-up}
In this final section we study the deformed Hermitian Yang Mills equation on $\PP^n$ blowup at one point using Calabi ansatz.
\subsection{A review of the cotangent flow}
Unlike for the $J$-equation, in the case of the dHYM equations, there are multiple flows that have been studied in literature. Two main examples are those of the line bundle mean curvature flow introduced by Jacob and Yau \cite{JY}, and the tangent flow introduced by Takahashi  \cite{Tak}. In this section we instead review a third flow - the cotangent flow - introduced by Fu-Yau-Zhang \cite{FYZ}.  The cotangent flow is defined as follows:   
\begin{equation}\tag{$\star_t$}\label{eq:flow}
\begin{cases}
\frac{d}{dt}\varphi = \cot\theta_\varphi - \cot\vartheta(X,\al,\be)\\
\varphi(0) = \underline{\varphi}.
\end{cases}
\end{equation}
From a PDE point of view it is important to assume that 
\begin{equation}\tag{B}\label{eq:condition-b}
B_{\underline\varphi} := \max_{X}\theta_{\uvarphi} <\pi,
\end{equation} since this condition is preserved along the flow and makes sure the right hand side is convex. The following theorem is proved in \cite{FYZ}. 
\begin{thm}\label{thm:long-time}
Let $(X^n, \omega)$ be a compact K\"ahler manifold and $\chi$ a smooth closed real $(1,1)$ form in $X$ with $\vartheta(X,\al,\be) \in (0,\pi)$. If the initial data $\uvarphi$ satisfy the condition $(B)$ above, the flow \eqref{eq:flow} has a unique solution on $[0,\infty)$.
\end{thm}

In dimension two one can easily verify that an initial function satisfying condition \eqref{eq:condition-b} can {\em always} be found, no matter what class $\al$ is chosen so long as $\al\cdot\be\neq 0$. 
\begin{lem}\label{lem:dimension-two-trace-condition}
Let $X$ be a compact K\"ahler surface with a K\"ahler class $\beta$ and a real $(1,1)$ class $\alpha$ such that $\alpha\cdot\beta\neq 0$. Let $\omega\in\beta$ be a K\"ahler form and $\chi\in\alpha$ be a real $(1,1)$ form. Then there exists a $\underline{\varphi}\in C^\infty(X,\RR)$ such that $\La_{\omega}\chi_{\uvarphi}$ is a positive constant. In particular, $\uvarphi$ also satisfies $$B_{\uvarphi} <\pi.$$ Moreover, if $\varphi_t$ solves \eqref{eq:flow} with $\varphi_0 = \uvarphi$ on $[0,T)$, then for all $t\in [0,T)$, we have $\La_{\omega}\chi_{\varphi_t} >0$.
\end{lem}
\begin{proof}
Without loss of generality we may assume that $\al\cdot\be>0$. Let $f = \La_{\omega}\chi.$ Then by hypothesis, $$\int_{X}f\omega^2 = 2\al\cdot\be >0.$$  Let $\underline\varphi\in C^\infty(X,\RR)$ solve $$\Delta_\omega\underline{\varphi} = -f + \int_X f\omega^2.$$ Then $$\La_\omega\chi_{\uvarphi} = \La_\omega\chi+\Delta_\omega\underline{\varphi} = \int_X f\omega^2 >0.$$ Now let $\la_1,\la_2$ be eigenvalues of $\omega^{-1}\chi_{\uvarphi}$. Then $\la_1+\la_2>0$. But then since $\ac:(-\infty,\infty)\rightarrow (0,\pi)$ is decreasing, we have $$\ac\la_1 < \ac(- \la_2) = \pi - \ac\la_2,$$ and the result follows from compactness of $X$.  For the third part, let $$t_0 = \sup\{t \geq 0~|~ \min\La_\omega\chi_{\varphi_t}>0\}.$$ Suppose $t_0<T$. Let $\la_{1,t}\geq \la_{2,t}$ be eigenvalues of $\omega^{-1}\chi_{\varphi_t}$. By the definition of $t_0$, there exists a sequence of times $t_k\rightarrow t_0$, $t_k<t_0$ and points $p_k\rightarrow p_0$ such that $$\la_{1,t_k}(p_k) + \la_{2,t_k}(p_k)\rightarrow  0.$$  But then since the flow is smooth in a neighborhood of $t_0$, $\cot{\theta_{\varphi_t}}$ is bounded, and this forces $$\la_{1,t_k}(p_k)\la_{2,t_k}(p_k)\rightarrow 1.$$ But this is a contradiction.  

\end{proof}

\begin{cor}
Let $n=2$ and $\al\cdot\be \neq 0$. Then there exists a $\uvarphi\in C^\infty(X,\RR)$ such that there is a solution to \eqref{eq:flow} for all time.
\end{cor}

In dimension two, we say that a potential $\varphi$ is {\em admissible} if $\La_{\omega}\chi_\varphi >0$, or equivalently $\chi_\varphi\wedge\omega>0$. The proof of the Lemma above actually shows the long time existence of the cotangent flow starting from any initial admissible potential.

\subsection{The set-up and some preliminary observations} The dHYM equations on $\Bl_{x_0}\PP^n$ with Calabi ansatz were first studied \cite{JacSh} where a complete solution to the existence problem is provided when the pair $(X,\al,\be)$ is stable. In this section we restrict our attention to the case when $n=2$ and study the convergence of the cotangent flow. In particular, we prove that the cotangent flow converges to the solutions constructed in Section \ref{subsec:dhym-surface}. We use the notation from Section \ref{sec:j-equation-calabi-symm} in the special case that $M = \PP^1$, $L = \sO_{\PP^1}(-1)$. Then $X = \PP(\sO_{\PP^1}(-1)\oplus \sO) = \Bl_{x_0}\PP^2$. It is now more convenient to use the exceptional divisor $E$ as basis element for $H^{1,1}(X,\RR)$, and we also use the suggestive notation $H = D_\infty$ for the divisor or hyperplane at infinity. By scaling we normalize so that $\be = b[H] - [E]$, and let $\al = p[H] - q[E]$ an arbitrary real class.  Thinking of $\Bl_{x_0}\PP^2$ as a compactification of $\CC^2\setminus\{0\}$ with the usual coordinates $(z_1,z_2)$ we can write any invariant metric $\omega\in \be$ as $$\omega  =  \ddbar u(\rho),$$ where $\rho = \log(|z_1|^2 + |z_2|^2).$ The condition that $\omega$ is a (smooth) K\"ahler metric is that $u',u''>0$ and that $u$ satisfies certain asymptotics as $\rho\rightarrow \pm\infty$. The precise nature of the asymptotics will not be important to us. The condition that $\omega\in \be$ is then equivalent to $u':\RR\rightarrow (1,b)$ with $$\lim_{\rho\rightarrow-\infty}u'(\rho) = 1,~\lim_{\rho\rightarrow\infty}u'(\rho) = b.$$  Similarly we can express any smooth real $(1,1)$ form $\chi\in \al$ as $$\chi = \ddbar v(\rho),$$ and the corresponding boundary asymptotics for $v'(\rho)$ are $$\lim_{\rho\rightarrow-\infty}v'(\rho) = q,~\lim_{\rho\rightarrow\infty}v'(\rho) = p.$$
Since $\al = p[H]- q[E]$ and $\be = b[H]-[E]$ we compute that 
\begin{align*}
\al^2 = p^2 -q^2,\ \ \be^2=b^2-1 \ \ \text{and}\ \ \al\cdot\be=bp-q,
\end{align*}
and so
\begin{align}
\label{c_0 formula involving b,p and q}
c_0 = \cot\vartheta(X,\al,\be) = \frac{p^2 -q^2 -b^2 +1}{2(bp-q)}.
\end{align}
We let $\psi:[1,b]\rightarrow \RR$ such that $$v'(\rho) = \psi(u'(\rho)).$$  It is then easy to see that the eigenvalues of $\omega^{-1}\chi$ are given by $\psi'$ and $\psi/x$ (each with multiplicity one). In particular $$\theta_\chi(x) = \ac\psi'(x) +\ac\frac{\psi(x)}{x}.$$  The dHYM equation is then equivalent to $\theta_\chi' \equiv 0$ or equivalently 
\begin{align}\label{eq:dhym-calabi}
\begin{cases}
(\psi^2 - x^2)' = 2c_0(x\psi)'\\
\psi(1) = q,~\psi(b) = p.
\end{cases}
\end{align} 

As for the $J$-equation we also consider the following family of auxiliary equations for $q\leq s<pb$: 

\begin{align}\label{eq:dhym-calabi-s}
\begin{cases}
(\psi^2 - x^2)' = 2\tilde c_s(x\psi)'\\
\psi(1) = s,~\psi(b) = p.
\end{cases}
\end{align} 

The auxiliary equations are precisely the dHYM equations for the perturbed classes $\tilde\al_s = \al - (s-q)[E] = p[H]-s[E]$.
As in \cite{JacSh}, the above equation integrates to 
\begin{equation}\label{eq:dhym-calabi-s-integrated-solution}
\psi^2 - 2\tilde c_sx\psi - x^2 = \tilde A_s,
\end{equation} where 
\begin{align}\label{eq:cs and As}
\tilde c_s &= \frac{p^2-s^2-b^2+1}{2(pb-s)} = \frac{\tilde\al_s^2-\be^2}{2\tilde\al_s\cdot\be}\nonumber\\
\tilde A_s &= s^2 - 2\tilde c_s s - 1 = p^2 - 2bp\tilde c_s - b^2.
\end{align} 

We begin with the following elementary numerical inequalities. 
\begin{lem}\label{lem:dhym-calabi-elem-inequalities}
Since $s<pb$, we have 
$$ \tilde c_s < \min\left(\frac{p}{b}, \frac{p-s}{b-1}, bp\right). $$
In particular, we have
\begin{enumerate}
\item $(\tilde\al_s-\tilde c_s\be)$ is a big class, and it is K\"ahler (resp. nef) if and only if $s>\tilde c_s$ (resp.\ $\geq $).

\item The triple $(X, \tilde\al_s,\be)$ is dHYM stable (resp. semi-stable) if and only if $s>\tilde c_s$ (resp.\ $\geq$).
\end{enumerate}

\end{lem}
\begin{proof}
Since $s< bp$, we see that $\tilde\alpha_s\cdot\beta >0$. Then applying Lemma \ref{lem:bigness-dhym} to the pair $(\tilde\alpha_s, \beta)$ we see that $(\tilde\al_s-\tilde c_s\be)$ is a big class. In particular, we get
$$ p - \tilde c_s b >0,~~~ p - \tilde c_s b > s - \tilde c_s. $$
Alternatively, using $(\ref{c_0 formula involving b,p and q})$ one can also obtain these inequalities by straightforward computation:
\begin{align*}
(p-\tc_sb)-(s-\tc_s) &= \frac{(b+1)\Big((p-s)^2 + (b-1)^2\Big)}{2(pb-s)}>0,\\
p-\tc_sb &= \frac{b(p^2 +s^2) -2ps +
b(b^2-1)}{2(pb-s)} > \frac{(p -s)^2 + b(b^2-1)}{2(pb-s)}>0,\\
bp - \tilde c_s &= \frac{(p^2+1)(b^2-1) + (bp-s)^2}{2(bp-s)} >0.
\end{align*}
Note that
$$ \tilde\al_s -\tc_s\be = (p-\tc_sb)[H] - (s-\tc_s)[E]. $$ Hence  $(\tilde\al_s-\tc_s\be)$ is a K\"ahler class if and only if $s-\tc_s>0$. Finally note that on a K\"ahler surface dHYM-stability of $(X, \tilde\al_s,\be)$ is equivalent to the K\"ahler-ness of $(\tilde \al_s - \tc_s\be)$.
\end{proof}

\begin{lem}\label{lem:dhym-calabi-solvability} The following are equivalent.
\begin{enumerate}
\item $s\geq \tilde c_s$.
\item There exists a unique solution $\tilde\psi_s$ to \eqref{eq:dhym-calabi-s-integrated-solution} which is continuous on $[1,b]$ and differentiable on $(1,b)$. Explicitly, we have
\begin{equation}\label{eq:family of solutions for star s, positive branch}
\tilde \psi_s(x) = \tilde c_s x + \sqrt{\tilde A_s + (1+ \tilde c_s^2)x^2}.
\end{equation}
Moreover the solution also satisfies $(x\tpsi_s)' > 0$.
\end{enumerate}
\end{lem}
\begin{proof}
Solving the quadratic equation we see that any solution to \eqref{eq:dhym-calabi-s} has to be of the form $$\tpsi_s(x) = \tilde c_s x \pm \sqrt{\tilde A_s + (1+\tilde c_s^2)x^2}.$$ We denote the positive and negative branches by $\tpsi_s^+$ and $\tpsi_s^-$ respectively and claim that the two branches do not intersect on $(1,b]$. To see this consider the quadratic $$p(x) = \tilde A_s + (1+\tc_s^2)x^2.$$ It is enough to show that there is no zero in $(1,b]$. To see this, note that $$p(1) = (s-\tc_s)^2\geq 0,~p(b) = (p-\tc_sb)^2>0.$$ If there is a zero in $(1,b]$, then there exists a point $x_0\in (1,b)$ such that $p'(x_0) = 0$. But this a contradiction since $$p'(x_0) = 2x_0(1+\tc_s^2)>0.$$ 
In particular, if there is a continuous solution to \eqref{eq:dhym-calabi-s-integrated-solution} then it has to be either $\tpsi_s^+$ or $\tpsi_s^-.$ Since $$p = \tpsi_s(b) = \tc_sb+ |p-\tc_sb|,$$ and since $p>\tc_sb$, this forces $\tpsi_s(x) = \tpsi_s^+(x)$. But then $$s = \tpsi_s(1) = \tc_s+|s-\tc_s|,$$ and this forces $s\geq \tc_s$. Conversely, suppose $s\geq\tc_s.$. Then by the same argument as above $\tpsi_s(x) = \tpsi_s^+(x)$ is the unique solution to \eqref{eq:dhym-calabi-s-integrated-solution}.

\end{proof}

Recall that $\xi\in[c_0, t_0)$ is the unique real number solving \eqref{eq:volume equation on Kahler surface, unique solution xi}, and here $t_0 = \min\left(\frac{p}{b},\frac{p-q}{b-1}\right)$, $c_0 = \frac{\al^2-\be^2}{2\al\cdot\be}= \tilde c_q$. We have the following formulae for the volume of big (and K\"ahler) classes on $\PP^n$ blowup at one point: \begin{equation}\label{eq:volume formulae for Kahler and big classes on P^n blown up at one point}
\vol\Big(x[H]-y[E]\Big) = \begin{cases} x^n-y^n ~ \text{ when } x>y>0;\\ x^n ~\text{ when } x>0,~ y\leq0;\\0,~\text{otherwise}. \end{cases}
\end{equation} 
If $\xi < q$, the class $\al-\xi\be$ is also K\"ahler and using \eqref{eq:volume formulae for Kahler and big classes on P^n blown up at one point} we easily see that $\xi= c_0$. While if $\xi\geq q$, the class $\al-\xi\be$ is big but not K\"ahler and again using \eqref{eq:volume formulae for Kahler and big classes on P^n blown up at one point} we easily see that
\[ \xi = bp-\sqrt{(p^2+1)(b^2-1)},\] since $\xi < \frac{p}{b} < bp + \sqrt{(p^2+1)(b^2-1)}$. Also note that for the case $\xi\geq q$ the equation \eqref{eq:volume equation on Kahler surface, unique solution xi} can be re-written as $\xi= \tilde c_{\xi}$.
\begin{lem}
As a function of $s$, $\tilde c_s$ is strictly concave function on $(-\infty,bp)$ and it has unique global maximum at $s=\xi$. In particular, $$ \xi = \sup \Big\{\tilde c_s ~|~ -\infty < s< bp\Big\} = \inf\Big\{s\in(-\infty,bp)~|~ s\geq \tilde c_s\Big\}.$$ Moreover, the condition $q\leq c_0$ is equivalent to $\xi\geq q$.
\end{lem}
\begin{proof}
Differentiating we get
\begin{align*}
\frac{d \tilde c_s}{ds} = \frac{(bp-s)^2 - (p^2+1)(b^2-1)}{2(bp-s)^2}, ~~~
\frac{d^2 \tilde c_s}{ d s^2} = -(p^2+1)(b^2-1)(bp-s)^{-3} <0.
\end{align*}
Thus the function $\tilde c_s$ is a strictly concave function on $(-\infty,bp)$, and $s=\xi$ is the unique critical point in this interval and so it is the global maxima.
Let $f(s) = \tilde c_s -s$. Then the above derivative formula shows that $f'(s)<0$ and hence $f$ is strictly decreasing. In particular, $f(s) > f(\xi)=0$ for all $s<\xi$ and the characterization about $\xi$ are immediate. Since $c_0 = \tilde c_q$, we have $q\leq c_0$ if and only if $f(q)\geq 0$, which is equivalent to $\xi\geq q$ (with equality holds in one if equality holds in other).
\end{proof}

We end the section by a comparison theorem for the family of solutions $\tilde \psi_{s}$. This will be needed in the following sections on convergence of the cotangent flow. 

\begin{lem}\label{lem:dhym-calabi-comparison}
For all $\xi\leq s_1<s_2<pb$, and all $x\in [1,b]$, $$\tilde\psi_{s_2}(x) \geq \tilde\psi_{s_1}(x)$$ with equality if and only if $ x = b$.
\end{lem}
\begin{proof}
Let $s\in[\xi,bp)$, then we have $s\geq \tilde c_s$ and so the solution $\tilde \psi_s$ is given by \eqref{eq:family of solutions for star s, positive branch}. Suppose $x\in(1,b)$. Now differentiating $\tilde\psi_s$ with respect to $s$ we obtain
\begin{align*}
\frac{d \tilde\psi_s}{d s} = \frac{d \tilde c_s}{ds}\left(x - \frac{bp- \tilde c_{s}x^2}{\sqrt{\tilde A_s + (1+ \tilde c_s^2)x^2}}\right).
\end{align*}
From Lemma $\ref{lem:dhym-calabi-elem-inequalities}$ we first observe that $bp - \tilde c_s x^2 >0$ for all $x\in[1,b]$ and for all $s<bp$. To see this note that if $\tilde c_s\leq 0$, then $bp- \tilde c_s x^2 \geq bp- \tilde c_s >0$. On the other hand if $\tilde c_s >0$, then $$ bp -\tilde c_s x^2 \geq bp - \tilde c_s b^2 = b^2\left(\frac{p}{b}- \tilde c_s\right)>0.$$
Next we show that $x\in(1,b)$, we have
$$ g(x) := x-\frac{bp- \tilde c_{s}x^2}{\sqrt{\tilde A_s + (1+ \tilde c_s^2)x^2}} < 0.$$ 
Suppose there exists some $x\in(1,b)$ such that $g(x)\geq 0$. Then
\begin{align*}
x^2\Big(\tilde A_s + (1+ \tilde c_s^2)x^2\Big) \geq  b^2p^2 -2bp \tilde c_{s}x^2 + \tilde c_s^2 x^4.
\end{align*}
Since $\tilde A_s = p^2-b^2-2bp \tilde c_s$, it is easy to see that the above inequality gives $x\geq b$, which is a contradiction. Hence $\frac{d\tilde\psi_s}{d s}>0$ for all $s\in(\xi,bp)$ since $\frac{d\tilde c_s}{ds}<0$ on this interval. In particular, for $\xi\leq s_1 < s_2 < bp$ we have $\tilde \psi_{s_2}(x) >  \tilde \psi_{s_1}(x)$ for all $x\in(1,b)$. On the boundary, we note that $\tilde\psi_{s_1}(b)=\tilde\psi_{s_2}(b)=p$ and $\tilde\psi_{s_1}(1)=s_1 < s_2 = \tilde\psi_{s_2}(1)$, and hence the lemma is proved.
\end{proof}

\subsection{Comparison and monotonicity along the flow} \label{subsec:dhym-key-lem} For the rest of this section, we assume that we are not in the stable case. That is, $q\leq c_0$. Let $\chi_{\varphi_t} = \ddbar v(\rho,t)$ be a path of real $(1,1)$ forms in $\al$, and let $\psi:[1,b]\times[0,\infty)\rightarrow \RR$ be given by $$\psi(u'(\rho),t) = v'(\rho,t),$$ where the prime (at least for the time being) denotes the $\rho$-derivative. Changing variables $x = u'(\rho)$ and letting $Q(x):= u''(u'^{-1}(x))$, $\varphi_t$ satisfies the cotangent flow if and only if $\psi$ satisfies the following equation:

\begin{equation}
\begin{cases}
\frac{\partial\psi}{\partial t}  = Q(x)\frac{\d}{\d x}\cot\theta_t\\
\psi(x,0) = \psi_0(x),
\end{cases}
\end{equation}
for some initial data $\psi_0$, where  $$\theta_t(x):= \theta(x,t) := \ac\frac{\d\psi}{\d x}(x,t) + \ac\frac{\psi(x,t)}{x}.$$ For convenience, we now change our notation and use primes will denote the $x$-derivative. Note that $Q(x) \geq 0$ with equality if and only if $x \in \{1,b\}.$ Expanding out the right hand side, we obtain 
\begin{equation}\label{eq:evolution-eq-dhym-calabi-expanded}
\frac{\partial\psi}{\partial t}  = \frac{Q(x)\csc^2(\theta)}{(x^2+\psi^2)(1+ (\psi')^2)}
\Big((x^2+\psi^2)\psi'' + (1+ (\psi')^2)(x\psi' -\psi)\Big),
\end{equation} and the equation is easily seen to be strictly parabolic away from the boundary points $\{1,b\}$. We choose our initial data to be {\em admissible}, that is it satisfies $$\psi_0' + \frac{\psi_0}{x}> 0,$$ and the boundary conditions $\psi_0(1) = q$ and $\psi_0(b) = p$.

We first prove our main convergence result for a very special choice of initial metric $\chi_0$ such that $\La_{\omega}\chi_0$ is a positive constant, that is $\psi_0:[1,b]\rightarrow\RR$ satisfies $$\begin{cases}\psi_0'(x) + \frac{\psi_0(x)}{x} = 2\mu>0,\\\psi_0(1) = q,~ \psi_0(b) = p. \end{cases}$$ We then must have 
\begin{equation}\label{eq:spl-initial-cotangent}
\psi_0(x) = \frac{\la}{x} + \mu x,
\end{equation} where \[\mu = \frac{bp-q}{b^2-1},~~~\la= \frac{b(bq-p)}{b^2-1}.\]

Our first observation is that the graph of the initial function is below that of $\tilde\psi_\xi$, and that the flow is increasing at $t = 0$. 

\begin{lem}\label{lem:inequalities at time zero for cotangent flow}
Let $q \leq c_0$. For all $x\in [1,b]$ we have 
\begin{enumerate}
\item $\tilde\psi_\xi(x) \geq \psi_0(x)$ with equality if and only if $x = b$.
\item $\frac{\partial \psi}{\partial t}\Big|_{t=0}\geq 0$.
\end{enumerate}
\end{lem}
\begin{proof}
\begin{enumerate}
\item Let's assume $x\in[1,b)$ since at $x=b$ we have $\tilde \psi_\xi(b) = \psi_0(b)=p$. For any $s\in\RR$ define a function $\psi^{(s)}:[1,b]\to\RR$ by
$$\psi^{(s)}(x)=\frac{\la^{(s)}}{x}+\mu^{(s)}x,$$
where
$$\mu^{(s)}=\frac{bp-s}{b^2-1},~~\la^{(s)}= \frac{b(bs-p)}{b^2-1}.$$
When $s=\xi$,
\begin{align*}
\mu^{(\xi)}-\xi &= \frac{bp-\xi}{b^2-1} -\xi =\frac{b(p-b\xi)}{b^2-1}>0,\\
\la^{(\xi)} &= \frac{b(b\xi-p)}{b^2-1} = -(\mu^{(\xi)}-\xi).
\end{align*}
Note that we can re-write $\psi^{(s)}$ as follows:
\begin{align*}
\psi^{(s)}(x) &= \frac{b(bs-p)}{(b^2-1)x} + \frac{(bp-s)x}{b^2-1}\\
&= \frac{s}{(b^2-1)x}(b^2-x^2) + \frac{bp}{(b^2-1)x}(x^2-1).
\end{align*}
Since $x\in[1,b)$, the above expression implies $\psi^{(s)}(x)<\psi^{(\xi)}(x)$ for all $x\in[1,b)$ when $s<\xi$.

Now we show that $\psi^{(\xi)}(x)<\tilde\psi_\xi(x)$ for all $x\in[1,b)$. We compute
\begin{align*}
\psi^{(\xi)}(x) - \tilde\psi_\xi(x) &= (\mu^{(\xi)}-\xi)x + \frac{\la^{(\xi)}}{x} - \sqrt{(1+\xi^2)(x^2-1)}\\
&= \frac{(\mu^{(\xi)}-\xi)(x^2-1)}{x} - \sqrt{(1+\xi^2)(x^2-1)}\\
&= \frac{b(p-b\xi)(x^2-1)}{(b^2-1)x} - \sqrt{(1+\xi^2)(x^2-1)}\\
&= \frac{b(x^2-1)\sqrt{1+\xi^2}}{x\sqrt{b^2-1}} - \sqrt{(1+\xi^2)(x^2-1)}\\
&= \frac{\sqrt{(1+\xi^2)(x^2-1)}}{x\sqrt{b^2-1}}\Big(b\sqrt{x^2-1} - x\sqrt{b^2-1}\Big)<0,
\end{align*}
since $x\in[1,b)$. In second line of the above computation, we used $\la^{(\xi)}=-(\mu^{(\xi)}-\xi) $. In the fourth line, we used $(p-b\xi)^2 =(1+\xi^2)(b^2-1)$ which is the equation $\xi$ satisfies. To complete the proof we note that $\psi^{(s)} \equiv \psi_0$ when $s = q$, and that $\xi\geq q$.
\item At $t=0$, $\theta_0(x)=\ac\psi_0^\prime + \ac\frac{\psi_0}{x}$ and then
\begin{align*}
\cot\theta_0 = \frac{\psi_0^\prime\cdot\frac{\psi_0}{x}-1}{\psi_0^\prime + \frac{\psi_0}{x}} = \frac{1}{2\mu}\left(\left(\frac{\la}{x^2}+\mu\right)\left(-\frac{\la}{x^2} + \mu\right) -1\right) =\frac{1}{2\mu}\left(\mu^2 -1 -\frac{\la^2}{x^4}\right). 
\end{align*}
So differentiating with respect to $x$ we obtain $$(\cot\theta_0)^\prime(x,0) = \frac{2\la^2}{\mu x^5}\geq 0,$$ and hence $$\frac{\d \psi}{\d t}(x,0)\geq0.$$ 
\end{enumerate}
\end{proof}

Next, we prove that the two properties above propagate along the flow using the maximum principle.

\begin{lem}\label{lem:comparison-to-limit} Let $q\leq c_0$. For all $t\in [0,\infty)$ and $x\in [1,b]$, $$\psi(x,t) \leq \tilde\psi_\xi(x).$$
	
\end{lem}

\begin{proof}
In view of Lemma \ref{lem:dhym-calabi-comparison},
it is enough to show that $\psi(x,t) \leq \tilde\psi_s(x)$ for all $s> \xi$. We let $$H_s(x,t) := \psi(x,t) - \tilde\psi_s(x).$$ The evolution equation for $H_s$ is given by

\begin{align*}
\frac{\d H_s}{\d t} = \frac{Q\csc^2(\theta)}{(x^2+\psi^2)(1+ \psi'^2)}
\Bigg( (x^2+\psi^2)H''_{s} &+ (1+ \psi'^2)(xH'_{s}-H_s)
+ (H_{s}^2+2H_{s}\tilde\psi_s)\tilde\psi''_s \\ &+ \big((H'_{s})^2 + 2H'_{s}\tilde\psi'_s \big)(x\tilde\psi'_s -\tilde\psi_s)\Bigg),
\end{align*}
where we used the fact that $\tilde\psi_s$ satisfies the following second order ODE:
\begin{equation}
\label{equ3}
(x^2+\tilde\psi_{s}^2)\tilde\psi''_s(x) + (1+ (\tilde\psi'_s)^2)(x\tilde\psi'_s -\tilde\psi_s)=0,
\end{equation}
on $(1, b]\times [0, \infty)$. At the spatial and time boundaries we observe that, 
$$ H_s(x, 0) < 0, ~\forall x\in (1, b),~H_s(1,t)<0,~\forall t\in [0,\infty),~H_s(b,t) = 0,~\forall t\in [0,\infty),$$
where the inequalities follows from Lemma \ref{lem:inequalities at time zero for cotangent flow} and Lemma \ref{lem:dhym-calabi-comparison}.
	
Arguing by contradiction, suppose there exists a $T>0$ such that $$H^*:=\underset{[1,b]\times[0,T]}{\sup} H_s(x,t) >0.$$ Note that $$a(x,t):= \frac{Q \csc^2 \theta}{(x^2 +\psi^2)(1+\psi'^2)}$$ is uniformly bounded above as the angle $\theta(x,t)$ lies in a compact set of $(0, \pi)$ and $Q=u''\circ u'^{-1}$. 
We also have that $||\tilde\psi_s||_{C^{2}[1,b]}$ is finite for $s>\xi$ (Note that this fails at $s = \xi$). We define a constant $A=A(T)>0$ by
$$ A := 1+ \lVert \tilde\psi_s\rVert_{C^2[1,b]} + \underset{[1,b]\times[0,T]}{\sup} a(x,t).$$ We let $B= 1+A^2H^* + 2A^3$ and consider the function $$F(x,t):= e^{-Bt}H_s(x,t).$$ First we see that
$$ F^* := F(x^*,t^*) = \underset{[1,b]\times[0,T]}{\sup} F(x,t) >0$$ for some $(x^*,t^*)\in[1,b]\times[0,T]$. Next, the spacial and time boundary conditions implies that $x^*\in(1,b)$ and $t^*\neq 0$. It implies that $$\frac{\d F}{\d t}(x^*,t^*)\geq0,~~F'(x^*,t^*)=0,~~F''(x^*,t^*) \leq 0.$$ Then by the maximum principle,
\begin{align*}
0 &\leq \frac{\d F}{\d t}(x^*,t^*) \\
&= -BF(x^*,t^*)+ a(x^*,t^*)\left((x^2+\psi^2)F'' + (1+ \psi'^2)(xF'-F)
+ (FH_s+2F\tilde\psi_s)\tilde\psi''_s \right)(x^*, t^*) \\ 
&\hspace*{2cm} +a(x^*,t^*)(F'^2 H' + 2F'\tilde\psi'_s )(x\tilde\psi'_s -\tilde\psi_s)(x^*, t^*)\\
&\leq -BF^* + a(x^*,t^*)\Big(H^*+2\tilde\psi_s(x^*)\Big)\tilde\psi''_s(x^*)F^* \\
& \leq -BF^* +(A^2H^* + 2A^3)F^* \\
&= -F^*,
\end{align*}
which is a contradiction. Hence, we must have $$\underset{[1,b]\times[0,\infty)}{\sup} H_s(x,t)\leq 0,$$ and so we get the desired result.
\end{proof}

Next, we prove that the two properties above propagate along the flow using the maximum principle.  

\begin{lem}\label{lem:dhym-calabi-monotonicity} 
Let $q\leq c_0$. For all $t\in [0,\infty)$ and all $x\in [1,b]$, $$\frac{\partial \psi}{\partial t}(x,t) \geq0.$$
\end{lem}
\begin{proof}
At the boundary points, since $Q(1) = Q(b) = 0$, we have $$\frac{\partial \psi}{\partial t}(1,t) = \frac{\partial \psi}{\partial t}(b,t) = 0$$ for all $t$. The evolution equation for $\psi$ is given by 
\begin{equation}
\frac{\d \psi}{\d t} =
Q(x)\csc^2\left(\ac\frac{\psi}{x} + \ac\frac{\d \psi}{\d x}\right)\left(\frac{\frac{\d^2 \psi}{\d x^2}}{1+ (\frac{\d \psi}{\d x})^2} + \frac{\frac{1}{x}{\frac{\d \psi}{\d x}} - \frac{\psi}{x^2}}{1 + (\frac{\psi}{x})^2}\right).
\end{equation}
Taking one time derivative we see that $\dot\psi$ satisfies the following equation: $$\frac{\partial \dot\psi}{\partial t} = L(\dot\psi) + a\dot\psi,$$ where $L(f) := {\bf A}(x,t)f'' + {\bf B}(x,t)f'$, $$\begin{cases}\mathbf{A}(x,t) &:= \frac{Q(x)\csc^2(\theta)}{1+\psi'^2},\\
\mathbf{B}(x,t) &:= Q(x)\csc^2(\theta)\Bigg(\frac{-2x\psi''}{(x\psi'+\psi)(1+ \psi'^2)} + \frac{2(\psi\psi'-x)(x\psi'-\psi)}{(x\psi'+\psi)(1+ \psi'^2)(x^2+\psi^2)} + \frac{x}{x^2+\psi^2}\Bigg),\\
a(x,t) &:= Q(x)\csc^2(\theta)\Bigg(-\frac{4x\psi'}{(x\psi'+\psi)(x^2 + \psi^2)} + \frac{2x\psi''(\psi\psi'-x)}{(x\psi'+\psi)(1+ \psi'^2)(x^2+\psi^2)}+\frac{1}{x^2+\psi^2}\Bigg).\end{cases}$$
Arguing by contradiction, suppose there exists a $T>0$ such that $$\inf_{[1,b]\times[0,T]}\dot\psi(x,t)<0.$$ Let $C = C(T)>0$ such that $$C = 1+\sup_{[1,b]\times[0,T]}|a(x,t)|.$$ Note that the supremum is finite since $\theta$ is strictly between $(0,\pi)$ and so $\csc^2\theta$ is bounded above and below, and $x\psi' + \psi$ is bounded below. We now let $$F(x,t) = e^{-Ct}\dot\psi(x,t),$$ and let $(x^*,t^*) \in [1,b]\times[0,T]$ such that $$F^*:= F(x^*,t^*) := \inf_{[1,b]\times[0,T]}e^{-Ct}\dot\psi(x,t) <0.$$ Then by the boundary conditions, $x^*\in (1,b)$ and $t^*\neq 0$. By the maximum principle, 
\begin{align*}
0\geq \frac{\partial F}{\partial t}(x^*,t^*) &= -CF^* + L(F)(x^*,t^*) + a(x^*,t^*)F^*\\
&\geq (a(x^*,t^*)-C)F^*.
\end{align*}
Now $a(x^*,t^*) - C\leq -1$ and $F^*<0$ and so we have a contradiction. 
\end{proof}

\subsection{Proof of Theorem \ref{thm:conv-cotangent-flow-blow-up}} We are now in a position to prove Theorem \ref{thm:conv-cotangent-flow-blow-up}. As in Section \ref{sec:j-flow-calabi} we first have the following analog of Proposition \ref{prop:conv-j-flow-calabi-psi}.

\begin{prop}\label{prop:conv-cotangent-flow-calabi-psi}
Let $q\leq c_0$, and let $\psi(x,t)$ solve the cotangent flow \eqref{eq:evolution-eq-dhym-calabi-expanded}with $\psi(x,0) = \psi_0(x)$ given by \eqref{eq:spl-initial-cotangent}. Then $\psi(x,t)$ converges uniformly, and in $C^1_{loc}$ on $(1,b]$, to $\tilde\psi_\xi$. 
\end{prop}
\begin{proof}
By the comparison and monotonicity we immediately have that the point-wise limit $$\psi_\infty(x) = \lim_{t\rightarrow\infty}\psi(x,t)$$ exists, is a lower semi-continuous function on $[1,b]$ and satisfies $\psi_\infty(x)\leq \tilde\psi_\xi(x)$ for all $x\in [1,b]$. Moreover $x\psi_\infty(x)$ is increasing on $[1,b]$.  Note that $\psi_\infty(b-) = p$. On the other hand, we let $$s := \lim_{x\rightarrow 1^+}\psi_\infty(x).$$
The proof now has several steps. 
\begin{itemize}
\item Let $\theta(x,t)$ be as before. By the choice of our initial condition, and since condition $(B)$ is preserved, note that $\theta(x,t)\in (0,\pi)$ for all $(x,t)\in [1,b]\times [0,\infty)$. We set $c(x,t) = \cot\theta(x,t)$. Then by definition, $$(\psi^2(x,t)- x^2)' = 2c(x,t)(x\psi(x,t))'.$$We now claim that there exists a sequence $t_j\rightarrow \infty$ and $c_\infty\in\RR$ such that $$c(x,t_j)\rightarrow c_\infty$$ for each $x\in (1,b)$ with the convergence being uniform on any compact set $K\subset (1,b)$. First note that by the flow equation and Lemma \ref{lem:dhym-calabi-monotonicity} , $c'(x,t)\geq 0$ and so $c(x,t)$ is increasing in $x$. Moreover, $c(x,t)$ is a bounded function of $(x,t)$ since $0<\inf \theta(x,t) \leq \sup \theta(x,t) < \pi$. Let us now fix a $\delta>0$. Then since $\dot\psi\geq 0$ we have $$\int_{T}^{\infty}\int_{1+\delta}^{b-\delta}\Big|\frac{\partial \psi}{\partial t}\Big| = \int_{1+\delta}^{b-\delta}\Big(\psi_\infty(x) - \psi(x,T)\Big)\,dx \leq C$$ for some uniform constant $C$ independent of $T$. Next, from the flow equation for $\psi$, monotonicity of $c(x,t)$ and the fact that $Q(x)$ is uniformly lower bounded away from zero on $[1+\delta,b-\delta]$ we see that there exists a constant $C_\delta$ independent of $T$ such that $$\int_T^\infty \Big(c(b-\delta,t) - c(1+\delta,t)\Big)\,dt = \int_T^\infty\int_{1+\delta}^{b-\delta}c'(x,t)\,dx dt<C_\delta.$$So there exists a sequence $t_j\rightarrow\infty$ such that $$\lim_{j\rightarrow\infty}\Big(c(b-\delta,t_j) - c(1+\delta,t_j)\Big) = 0.$$ By passing to a subsequence we may assume that $c(1+\delta,t_j)\rightarrow c_\infty$. By monotonicity we then have that $c(x,t_j)\rightarrow c_\infty$ for all $x\in [1+\delta,b-\delta]$. A simple diagonalisation argument then proves the claim. 
\item Next, we claim that for all $x\in (1,b)$ we have $$\psi_\infty^2 - 2c_\infty x\psi_\infty - x^2 = A_\infty$$ for some constant $A_\infty$.
To see this, let $$h_j(x) =\psi^2(x,t_j) - 2c_\infty x \psi(x,t_j) - x^2,~ h(x) = \psi_\infty^2(x) - 2c_\infty x \psi_\infty(x) - x^2.$$ Our aim is to prove that $h(x)$ is a constant. Once again let $\delta>0$. 
\begin{align*}
h_j(x) - h_j(1+\delta) &= \int_{1+\delta}^xh_j'(y)\,dy\\
&= \int_{1+\delta}^x\Big(\psi^2(y,t_j) - y^2\Big)'\,dy - 2\int_{1+\delta}^xc_\infty(y\psi(y,t_j))'\,dy\\
&= \int_{1+\delta}^x\Big(\psi^2(y,t_j) - 2c(y,t_j)y\psi(y,t_j) - y^2\Big)'\,dy - 2\int\Big(c_\infty - c(y,t_j)\Big)(y\psi(y,t_j))'\,dy \\
&+ 2\int_{1+\delta}^xc'(y,t_j)y\psi(y,t_j)\,dy\\
&= 2\int_{1+\delta}^xc'(y,t_j)y\psi(y,t_j)\,dy - 2\int\Big(c_\infty - c(y,t_j)\Big)(y\psi(y,t_j))'\,dy\\
&\leq C\int_{1+\delta}^xc'(y,t_j)\,dy - 2\int_{1+\delta}^xc'(y,t_j)y\psi(y,t_j)\,dy - 2\Big|(c_\infty - c(y,t_j))y\psi(y,t_j)\Big|_{1+\delta}^x\\
&\leq C\Big(c(x,t_j) - c(1+\delta,t_j)\Big).
\end{align*}
Letting $j\rightarrow \infty$, we see that $h(x) = h(1+\delta)$ for all $\delta>0$ and all $x\in [1+\delta,b-\delta]$.
\item Finally we prove that $\psi_\infty \equiv \tilde\psi_\xi$. Firstly,  note that since $\psi_\infty(1+)= s$ and $\psi_\infty(b) = p$, we see that $\psi_\infty \equiv\tilde\psi_s.$ Next, since $x\psi_\infty$ is increasing, this implies that $s\geq \tilde c_s$ or equivalently $s\geq \xi$. Next, by Lemma \ref{lem:comparison-to-limit} we have that $\psi_\infty\leq \tilde\psi_\xi$. On the other hand, by Lemma \ref{lem:dhym-calabi-comparison}, if $s>\xi$, then $\psi_\infty(x) = \tilde\psi_s(x) > \tilde\psi_\xi$ for any $x\in [1,b)$, a contradiction. Hence $s = \xi$ and this completes the proof of the proposition. 
\end{itemize}
\end{proof}

The higher order estimates can now be proved in a similar manner to those for the $J$ flow using the estimates from \cite{CJY,Tak,FYZ}.

\subsection{Minimizing the volume functional} Solutions to the deformed Hermitian Yang Mills equations minimise a certain volume functional $\Nu_\omega:C^\infty(X,\RR)\rightarrow \RR_+$ defined by  $$\Nu_\omega(\varphi) = \int_X r_\omega(\chi_\varphi)\omega^n,$$ where $r_\omega$ is the unique positive function defined by $$(\chi+\sqrt{-1}\omega)^n = r_\omega(\chi)e^{\sqrt{-1}\theta_\omega(\chi)}.$$ If $n = 2$, then $$\Nu_\omega(\chi) = \int_X\Big[\Big(\frac{\chi^2 - \omega^2}{\omega^2}\Big)^2 + \Big(\frac{2\chi\wedge\omega}{\omega^2}\Big)^2\Big]^{1/2}\omega^2.$$ It was observed in \cite{JY} that $$\Nu_\omega(\chi) \geq 2(\csc\vartheta(X,\al,\be))\al\cdot\be,$$ with equality if and only if $\chi$ solved the dHYM equation.  On $\Bl_{x_0}\PP^2$, if we let $\chi = \ddbar v(\rho)$ and $\omega = \ddbar u$, and $\psi:[1,b]\rightarrow \RR$ as above with $\psi(u'(\rho)) = v'(\rho)$, then (possibly upto a constant factor) $$\Nu_\omega(\chi)  = \int_1^b\Big[(x\psi' + \psi)^2 + (\psi'\psi - x)^2\Big]^{1/2}\,dx.$$ 

Let $v(t)$ be the solution of the cotangent flow with the special initial data. Then as we have seen above, $\chi(t)$ converges on $X\setminus E$ to a form $\chi_\infty$ that extends to a conical form on $X$. We continue to denote the global conical form by $\chi_\infty$,  and let $$\alpha_\infty = [ \chi_\infty] = p[H] - \zeta [E].$$ Then $\chi_\infty \cdot \omega\geq 0$ and 
$$ \theta_{\omega}(\chi_\infty) =  \theta_{min}$$ on $X\setminus E$. Finally, let $\theta \in \zeta[H] - q[E]$ be any {\em smooth} form. %
Then $(X, \theta_j)$ converges in Cheeger-Gromov sense to $(Y, \theta_\infty)$ which is isomorphic to $(X, \theta)$ and $\omega_\infty$ also lives on $Y$.   We then have the following analogue of Theorem \ref{mainthm2}\eqref{mainthem2-energy-conv}.

\begin{prop}
Let $v(t)$ be the solution of the cotangent flow with the special initial data. Let $\chi_\infty$,  and $\theta$ as above.  The volume functional $\cV$ satisfies
$$ \lim_{t\rightarrow\infty} \cV_\omega (\chi(t)) = \cV_\omega (\chi_\infty)+ \cV_{\pi^*\omega_{\PP^1}} (\theta).$$
Here 
$$\cV_{\omega}(\chi_\infty) = \cot \theta_{min} ( \alpha_\infty \cdot \beta)$$
is the volume functional on $(X, \omega)$ for $\chi_\infty$ and
$$\cV_{\pi^*\omega_{\PP^1}}(\theta) = \int_{q}^{\zeta}  \sqrt{1+x^2} dx$$
is the volume functional from the base.

\end{prop}

The proof of the above  result is similar (and in fact simpler) than the proof of Theorem \ref{mainthm2} and \eqref{mainthem2-energy-conv}, and we skip it. Note on the other hand, that the analogue of Theorem 



\appendix

\section{Some intersection theoretic calculations} The aim of this appendix is to prove the following Proposition. Recall that $(M^n,\omega_M)$ is an $n$-dimensional K\"ahler manifold with $\omega_M \in -2\pi c_1(L)$ for some negative line bundle $\pi:L\rightarrow M$. We let $\pi: X=\PP(\sO\oplus L^{\oplus(m+1)})\rightarrow M $ be a projective bundle of dimension $m+n+1$. We consider K\"ahler classes $\al = \pi^*[\omega_M] + a[D_\infty]$ and $\be = \pi^*[\omega_M] + b[D_\infty]$ where $a,b>0$ and $D_\infty$ is the divisor at infinity given by  Recall that $$\mu_s = \frac{(1+a)^na^mb + nI_{m,n-1,s}(a)}{I_{m,n,s}(a)},$$ where $$I_{m,n,s}(x) = \int_s^xt^m(1+t)^n\,dt.$$

\begin{prop}\label{lem: intersection-number-blow-up}
Let $\tilde\pi:\tilde X = \Bl_{P_0}X\rightarrow X$ be the blow down map with exceptional divisor $E$, and $\tilde\al_s = \tilde\pi^*\al - sE$. Then $$\mu_s = (m+n+1) \frac{(\tilde\al_s)^{m+n}\cdot\tilde\pi^*\be}{(\tilde\al_s)^{m+n+1}}.$$ In particular, 
\begin{enumerate}
\item $\mu_0 = \mu(X,\al,\be)$, and
\item if $\mu\leq n$, then $\la$ is the unique real number $\la$ in $[0,a)$ (with $\la = 0$ when $\mu = n$) such that $$\zeta_{inv} = (m+n+1) \frac{(\tilde\al_\la)^{m+n}\cdot\tilde\pi^*\be}{(\tilde\al_\la)^{m+n+1}}.$$ 
\end{enumerate}
\end{prop}

We begin by proving the first part independently. 
\begin{lem}\label{lem:mu-0}
$\mu_0 = \mu(X,\al,\be) = (m+n+1) \frac{\al^{m+n}\cdot\be}{\al^{m+n+1}}$.
\end{lem}
\begin{proof}
The Lemma can be quite easily proved by simply expressing the integrals of $\chi^{m+n+1}$ and $\chi^{m+n}\wedge\omega$ in terms of $\psi$. Instead we will compute using intersection theory of projective bundles. The reason is that this argument will generalize easily to prove the full proposition.

Recall that $r = m+2$ is the rank of the bundle $E = \sO_M\oplus L^{\oplus(m+1)}$. We let $d = [\omega_M]^n\cdot[M]$ be the degree of $-L$ on $M$. We claim that 
\begin{align}
\al^{n+r-1} &= d(n+r-1){n+r-2\choose n}I_{m,n}(a) \text{ and }\label{eq:identity-original-1}\\
\al^{n+r-2}\wedge\be &= d{n+r-2\choose n} \Big((1+a)^na^{r-2}b + nI_{m,n-1}(a)\Big).\label{eq:identity-original-2}
\end{align}
The Lemma clearly follows from these two claims. For simplicity, and without any loss of generality we may assume that $d = 1$. To prove the first claim we compute 
\begin{align*}
\al^{n+r-1} &= \Big([\omega_M]+a\eta\Big)^{n+r-1}\\
&= \sum_{l=0}^n{n+r-1\choose n-l}a^{r-1+l}[\omega_M]^{n-l}\cdot\eta^{r-1+l},
\end{align*}
where we used the fact that $[\omega_M]^k\cdot\eta^{n+r-1-k} = 0$ if $k>n$. Next, using Lemma \ref{lem:intersection-key-lemma} below, 
\begin{align*}
\al^{n+r-1} &= \sum_{l=0}^n{n+r-1\choose n-l} {r+l-2\choose l}a^{r-1+l}\\
&=(n+r-1){n+r-2\choose n} \sum_{l=0}^n{n\choose l}\frac{a^{r-1+l}}{r-1+l}\\
&= (n+r-1){n+r-2\choose n} I_{m,n}(a),
\end{align*}
where the final equality follows easily by applying the binomial theorem to expand $(1+x)^nx^{r-1}$ and integrating each term. To prove the second equality we proceed in a similar vein and compute 
\begin{align*}
\al^{n+r-2}\cdot\be &= \Big([\omega_M]+a\eta\Big)^{n+r-2}\cdot\Big([\omega_M]+ b\eta\Big)\\
&= \Big(\sum_{l=0}^n{n+r-2\choose n-l}a^{r-2+l}[\omega_M]^{n-l}\cdot\eta^{r-2+l}\Big)\cdot\Big([\omega_M] + b\eta\Big)\\
&=\sum_{l=0}^n{n+r-2\choose n-l}a^{r-2+l}[\omega_M]^{n-l+1}\cdot\eta^{r-2+l} +\sum_{l=0}^n{n+r-2\choose n-l}a^{r-2+l}b[\omega_M]^{n-l}\cdot\eta^{r-1+l}\\
&=\sum_{l=0}^{n}{n+r-2\choose n-l}{r+l-3\choose l-1}a^{r-2+l} + \sum_{l=0}^n{n+r-2\choose n-l}{r+l-2\choose l}a^{r-2+l}b\\
&= n{n+r-2\choose n}\sum_{l=1}^n{n-1\choose l-1}\frac{a^{r-2+l}}{r-2+l}+ {n+r-2\choose n}b\sum_{l=0}^n{n\choose l}a^{r-2+l} \\
&={n+r-2\choose n}\Big(nI_{m,n-1}(a)+(1+a)^na^{r-2}b\Big).
\end{align*} Here we have again used Lemma \ref{lem:intersection-key-lemma} in the fourth line, and the binomial theorem to evaluate $I_{m,n-1}(a)$.
\end{proof}
\begin{lem}\label{lem:intersection-key-lemma}
With the notation as above, and $l=0,1,\cdots,n$, $$[\omega_M]^{n-l}\cdot\eta^{r-1+l} = {r+l-2\choose l}d.$$
\end{lem}
The proof of the lemma relies on the following elementary identity of binomial coefficients which seems well known to combinatorists (cf. \cite[Identity 3.49]{Gould}). We include a short proof for the convenience of the reader.  
\begin{lem}\label{lem:combinatorial-identity}
For positive integers $q$ and $s$, $$\sum_{j=1}^{s+1}(-1)^{j-1}{s+1\choose j}{s+q-j\choose q-j} = {q+s\choose q}.$$
\end{lem}
\begin{proof}
The right hand side is simply the number of unordered integer solutions to $$x_0+x_1+\cdots+x_s = q,~x_i\geq0 ~ \forall i.$$ We now count this in a different way. Let $$E_i = \{(x_0,x_1,\cdots,x_s)~|~ x_0+x_1+\cdots+x_s = q,~x_i\geq 1\}.$$ Since $q>0$, clearly the total count is the cardinality of $\cup_{i=1}^s E_i.$ For any $j=1,\cdots,s$, let $0\leq i_1<i_j\cdots<i_j\leq s$, the cardinality of $\cap_{k=1}^jE_{i_k}$ is the number of solutions to $$x_0'+\cdots x_s' = q-j,$$ where $$x_i' =\begin{cases} x_i-1,~i = i_k\text{ for some k}\\x_i,~\text{otherwise},\end{cases}$$ and hence equals ${s+q-j\choose q-j}$. But there are exactly ${s+1\choose j}$ ways of choosing the indices $i_1,\cdots,i_j$, and the proof is completed by an application  of the inclusion-exclusion principle.   
\end{proof}
\begin{proof}[Proof of Lemma \ref{lem:intersection-key-lemma}]
Recall \cite{EH} that the cohomology ring $H^*(M)$ sits as a sub-ring in $H^*(X)$ via the pull-back $\pi^*$, and the latter is in fact  generated by adjoining the element $\eta  = [D_\infty]$ to $H^*(M)$ with a single relation $$\eta^r + c_1(E)\eta^{r-1}+\cdots + c_r(E) = 0.$$ By the sum formula for the Chern classes, $c_r(E) = 0$ and for $j=1,\cdots,r-1$ $$c_j(E) = (-1)^j{r-1\choose j}[\omega_M]^j,$$ so that $$\eta^r = \sum_{j=1}^{r-1}(-1)^{j-1}{r-1\choose j}[\omega_M]^j\cdot\eta^{r-j}.$$  We prove the lemma by induction on $l$. For $l=0$, the lemma is obvious. Suppose the lemma holds for $l=0,1,\cdots,q-1$. Then 
\begin{align*}
[\omega_M]^{n-q}\cdot\eta^{r-1+q} &=\sum_{j=1}^{r-1}(-1)^{j-1}{r-1\choose j}[\omega_M]^{n-(q-j)}\cdot\eta^{r-1+(q-j)}\\
&= \sum_{j=1}^{r-1}(-1)^{j-1}{r-1\choose j}{r-2+q-j\choose q-j}\\
&= {r-2+q\choose q}
\end{align*}
by the Lemma above applied to $s = r-2$. 

\end{proof}

\begin{proof}[Proof of Proposition \ref{lem: intersection-number-blow-up}]
We recall that for the blow-up, we have the following commutative diagram: 
\[
 \begin{tikzcd}
    E \arrow{r}{\tilde\pi}\arrow{d}{j} & P_0 \arrow{d}{i} \\
   \tilde X \arrow{r}{\tilde \pi} & X
  \end{tikzcd}\]
  
 The exceptional divisor $E = \PP(\mathcal{N}_{P_0\slash X})$ is simply the projectivization of the normal bundle of $P_0$ in $X$. We denote  the class $c_1(\sO_{ \PP(\mathcal{N}_{P_0\slash X})}(1))$ on $E$ by $\theta$. By general theory \footnote{Sections $\Ga$ of $\PP(E)$ are in one to one correspondence with surjective morphisms $\phi:E^*\rightarrow F$ where $F$ is a line bundle. To see this, for each $x\in M$, one gets a co-dimension one  hypersurface $\ker (\phi_x:E^*_x\rightarrow F_x)$ of $E_x^*$ which in turn gives a line in $E_x$. The normal bundle of $\Ga$ is then  $\ker \phi\otimes F^*$. In our setting, we have that $E = \sO\oplus L^{\oplus (m+1)}$ and $F = \sO$.} the normal bundle $\sN_{P_0/X}$ is simply $i(^*L)^{\oplus(m+1)} = (i^*L)^{\oplus(r-1)}$. We let $d = [\omega_M]^n\cdot[M]$ be the degree of $-L$ on $M$.

We claim that the following two identities hold: 
\begin{align}
\tilde\al_s^{n+r-1} &= d(n+r-1){n+r-2\choose n}I_{m,n,s}(a) \text{ and }\label{eq:identity-blow-up-1}\\
\tilde\al_s^{n+r-2}\wedge\tilde\be &= d{n+r-2\choose n} \Big((1+a)^na^{r-2}b + nI_{m,n-1,s}(a)\Big),\label{eq:identity-blow-up-2}
\end{align}
from which the proposition easily follows. By re-scaling, we may assume that $d = 1$. For simplicity we denote by $N = r-1+n$ the dimension of $\tilde X$. To prove the first identity we compute 
\begin{align*}
\tilde\al_s^{N} &= \Big(\pi^*\al - s[E]\Big)^{N}\\
&=\sum_{k=0}^{N}{N\choose k}(-1)^ks^k[E]^k(\pi^*\al)^{N-k}\\
&=\al^N+ \sum_{k=N-n}^N{N\choose k}(-1)^ks^k\int_E[E]^{k-1}(\pi^*\al)^{N-k}\\
&= \al^N - \sum_{k=N-n}^N{N\choose k}s^k\theta^{k-1}(\tilde\pi^*\al\Big|_{P_0})^{N-k},
\end{align*}
where we used the fact that $\tilde \pi^*\al$ is trivial in directions normal to $P_0$ in $E$, and so $[E]^k(\pi^*\al)^{N-k} = 0$ whenever $N-k>n$. Now we change indices $l=k-(N-n)$ to rewrite the above equation as 
\begin{equation}\label{eq:identity-blow-up-1-intermediate}
\tilde\al_s^{N} =  \al^N - \sum_{l=0}^n{N\choose n-l}s^{r-1+l}\theta^{r-2+l}\cdot(\tilde\pi^*\al\Big|_{P_0})^{n-l},
\end{equation} where the intersection numbers are now computed on $E$. We now claim that 
\begin{equation}\label{lem:intersection-key-lemma-blow-up}
\theta^{r-2+l}\cdot(\tilde\pi^*\al\Big|_{P_0})^{n-l} = {r-2+l\choose l}d,
\end{equation}
where as before $d = \langle[\omega_M]^n,[M]\rangle$. To see this, first note that if $s:M\rightarrow X$ is the section with image $P_0$, then since $P_0$ has no intersection with $D_\infty$, and since $\pi\circ s = id$, we have that $$\int_{P_0}(\tilde\pi^*\al)^n = \int_M\omega_M^n = d.$$ In particular this proves the above identity for $l=0$. We then proceed by induction and assume that the identity holds for $l = 0,1,\cdots,q-1$. Noting that $\theta$ satisfies the relation $$\theta^{r-1} + c_1(\sN_{P_0/X})\theta^{r-2}+\cdots+ c_{r-1}(\sN_{P_0/X}) = 0,$$ and that $$c_j(\sN_{P_0/X}) = (-1)^j{r-1\choose j}[i^*\omega_M]^j,$$ we have that 
\begin{align*}
\theta^{r-2+q}\cdot(\tilde\pi^*\al\Big|_{P_0})^{n-q}  &= \sum_{j=1}^{r-1}(-1)^{j-1}{r-1\choose j}\theta^{r-2+(q-j)}[i^*\omega_M]^{n-(q-j)}\\
&=  d\sum_{j=1}^{r-1}(-1)^{j-1}{r-1\choose j}{r-2+q-j\choose q-j}\\
&= {r-2+l\choose q}d,
\end{align*}
where we used the inductive hypothesis in the second line. Continuing with our proof of \eqref{eq:identity-blow-up-1}, and combining \eqref{eq:identity-blow-up-1-intermediate}, \eqref{lem:intersection-key-lemma-blow-up} and \eqref{eq:identity-original-1} with Lemma \ref{lem:combinatorial-identity} we see that 
\begin{align*}
\tilde\al^N &= \al^N - d\sum_{l=0}^n{N\choose n-l}s^{r-1+l}{r-2+l\choose l}\\
&= d(n+r-1){n+r-2\choose n}\Big(I_{m,n}(a)  - \sum_{l=0}^n{n\choose l}\frac{s^{r-1+l}}{r-1+l}\Big)\\
&= d(n+r-1){n+r-2\choose n}\Big(I_{m,n}(a) -I_{m,n}(s)\Big)\\
&= d(n+r-1){n+r-2\choose n}I_{m,n,s}(a). 
\end{align*}
The identity \eqref{eq:identity-blow-up-2} is proved similarly by making use of the identity  \eqref{eq:identity-original-2}.

\end{proof}

\section*{Acknowledgements} The authors would like to thank Hao Fang, Bin Guo, Vamsi Pingali, Lars Martin Sektnan, P. Sivaram, Jacob Sturm and Xiaowei Wang for useful discussions. They would also like to thank David Witt Nystr\"om for some clarifications on Zariski decomposition and regularity of the volume function on the big cone. Part of the work was done during the first and third named authors' respective visits to Mathematisches Forschungsinstitut Oberwolfach (MFO), and they would like to thank the organizers and staff for the invitation and hospitality.

\end{document}